\documentclass[a4paper,twoside,12pt]{amsart} 
\usepackage{amsmath,amssymb,amsfonts}
\usepackage{amsthm}
\usepackage{ulem}
\usepackage{graphics}   
\usepackage{cancel}
\usepackage{color}                    
\usepackage{hyperref}                 
\usepackage{graphicx,eucal}
\usepackage{latexsym,epsfig,epic,eepic}

\usepackage{xcolor}
\usepackage{graphicx}
\usepackage{lipsum}

\usepackage{tikz}
\theoremstyle{definition}
\newtheorem{theorem}{Theorem}[section]
\newtheorem{proposition}[theorem]{Proposition}
\newtheorem{corollary}[theorem]{Corollary}
\newtheorem{lemma}[theorem]{Lemma}

\newtheorem{example}[theorem]{Example}

\newtheorem{definition}[theorem]{Definition}
\newtheorem{remark}[theorem]{Remark}

\newcommand{\stkout}[1]{\ifmmode\text{\sout{\ensuremath{#1}}}\else\sout{#1}\fi}

\newcommand{\Z}{\mathbb{Z}}

\newcommand{\R}{\mathbb{R}}
\newcommand{\C}{\mathbb{C}}

\newcommand{\Per} {\operatorname{Per} }
\newcommand{\percompact}{{\mathcal Per}}

\newcommand{\per}{\operatorname{Per}}

\newcommand{\mnc}{\overline{\mathcal{S}}}
\newcommand{\mrs}{\mathcal{S}}

\newcommand{\PP}{R}

\author{Liza Arzhakova}
\address{}
\email{elizaveta.arzhakova1@gmail.com} 

\author{Gabriel Calsamiglia}
\address{Departamento de Matem\' atica Aplicada, Instituto de Matemática e Estatística, Universidade Federal Fluminense,Rua Professor Marcos Waldemar de Freitas Reis, s/n, Campus Gragoat\' a, Bloco G, Niter\' oi, 24210-201 Brazil} 
\email{gcalsamiglia@id.uff.br}

\author{Bertrand Deroin}
\address{CNRS \& CY Cergy Paris Universit\'e, Laboratoire AGM, UMR 8088, 2 av. Adolphe Chauvin, 95302 Cergy-Pontoise Cedex, France}
\email{bertrand.deroin@cyu.fr}


\date{2025}

\title{Isoperiodic meromorphic forms with at least three simple poles}
\begin{document}
\begin{abstract}
    In this paper we prove the connectedness of  isoperiodic moduli spaces of meromorphic differentials with at least three simple poles on homologically marked smooth curves whose periods are either not contained in a real line, or not contained in the rational space generated by the peripheral periods. From this topological property we deduce dynamical properties of the underlying foliation in the moduli space meromorphic differentials, by describing leaf closures associated to those spaces.  
\end{abstract}
\maketitle
\section{Introduction}

\subsection{Overview}

This work is the third of a series aiming at studying the dynamics and topology of isoperiodic foliations on moduli spaces of meromorphic one forms with simple poles on algebraic curves.
The leaves of those foliations are the maximal families of meromorphic forms so that the integrals on every closed cycle in the domain where the form is holomorphic remain constant.

In our first two papers on this topic, we developped the theory of isoperiodic foliations on moduli spaces of holomorphic differentials \cite{CDF2}, and on moduli of meromorphic differentials with two (simple) poles \cite{CD}.   

Here we focus on the general case of meromorphic forms with \textit{any} number of simple poles. The main novelty is that we have a family  of foliations for each choice of genus and number of poles, depending on the data of the residues at the poles (which are \(\frac{1}{2\pi i} \) multiples of peripheral periods around poles). The structure of subspaces of forms with fixed residues has been considered in \cite{GT, LT, CGPT}. Notice that multiplication in the moduli space by a complex number defines a conjugation of these different foliations, so we can consider the parameter as the projective class of tuples of residues that forms the right parameter space, a space of dimension \(n-2\) if \(n\) is the number of poles (as the sum of residues is always zero). These parameters in the moduli were not present in the previous cases of holomorphic differentials, or meromophic differentials with two poles and play an important role in the present paper.  

One of our main results is that, for moduli spaces of meromorphic differentials with at least three poles, and for a generic choice of residues, the foliation is minimal and ergodic. We also study the dynamics of the isoperiodic foliations for non generic parameters, and eventually compute explicitely the closures of the leaves.  We leave only one case: we do not compute closures of isoperiodic leaves corresponding to forms with real periods, all contained in the \(\mathbb Q\)-vector space generated by the peripheral periods, i.e.  $2\pi i $-times the residues. This case poses analogous difficulties as the finite degree cases of \cite{CDF2}, and seems to be complicated to achieve in a reasonable number of pages, so we decided to postpone for a separate paper.  

The general strategy is similar to the one of the first two papers: it is based on a transfer principle, that allows to relate the dynamics of isoperiodic foliations to the the action of modular group on the set of the period homomorphisms naturally defined by integration of the forms on cycles. This latter can be understood using Ratner's theory. The main feature that leads to the applicability of this transfer principle is a topological property of the period map, that associates to each form its period homomorphism, globally defined on a convenient pointed Torelli covering of moduli spaces of meromorphic differentials: namely that its fibers are connected. The proof of the latter takes up most of the paper, and uses convenient bordifications of the pointed Torelli space, and an inductive argument relating the boundary components with combinatorial properties of curves of lower complexity (either genus or number of poles). 

An interesting application of the connectedness of the fibers of the period map, that was already observed in the second paper on the two poles case, is that there are transcendental isoperiodic leaves that are closed in moduli space: the leaves corresponding to forms whose periods form a lattice in \(\mathbb C\). This phenomenon does not occur in the case of holomorphic forms. 

\subsection{Statement of results}
Let $\mathcal{M}_{g,n}$ denote the moduli space of smooth genus $g\geq 0$ complex curves with $n$ marked ordered points. Consider the bundle $\Omega\mathcal{M}_{g,n}\rightarrow \mathcal{M}_{g,n}$ whose fiber over a point $(C,(q_1,\ldots,q_n))$ is the space of meromorphic forms $\omega$ on $C$ having  (at worst) simple poles on $Q=(q_1,\ldots,q_n)$. It is a holomorphic vector bundle on a complex orbifold. To each element $(C,Q,\omega)\in\Omega\mathcal{M}_{g,n}$ we can define the period homomorphism $H_1(C\setminus Q)\rightarrow\mathbb{C}$ defined by $\gamma\mapsto\int_\gamma\omega$. The homology class of a positively oriented peripheral curve around $q\in Q$ will be denoted by $\pi_{q}$ and its image called peripheral period (it is $2\pi i$-times the residue).  The subgroup generated by those peripheral classes is the peripheral submodule $\Pi_Q$.

To be able to compare period homomorphisms on different curves we need to identify their homology groups. The orbifold universal cover  of $ \mathcal{M}_{g,n}$ is the Teichm\" uller space. A model for it can be defined with the data of a pair $(\Sigma, \PP)$ where \(\Sigma \) is a reference closed oriented surface of genus $g$ and \(\PP \subset \Sigma\) is a finite set of marked points having cardinal \(n\). Such a pair $(\Sigma,\PP)$ is said to be of type $(g,n)$. 
 The Teichm\" uller space $\mathcal{T}_{\Sigma,\PP}$ is as a set formed by the equivalence classes of triples $(C,Q,[f])$ where $C$ is a compact Riemann surface of genus $g$, \(Q\subset C\) a subset  of $n$ marked points and  $f:(\Sigma,\PP)\rightarrow (C,Q)$ is an orientation preserving homeomorphism with $f(\PP)=Q$ and the class is taken up to isotopy fixing each point of \(\PP\). In this representation, the covering group of the map that forgets $f$,  $$\mathcal{T}_{\Sigma,\PP}\rightarrow \mathcal{M}_{g,n}$$ is the  mapping class group $\text{Mod}_{\Sigma,\PP}$, namely the set of isotopy classes of orientation preserving diffeomorphisms of $\Sigma$ that fix each marked point in $\PP$.

The pointed Torelli space is an intermediate Galois covering, defined as the quotient of $\mathcal{T}_{\Sigma, \PP}$ by the subgroup $\mathcal{I}_{\Sigma,\PP}\subset\text{Mod}_{\Sigma,\PP}$ formed by the elements that act trivially on \(H_1(\Sigma\setminus\PP,\mathbb{Z})\), referred to as the pointed Torelli group.\footnote{A word of warning: this group is not to be confused with  other known subgroups of $\text{Mod}_{\Sigma,\PP}$ characterized by their (trivial) action on some other homology groups, for instance  the subgroup that fixes every class in the relative homology group $H_1(\Sigma,\PP;\mathbb{Z})$ used in \cite{CDF2}} 
 We denote it by 
 \[\mathcal{S}_{\Sigma,\PP}=\frac{\mathcal{T}_{\Sigma,\PP}}{\mathcal{I}_{\Sigma,\PP}}.\]  A point in pointed Torelli space corresponds to a triple \((C,Q,m)\) where \(C\)  is a smooth complex curve of genus \(g\) with a subset $Q\subset C$ of $n$  marked points and  \( m : H_1 (\Sigma\setminus \PP, \Z) \rightarrow H_1 (C\setminus Q, \Z)\) is an isomorphism that preserves the intersection form on both sides, and map the peripheral element around a point \(r\in \PP\) to the one around the point of \(Q\) corresponding to \(r\), up to the equivalence \((C,Q, m) \sim (C',Q', m') \) if there exists a biholomorphism \( h : (C,Q)\rightarrow (C',Q') \) such that \( m' = h_* \circ m\). The mapping class group of $(\Sigma,\PP)$ acts on \(\mathcal{S}_{\Sigma,\PP}\) by precomposition on the marking and the quotient map gives a covering map 
\begin{equation} 
\label{eq:torellicover}
\mathcal{S}_{\Sigma,\PP}
\rightarrow\mathcal{M}_{g,n}
\end{equation} 
with covering group formed by the elements of $\text{Aut}(H_1(\Sigma\setminus\PP))$ that preserve the intersection form and fix each element in its kernel. The latter coincides with the so called peripheral submodule -- denoted $\Pi_\PP\subset H_1(\Sigma\setminus\PP)$ --  since it is generated by peripheral classes at the points of $\PP$.

The Hodge bundle $\Omega\mathcal{M}_{g,n}\rightarrow \mathcal{M}_{g,n}$ can be pulled back to $\mathcal{S}_{\Sigma,\PP}$ by using the map \eqref{eq:torellicover} producing a bundle $\Omega\mathcal{S}_{\Sigma,\PP}\rightarrow \mathcal{S}_{\Sigma,\PP}$ whose fiber over a point $(C,Q,m)$ is identified with the set of meromorphic forms $\omega$ on $C$ having at worst simple poles lying in $Q$ (but not all elements of \(Q\) are necessarily poles). 


\begin{definition}
The period of $(C,Q,m,\omega)\in \Omega\mathcal{S}_{\Sigma,\PP}$ is the homomorphism \[\per (C,Q,m,\omega): H_1(\Sigma\setminus \PP,\mathbb{Z})\rightarrow \mathbb{C}\text{   defined by}\] \[\gamma\mapsto\int_{m(\gamma)}\omega.\] 
The period map on \(\Omega\mathcal{S}_{\Sigma,\PP}\) is the map   \[\per_{\Sigma,\PP}:\Omega\mathcal{S}_{\Sigma,\PP}\rightarrow \text{Hom}(H_1(\Sigma\setminus \PP,\mathbb{Z}),\mathbb{C}) \]defined by \((C,Q,m,\omega)\mapsto \per(C,Q,m,\omega)\).
\end{definition}

In section \ref{s:preliminary resutls} (or in \cite{CD} ) it is proven that $\per_{\Sigma,\PP}$ is a holomorphic submersion around any point in the subset  $\Omega^*\mrs_{\Sigma,\PP}$ of non-zero forms. The underlying foliation is therefore regular and holomorphic. To analyze the foliation it induces on the moduli space we are interested in understanding the topology of its fibers.
\begin{remark}\label{rem:GL2R invariance} For any $a\in\mathbb{C}^*$,  multiplication of meromorphic forms by $a$  defines a homeomorphism of $\Omega\mathcal{S}_{\Sigma,\PP}$ that sends $\per_{\Sigma,\PP}^{-1}(p)$ to  $\per_{\Sigma,\PP}^{-1}(a\cdot p)$ for any  $p\in\text{Hom}(H_1(\Sigma\setminus\PP);\mathbb{C})$. Whenever needed we will use this $\mathbb{C}^*$-action  to simplify the arguments. For instance we can use this freedom to restrict to the case that  all periods are real whenever they all lie in a real line in $\mathbb{C}$, or that some peripheral class has  period  $1\in\mathbb{C}$ (and therefore the image of the peripheral submodule contains $\mathbb Z$)
\end{remark}


\begin{theorem}[Periods in a real line in $\mathbb{C}$]\label{t:Rgn}
Let $(\Sigma,\PP)$ be a closed oriented surface $\Sigma$ of genus $g\geq 0$  with a set $\PP\subset \Sigma$ of $n\geq 3$ marked points and $p\in \text{Hom}(H_1(\Sigma\setminus\PP);\mathbb{R})$ be a homomorphism whose image is not contained in $\mathbb{Q}\otimes p(\Pi_\PP)$. Then,  the fiber \(\per_{\Sigma,\PP}^{-1}(p)\) is connected. 
\end{theorem}

\begin{theorem}[Non $\mathbb{R}$-collinear periods]\label{t:Cgn}
Let $(\Sigma,\PP)$ be a closed oriented  surface of genus $g\geq 0$ with a set $\PP\subset \Sigma$ of $n\geq 3$ marked points and $p\in \text{Hom}(H_1(\Sigma\setminus\PP);\mathbb{C})$ be a homomorphism whose image is not contained in a real line in $\mathbb{C}$. Then,  the fiber \(\per_{\Sigma,\PP}^{-1}(p)\) is connected.  
\end{theorem}
Theorems \ref{t:Rgn} and \ref{t:Cgn} immediately implies 
\begin{corollary}\label{c:surjection of pi1 of leaf}
If $L\subset \Omega^*\mathcal{M}_{g,n}$ is the leaf of $\mathcal{F}_{g,n}$ associated to a period $p\in H^{1}(\Sigma\setminus\PP,\mathbb{C})$ with either non $\R$-collinear image, or with some period outside $\mathbb{Q}\otimes p(\Pi_\PP)$,  then $$\pi_{1}(L)\rightarrow\text{Stab}(p)\subset \text{Aut}(H_{1}(\Sigma\setminus\PP))\text{  is surjective. }$$ 
\end{corollary}

The mapping class group $\text{Mod}_{\Sigma,\PP}$ acts on source and target of the map $\per_{\Sigma,\PP}$ equivariantly, that is,  for every $(C,m, \omega)\in\Omega\mathcal{S}_{\Sigma,\PP}$ and $\varphi\in\text{Mod}_{\Sigma,\PP}$, 
$$\per_{\Sigma,\PP}(\varphi\cdot (C,m,\omega))=\per_{\Sigma,\PP}((C,m,\omega))\circ \varphi^{-1}_*.$$ The fibration induced by $\per_{\Sigma,\PP}$ thus induces a regular holomorphic foliation $\mathcal{F}_{g,n}$ on the quotient (moduli) space $\Omega^*\mathcal{M}_{g,n}$ called the isoperiodic foliation. 
Since the elements of $\text{Mod}_{\Sigma,\PP}$ act as the identity on the peripheral module $\Pi_R$ the map 
$$\Omega^*\mrs_{\Sigma,\PP}\rightarrow \text{Hom}(\Pi_R,\mathbb{C})$$ defined by $(C,Q,m,\omega)\mapsto \big(\Per_{\Sigma,\PP} (C,Q,m,\omega)\big)_{|\Pi_\PP}$ is invariant by the $\text{Mod}_{\Sigma,\PP}$  action and defines a holomorphic submersion on $\Omega^*\mathcal{M}_{g,n}$. Each level set is invariant by the isoperidic foliation and can be characterized by a vector $\alpha=(\alpha_1,\ldots, \alpha_n)\in\mathbb{C}^n$: $$\Omega^\alpha\mathcal{M}_{g,n}=\{(C,(q_1,\ldots,q_n),\omega): \int_{\pi_{q_i}}\omega =\alpha_i\}$$
The restriction of $\mathcal{F}_{g,n}$ to $\Omega^\alpha\mathcal{M}_{g,n}$  will be denoted by $\mathcal{F}^{\alpha}_{g,n}$. The $\mathbb{C}^*$ action  on $\Omega^*\mathcal{M}_{g,n}$ is by biholomorphisms that exchange the levels and preserves the foliation in each level. So, to analyze the leaf closures, when $n>0$, we can restrict the values in $\alpha$ to one of the three following cases depending on the closure $\overline{p(\Pi_R)}\subset\mathbb{C}$
\begin{itemize}
\item If it is cyclic we can suppose it is $\mathbb Z$
\item If it is a real line in $\mathbb C$ we can suppose it is $\mathbb{R}$
\item In any other case, we can suppose it contains $\mathbb{R}$. 
\end{itemize}

To analyze the dynamics of $\mathcal{F}_{g,n}$ we will apply the Transfer Principle of \cite{CDF2} to the present case: a leaf closure of $\mathcal{F}_{g,n}$ corresponds with an orbit closure of the $\text{Mod}_{\Sigma,\PP}$ action on $H^1(\Sigma\setminus \PP,\mathbb{C})$. With the aim of describing the orbit closures, we introduce some definitions: 
\begin{definition} \label{d:discrete factor}Given a period \(p\in H^1 (\Sigma\setminus\PP, \mathbb C )\), a discrete factor of \(p\) is a continuous surjective morphism \( \varphi : \mathbb C\rightarrow A\) to a Lie group such that \( \varphi \circ p (H_1 (\Sigma\setminus\PP, \mathbb Z))\) is discrete.  Given two periods \(p, q\in H^1 (\Sigma\setminus\PP, \mathbb C)\) having both \(\varphi \) as discrete factor, we say that they have the same image in \(\varphi\) if  \( \varphi \circ p (H_1 (\Sigma\setminus\PP, \mathbb Z))=\varphi \circ q (H_1 (\Sigma\setminus\PP, \mathbb Z))\). \end{definition}
Up to isomorphism of the target group $A$, the candidates to non-trivial discrete factors are
    $$a\Re+b\Im: \mathbb C\rightarrow \mathbb R \text{  for some }(a,b)\in\mathbb{R}^2\setminus (0,0).$$
\begin{definition}\label{d:V}
For any \( p \in H^1 (\Sigma\setminus \PP, \mathbb C)\), with cyclic peripheral periods $p(\Pi_\PP)= \mathbb{Z}$ and discrete factor  $\Im$  with image $\delta\mathbb{Z}\subset\mathbb{R}$ we define the class \( V(p) \in \mathbb R /\delta\mathbb{Z}\) as 
\[ V (p) := \Re P \cdot \Im P \text{ mod }  \delta\]
where $P:H_1(\Sigma,\mathbb{Z})\rightarrow\mathbb{C}$ is any lift  of the map $\overline{p}:H_1(\Sigma,\mathbb{Z})\rightarrow\mathbb{C}/\mathbb{Z}$ induced by $p$. 

The same definition can be applied to any $p$ with cyclic peripheral periods contained in the kernel of a non-trivial discrete factor. 
\end{definition}
The value of $V$ can be calculated directly by using the values of $p$ on a family of cycles of $\Sigma\setminus N$ that constitute a symplectic basis of $H_1(\Sigma,\mathbb Z)$. 

The following is a list of invariants under the action of the mapping class group $\text{Mod}_{\Sigma,\PP}$:
\begin{itemize}
    \item each peripheral period
    \item the image subgroup
    \item non-trivial discrete factors and their image groups
    \item the $V$-value when the  peripheral periods constitute a discrete submodule of the kernel of a discrete factor. 
\end{itemize}
\begin{definition}
 Given $(C,Q,\omega)\in\Omega\mathcal{M}_{g,n}$ we define its group of periods  
   $$\Lambda_{\omega}=\left\{\int_\gamma \omega\in \mathbb{C}: \gamma\in H_1(C\setminus Q)\right\}$$ and the peripheral period subgroup
    $$\Pi_{\omega}=\left\{\int_{\pi}\omega\in\mathbb{C}:\pi\in\Pi_Q\right\}.$$ 
\end{definition}

\begin{definition}
    Given $(C,Q,\omega)\in\Omega^{\alpha}\mathcal{M}_{g,n}$, a discrete factor of $(C,Q,\omega)$ is a discrete factor of $\per(C,Q,m,\omega)\in H^1(\Sigma\setminus\PP,\mathbb{C})$ where  $(C,Q, m,\omega)\in\Omega^{\alpha}\mathcal{S}_{\Sigma,\PP}$. If $\Pi_\omega=\mathbb{Z}$ and $\Im$ is a discrete factor of $(C,Q,\omega)$, define $V(\omega)=V(\per(C,Q,m,\omega))$.  
\end{definition}

\begin{theorem}
\label{t:leaf closures}
Let $g\geq 1$, $n\geq 2$ and  $(C_0,Q_0,\omega_0)\in \Omega^\alpha\mathcal{M}_{g,n}$ be a form with $n$ simple poles,  normalized via the action of $\text{GL}_2(\mathbb{R})$ on $\mathbb{C}$ to have $\mathbb{Z}\subseteq \Pi_{\omega_0}$ and if it is cyclic then $\Pi_{\omega_0}=\mathbb{Z}$. Suppose that if $\Lambda_{\omega_0}\subset \mathbb{R}$, then $\Lambda_{\omega_0}\not \subset \mathbb{Q}\otimes\Pi_{\omega_0}$. 

A point $(C,Q,\omega)\in\Omega^{\alpha}\mathcal{M}_{g,n}$  lies in $\overline{L}$, the closure of the isoperiodic leaf $L$ through $(C_0,Q_0,\omega_0)$ if and only if 
\begin{itemize}
\item any discrete factor of $(C_0,Q_0,\omega_0)$ is also a discrete factor of \((C,Q,\omega)\) and both have the same image in it. 
\item if $\Pi_{\omega_0}=\mathbb{Z}$ and $\Im:\mathbb{C}\rightarrow \mathbb{R}$ is a a discrete factor of $(C_0,Q_0,\omega_0)$, then \begin{equation*}  V((C,Q,\omega))= V(C_0,Q_0,\omega_0) .\end{equation*}
\end{itemize}
The saturated set $\overline{L}$ is a real analytic sub-orbifold and the restriction of the isoperiodic foliation $\mathcal{F}_{g,n}^\alpha$ to $\overline{L}$ is ergodic. If \(\Lambda\subset \mathbb C\) is a lattice, then the leaf \(L\) is closed and non algebraic.  
\end{theorem} 

The assumptions of Theorem \ref{t:leaf closures} do not allow the possibility for  \(\Lambda\) to be cyclic, in which case the leaf \(L\) is algebraic, as it is a Hurwitz space of finite covers of a  cylinder \( (\mathbb C^* , a\frac{dz}{z}) \) with \(a\in \mathbb C^*\).

We chose to state this result in such a synthetic form using discrete factors, but it can be made more precise. We describe the various possibilities in terms of the topological closures of periods and peripheral periods more in detail in  Corollary \ref{c: leaf closures}.

\subsection{Thanks}

The authors acknowledge the support from CY Initiative (ANR-16-IDEX-0008).
The second author acknowledges support from CNPq Projeto Universal 408687/2023-1 "Geometria das Equações Diferenciais Algébricas", from Programa PRONEX FAPERJ e CNPq (E-26/010.001270/2016)  “Métodos Geométricos em Equações Diferenciais Complexas”, from FAPERJ (E-26/210.418/2025) and from Coordenação de Aperfeiçoamento de Pessoal de Nível Superior - Brasil (CAPES) – Finance Code  001.

\section{Notations}

\begin{itemize}
\item \((\Sigma,\PP)\) is a closed oriented surface $\Sigma$  with a finite set $\PP\subset\Sigma$ of  marked points. Its type is $(g,n)$ where $g$ is the genus of $\Sigma$ and $n=\#\PP$.
\item $\text{Mod}_{\Sigma,\PP}$ group of isotopy classes of orientation preserving homeomorphism of $\Sigma$ fixing $\PP$ pointwise. 
\item $\mathcal{I}_{\Sigma,\PP}\subset \text{Mod}_{\Sigma,\PP}$ is the  pointed Torelli subgroup formed by elements acting trivially on $H_1(\Sigma\setminus \PP, \mathbb{Z})$. 
\item $\mathcal{T}_{\Sigma,\PP}$ Teichm\"uller space of the pair $(\Sigma,\PP)$ parametrizing compact Riemann surfaces with marked points and topologically marked by $(\Sigma,\PP)$. 
\item $\mathcal{M}_{g,n}=\frac{\mathcal{T}_{\Sigma,\PP}}{\text{Mod}_{\Sigma,\PP}}$ is the moduli space of curves of genus $g$ with $n$ marked points where (g,n) is the type of $(\Sigma,\PP)$.
\item $\mathcal{S}_{\Sigma,\PP}=\frac{\mathcal{T}_{\Sigma,\PP}}{\mathcal{I}_{\Sigma,\PP}}$ is the (pointed) Torelli space of $(\Sigma,\PP)$
\item  $\overline{\mathcal{T}}_{\Sigma,\PP} $ augmented Teichm\"uller space  of $(\Sigma,\PP)$ parametrizing stable curves with marked points topologically marked by a homotopy class of collapse maps defined on $(\Sigma,\PP)$.  
\item \(\overline{\mathcal{M}}_{g,n}=\frac{\overline{\mathcal{T}}_{\Sigma,\PP}}{\text{Mod}_{\Sigma,\PP}}\) compactification of Moduli space of type $(g,n)$. As a complex orbifold it is equivalent to the Deligne-Mumford compactification of $\mathcal{M}_{g,n}$. 
\item $\overline{\mathcal{S}}_{\Sigma,\PP} =\frac{\overline{\mathcal{T}}_{\Sigma,\PP}}{\mathcal{I}_{\Sigma,\PP}}$ bordification of Torelli space of $(\Sigma,\PP)$. 
\item $\overline{\mathcal T}_{\Sigma,\PP}^{sep}$ is the bordification with curves of compact type (stable forms only with separating nodes) 
\item \(\overline{\mathcal{M}}^{sep}_{g,n}=\frac{\overline{\mathcal{T}}^{sep}_{\Sigma,\PP}}{\text{Mod}_{\Sigma,\PP}}\) compact type bordification of Moduli space of type $(g,n)$.
\item \(\overline{\mathcal{S}}^{sep}_{\Sigma,\PP}=\frac{\overline{\mathcal{T}}^{sep}_{\Sigma,\PP}}{\mathcal{I}_{\Sigma,\PP}}\) compact type bordification of Torelli space of $(\Sigma,\PP)$.

\item Given $\mathcal{K}\subset \overline{\mathcal T}_{g,n}$ or any of its quotients, $\Omega\mathcal{K}\rightarrow \mathcal{K}$ denotes the bundle of stable forms on elements of $\mathcal{K}$ having (at worst) simple poles at marked points
\item $\Omega^*\mathcal{K}\subset \Omega\mathcal{K}$ is the subset of forms without any zero component.
\item $\Omega^*_0\mathcal{K}\subset \Omega^*\mathcal{K}$ is the subset of forms with zero residues at non-separating nodes. 
\item \(G\): \(P\)-labelled graph, \(V(G)\) set of vertices, \(L(G)\) set of half-legs, \(E(G)\) set of edges
\item \(\pi_r \in H_1 (\Sigma\setminus\PP, \mathbb Z)\) peripheral class around the the puncture $r\in\PP$.
\item $\Pi_E\subset H_1(\Sigma_g\setminus \PP)$  the submodule generated by the peripherals $\{\pi_e : e\in E\}$ around points of $E\subset \PP$
\item \(\Pi :=\Pi_\PP\) the peripheral submodule. 
\item \(\Lambda_\omega\subset \mathbb{C}\) the group of periods of a (stable) form i.e. integrals of \(\omega\) along smooth closed paths avoiding points with non-zero residue.
\item \(\Pi _{\omega}\subset\Lambda_{\omega} \) the subgroup of peripheral periods of  \(\omega\)
\item Given $(k_1,\ldots,k_l)\in\mathbb{N}_{>0}^l$ with $k_1+\cdots +k_l=2g-2+n$ denote  $\Omega\mathcal{M}_{g,n}(k_1,\ldots,k_l)$ the stratum of forms with $n$ simple poles at marked points and $l$ zeros of orders given by $(k_1,\ldots,k_l)$. 
\item Given a meromorphic stable form with simple poles $\omega$ having no zero component  and $\theta\in\mathbb{S}^1$, $\mathcal{G}_{\theta}$ denotes the directional (singular) geodesic foliation whose tangents are always $\theta$. The singular set of $\mathcal{G}_\theta$ is the support of the divisor of the form, plus the nodes.  
\item The extended period map is the continuous map  $$\percompact_{\Sigma,\PP}:\Omega^*_0\mnc_{\Sigma,\PP}\rightarrow \text{Hom}(H_1(\Sigma\setminus \PP,\mathbb{Z});\mathbb{C})$$ defined by integration on cycles in $\Sigma\setminus \PP$
\end{itemize}

\section{Strategy of the proof}
The proof of Theorems \ref{t:Rgn} and \ref{t:Cgn} follows by induction on the genus and the number of poles and are based in two previous papers that treat respectively the case of holomorphic forms (\cite{CDF2}) and of forms with two simple poles (\cite{CD}) plus finding isoperiodic connections via Schiffer variations. 

The base case for Theorem \ref{t:Cgn} is the case of genus zero, that can be treated directly in coordinates (see Lemma \ref{l: connectedness spherical case}).  For Theorem \ref{t:Rgn} the base case includes, both the case of genus zero and the case of genus one and three poles, contained in Section \ref{s:(1,3)}. It is treated separately for reasons that will become evident in the next paragraphs. The strategy for this case is first to join any point in the fiber isoperiodically with a point having a single zero (of order three, so the case of degree two covers has to be avoided to attain this step) and secondly to connect the different points with a single zero isoperiodically. Such a (homologically marked) form with a single zero and real periods can be codified with data associated to its horizontal saddle connection ribbon graph (see Section \ref{eq: combinatorial structure of a form}). The possible topological combinatorics for such ribbon graphs is limited to two cases when $(g,n)=(1,3)$ that can be distinguished by the presence/absence of a peripheral saddle connection.  The second step is attained by using the results in \cite{CD}, and a small collection of known isoperiodic connections between pairs of points obtained by carrying Schiffer variations along horizontal saddle connections (with appropriate choices of twins). The hypothesis on the periods in Theorem \ref{t:Rgn} is used in this last point to guarantee that such basic isoperiodic connections exist.

Fix some $g\geq 1$ and assume the inductive hypothesis, i.e. Theorem \ref{t:Rgn} (resp. \ref{t:Cgn}) is true up to genus $g-1$. Then run the following program for any given homomorphism $p:H_1(\Sigma\setminus\PP,\mathbb{Z})\rightarrow \mathbb{C}$ with image either not contained in a real line or not contained in $\mathbb{Q}\otimes p(\Pi_\PP)$: 

{\bf Step 1. Bordify $\per^{-1}_{\Sigma,\PP}(p)$} (see subsection \ref{ss:bordification}).  The Teichm\" uller space $\mathcal{T}_{\Sigma,\PP}$ admits a topological bordification $\overline{\mathcal{T}}_{\Sigma,\PP}$ formed by marked stable (nodal) curves of genus $g$ with $n$ marked points. It is stratified by the number of nodes of the underlying curves and each stratum is a complex manifold.  The action of $\text{Mod}_{\Sigma,\PP}$ preserves the stratification and the quotient is isomorphic to the Deligne-Mumford compactification $\overline{\mathcal{M}}_{g,n}$ of $\mathcal{M}_{g,n}$ where $g$ is the genus of $\Sigma$ and $n$ the cardinality of $\PP$. The added points, called the boundary, form a normal crossing divisor in $\overline{\mathcal{M}}_{g,n}$ at the orbifold chart level and the stratification coincides at this level with the stratification associated to such a divisor.  Quotienting $\overline{\mathcal{T}}_{\Sigma,\PP}$ by the pointed Torelli group  $\mathcal{I}_{\Sigma,\PP}$ produces a  stratified bordification $\overline{\mathcal{S}}_{\Sigma,\PP}$ of $\mathcal{S}_{\Sigma,\PP}$ and the natural map $\overline{\mathcal{S}}_{\Sigma,\PP}\rightarrow \overline{\mathcal{M}}_{g,n}$ is a topological branched cover over the boundary. The bundle $\Omega\overline{\mathcal{M}}_{g,n}\rightarrow \overline{\mathcal{M}}_{g,n}$ of stable  meromorphic forms with at worst simple poles at marked points, over the Deligne-Mumford compactification $\overline{\mathcal{M}}_{g,n}$ of $\mathcal{M}_{g,n}$ can be pulled back to a bundle $\Omega\overline{\mathcal{S}}_{\Sigma,\PP}\rightarrow\overline{\mathcal{S}}_{\Sigma,\PP}$. 

Any stable form in the boundary has a non-trivial local isoperiodic deformation space. Two conditions that guarantee that this local isoperiodic deformation space leaves the boundary are that the form has no zero components and the residue of the form at each non-separating node is zero. The bordification of $\per_{\Sigma,\PP}^{-1}(p)$ that we are interested in is its closure in the space  $\Omega^{*}_{0}\overline{\mathcal{S}}_{\Sigma,\PP}$ 
of forms in $\Omega\mathcal{S}_{\Sigma,\PP}$  having zero residues at all non-separating nodes and no zero components. In fact, as shown in the appendix of \cite{CD}, the local isoperiodic deformation space projects to a smooth complex manifold in the orbifold charts of the moduli space transverse to each boundary component of  $\Omega\overline{\mathcal{M}}_{g,n}$ passing through the point. The stratification of the boundary (defined by the number of nodes) induces a stratification of the bordification of the isoperiodic set and for each stratum of the ambient space passing through the point there is one isoperiodic component of the stratum that lies in it.  Around a point having only separating nodes, the local picture of the stratification is that of a normal crossing divisor. The picture changes by an  abelian ramified cover over the divisor when the curve underlying the form has at least one non-separating node (see the local model of this local branched cover in \cite{CDF2}[Section 4.4]).  The abelian ramified cover does not brake a nice property of the local stratification of a normal crossing divisor: a point in the codimension $k\geq 1$ stratum lies at the intersection of the closure of $k$ codimension one (local) connected components of the divisor. Any other local connected component of a non-open stratum accumulating the point has codimension $1\leq l\leq k$ and is \textit{precisely} the set of points that belong to the closure of $l$ of the $k$ codimension one components accumulating the point, and not more.  Moreover, the open stratum is locally connected at the point. 

To any stratified space  $X$ that is a locally abelian ramified cover over a normal crossing divisor we can define its dual boundary graph $\mathcal{C}(X)$. It has a vertex for each (global) connected component of the codimension one stratum, and a simplex between $k$ vertices for each connected component of the codimension $k$ stratum lying in the closure of the corresponding $k$ components. It is well known that the boundary complex associated to $\overline{\mathcal{T}}_{\Sigma,\PP}$ is isomorphic to  the \textit{curve complex}  $\mathcal{C}_{\Sigma,\PP}$ on a genus $g$ compact surface $\Sigma$ with $n$ marked points $\PP\subset \Sigma$ (see \cite{FM}{Chapter 4.1}). 

The closure of $\per_{\Sigma,\PP}^{-1}(p)$ in $\Omega^{*}_{0}\overline{\mathcal{S}}_{\Sigma,\PP}$ is also shown to be  stratified and a locally abelian ramified cover  over a normal crossing divisor. In particular,  the connectedness of  $\per_{\Sigma,\PP}^{-1}(p)$ is equivalent to that of its bordification. 

Moreover, the transversality condition allows to define a continuous map of complexes  \begin{equation}\label{eq:complex inclusion}\mathcal{C}(\per_{\Sigma,\PP}^{-1}(p))\rightarrow \mathcal{C}(\Omega\overline{\mathcal{S}}_{\Sigma,\PP})\end{equation} that associates to each component of an isoperiodic stratum the component of the ambient stratification where it sits.

For each different case of $p$ in Theorems \ref{t:Rgn} and \ref{t:Cgn} there is a subfamily of components of the codimension one stratum of the ambient space $\Omega\overline{\mathcal{S}}_{\Sigma,\PP}$ where we will be able to prove inductively that there is a single connected isoperiodic component of the isoperiodic stratum of codimension one associated to $p$. They are associated to a subfamily of strata that we call spherical boundary strata, because the simple closed curve representing each of them separates the surface in a component of genus zero (a sphere) with at least two of the poles, and a component of the same genus as $\Sigma$ with the rest of poles.  

When $p$ has periods not contained in a real line in $\mathbb{C}$, the spherical boundary strata that will play a central role are those where the spherical part contains all the poles, and hence the genus $g$ part is holomorphic.

When $p$ has periods in a real line, the previous situation cannot occur, as any holomorphic form has a lattice (in $\mathbb{C}$) contained in its group of periods. Therefore we consider spherical boundary strata where the spherical part misses just one of the poles, and the genus $g$ part has two poles.

The chosen components define a subfamily of vertices of $\mathcal{C}(\per_{\Sigma,\PP}^{-1}(p))$ that span a subcomplex that we denote by $\mathcal{C}'(\per_{\Sigma,\PP}^{-1}(p))$.

{\bf Step 2. Isoperiodic degeneration towards boundary points}. We will first prove that any point in $\per_{\Sigma,\PP}^{-1}(p)$ can be isoperiodically deformed in $\Omega_{0}^{*}\overline{\mathcal{S}}_{\Sigma,\PP}$ to some boundary point. For this aim we will use Schiffer variations along twin geodesics of the flat singular metric induced by any non-zero meromorphic differential on a smooth curve with some zero. 

{\bf Step 3.  Connectedness of the boundary.}  By using the inductive hypothesis, we will first show that any point in the boundary can be isoperiodically deformed (in the boundary) to some point of the types defined by the vertices of $\mathcal{C}'(\per^{-1}(p))$. 
Moreover,  the inductive hypothesis allows to prove that the restriction of the map \eqref{eq:complex inclusion} to the subcomplex $\mathcal{C}'(\per^{-1}(p))$ is injective at the level of the vertices. We will prove that it has connected image under the map \eqref{eq:complex inclusion}. 

The problem can be rephrased  in algebraic terms: the complex $\mathcal{C}(\Omega\overline{\mathcal{S}}_{\Sigma,\PP})$ is isomorphic to  $\mathcal{C}(\overline{\mathcal{S}}_{\Sigma,\PP})$  which in turn is isomorphic to  the quotient \begin{equation}\label{eq:quotient curve complex} \frac{\mathcal{C}_{\Sigma,\PP}}{\mathcal{I}_{\Sigma,\PP}}\end{equation}  of the curve complex $\mathcal{C}_{\Sigma,\PP}$ under the natural action of the group $\mathcal{I}_{\Sigma,\PP}$. Each $k$-simplex in \eqref{eq:quotient curve complex} is characterized by a $\mathcal{I}_{\Sigma,\PP}$-orbit of a family of $k$ disjoint essential simple closed curves $c=c_{1}\sqcup\ldots\sqcup c_{k}$ in $\Sigma\setminus\PP$. Whenever every $c_{i}$ separates $\Sigma\setminus\PP$ in two components, we can characterize the orbit of the $k$-simplex by an appropriate homological decoration of the dual graph corresponding to the curve system $c$. The decoration keeps track, in each component of $\Sigma\setminus\{ c\sqcup\PP\}$, of the homology of the component, the restriction of the intersection form, and the peripheral classes. Moreover it gives the relations between the two different peripheral classes in different components related by each curve $c_i$. 

When $c$ is just one simple closed curve defining a spherical boundary stratum, the boundary stratum is characterized by the following homological information according to the set of points $E\subset\PP$ that belong to the spherical part.

\begin{definition}
  Let $E\subset\PP\subset\Sigma$. An $E$-module is a submodule $M\subset H_1(\Sigma\setminus\PP)$ such that 
\begin{equation} H_1 (\Sigma\setminus\PP, \mathbb Z) = M+ \Pi _E\text{ and }M \cap \Pi_E = \mathbb Z \pi_E =\mathbb Z \pi_{P\setminus E}\text{ where }\pi_E := \sum_{e\in E} \pi_e.\end{equation} 
\end{definition}
As shown in Section \ref{ss: E-modules}, if we fix the set $E$, the set of $E$-modules is an affine space directed by the module 
\[ \text{Hom} \big( H_1 (\Sigma\setminus(\PP\setminus E));\frac{ \Pi_E}{ \mathbb Z \pi _E}\big).\]

In the case of the vertices of $\mathcal{C}'(\per_{\Sigma,\PP}^{-1}(p))$, each vertex is characterized by an $E$-module associated to the subset $E=\PP$ (for the case of non-$\mathbb{R}$-collinear periods) or a set $E$ with $\PP\setminus E$ consisting of one point (for the case of $\mathbb{R}$-collinear periods). It corresponds to the homology module of the component of nonzero genus in the partition $\Sigma\setminus\{c\sqcup\PP\}$ where $c$ is the separating curve that defines the vertex of $\mathcal{C}'(\per_{\Sigma,\PP}^{-1}(p))$. 

We will show that, there some basic isoperiodic paths between components associated to different $E$-modules. They generate all possible $E$-modules when $E$ misses at most one point of $\PP$. In this way we show that the image of $\mathcal{C}'(\per_{\Sigma,\PP}^{-1}(p))$ by the map \eqref{eq:complex inclusion} is connected. 

In the case $(g,n)=(1,3)$ with real periods $p$, Steps 1 and 2 are still true, but Step 3 is false. Indeed, in this case the boundary is necessarily disconnected, as an elliptic stable form with three poles, real periods and zero residues at non-separating nodes can have at most one separating node. Therefore there are no intersections between the different components corresponding to different $E$- modules. This is the main reason for which we need a special argument for this particular case.

As was mentionned in the overview, our main application of the connectedness of the fibers of the period map is that it permits to transfer  the dynamical properties of the action of modular group on the set of periods to properties satisfied by the isoperiodic foliation. We notably use this to compute the closures of the leaves, and to prove our ergodicity statements. The detailed analysis of how Ratner's theory might be applied to compute closures of orbits of the modular group on the set of periods is made in the last section of the paper, together with the proof of Theorem \ref{t:leaf closures}. 

\section{Preliminary results and objects}\label{s:preliminary resutls}
\subsection{Stable meromorphic forms with simple poles}
\begin{definition}
  A (possibly singular) connected complex curve $C$ with a set of $n$ marked distinct points $Q\subset C$ is said to be stable if its singularities are nodes and do not coincide with any of the marked points,  and the closure $\hat{C}_i$ of each component of $C\setminus \text{Sing}(C)$, called a part of $C$, has a finite group of automorphisms that fix the set of points $\hat{Q}_i$, defined as the union of the marked points of $Q$ that belong to $\hat{C}_i$ and the boundary points. The normalization of $(C,Q)$ is the (possibly disconnected) smooth curve with marked points $(\hat{C},\hat{Q})=\sqcup (\hat{C}_i,\hat{Q}_i)$. A stable curve $C$ is said of compact type if every node separates $C$ in two components. Otherwise $C$ is said to be of non-compact type. We will also be led to work with the so-called separating normalization of $C$, denoted $\hat{C}^s$ which follows the same definition of the normalization, but considering $C\setminus\text{Sing}^{sep}(C)$ where $\text{Sing}^{sep}(C)$ denotes the set of separating nodes of $C$. 
\end{definition}

The arithmetic genus of a stable curve is $g=h^1(C,\mathcal{O})$; this is the genus of the surface obtained after smoothening the nodes of the curve. When $C$ has $\delta$ nodes and its normalization has $\nu$ components of genera $g_1,\ldots, g_{\nu}$, the arithmetic genus satisfies (see \cite{Harris}[p. 48]) $$g=\sum_{i=1}^\nu (g_i-1)+\delta+1 .$$
\begin{definition}
 A stable meromorphic form with simple poles on a stable curve with  $n$ marked points $(C,Q)\in\overline{\mathcal{M}}_{g,n}$ is a meromorphic form $\omega_i$ with simple poles on each component  $(\hat{C}_i,\hat{Q}_i)$ of its normalization  $(\hat{C},\hat{Q})$ having at worst simple poles at the points of $\hat{Q}_i$ (that correspond either to points of $Q$ or to nodes), and, moreover, at each point of $C$ where there are two branches, the sum of the residues of the forms on the two branches is zero. The space of all such stable forms on $(C,Q)$ can be identified with the space of global sections of the line bundle $K_C(\sum_{q\in Q}q)$ over $C$. It is a vector space of dimension $2g+n-1$. given an indexed set $\alpha=\{\alpha_q\in\mathbb{C}\}_{q\in Q}$ we define $\Omega^{\alpha}(C,Q)$ the set of forms with peripheral period $\alpha_q$ at the point $q\in Q$. It is an affine space over the vector space of holomorphic forms on $C$. 
\end{definition}
  On any component of $\hat{C}$ the residue theorem holds for the restricted form,  telling us that the sum of residues in the component is zero. To be able to integrate $\omega$ along a path in $C$ it needs to avoid all poles of the restrictions to the components of $\hat{C}$. 

The order of $\omega$ at a smooth point $q\in C$ of the underlying curve $C$ is defined to be the $\text{ord}_q(f)$ where $\omega(z)=f(z)dz$ in a holomorphic coordinate $z:(C,q)\rightarrow \mathbb{C}$ around $q$. The order of $\omega$ at a node $q\in C$ is $$\text{ord}_q(\omega)=2+\text{ord}_q(\omega_{|C_1})+\text{ord}_q(\omega_{|C_2})$$ where $C_1$ and $C_2$ are the branches of $C$ at $q$. The condition defining $\omega$ imposes that the residues of the two branches of $\omega$ at a node sum up to zero and both forms have at worst simple poles. So whenever the residue is non-zero at a branch of a node $q$ we have \(\text{ord}_q(\omega)=0\).
 The order of the stable form at any point is clearly invariant by biholomorphism. As opposed to the case of meromorphic forms on smooth curves, a stable form can have a component of zeros (i.e. every point of the component is of infinite order) and not be zero everywhere. It suffices to have a component with isolated zeros to be a non-zero form.
 With these definitions the Riemann-Roch formula is still true for a stable form $\omega$ with no zero components: 
 \[\sum_{q\in C}\text{ord}_q(\omega)=2g-2\] where $g$ is the genus of $C$. 

  \begin{definition}
  $\Omega_{0}(C,Q)$  (resp. $\Omega_{0}^{\alpha}(C,Q)$) is the subset of forms in $\Omega(C,Q)$ (resp.  $\Omega^{\alpha}(C,Q)$) with zero residues at non-separating nodes of $C$ and 
    $\Omega_{0}^{*}(C,Q)$ (resp. $\Omega_{0}^{\alpha *}(C,Q)$) is the subset of forms in   $\Omega_{0}(C,Q)$ (resp. $\Omega^{\alpha}_{0}(C,Q)$) with no zero components (i.e. with isolated zeros). 
  \end{definition}

\subsection{Branched translation structure} \label{ss: singular translation}

In this section, we review classical material from the theory of translation structures. We refer to  \cite{boissy} for more on this topic. The tools in this section allow, among other things, to find twin paths for a given form with simple zeros. 

A non-zero meromorphic differential $(C,Q,\omega)$ with simple poles on a connected smooth curve $C$ induces a singular translation structure with singular set on the support of the divisor of the form.  The local charts of this translation structure are given by the germ obtained by integration of $\omega$ around any point. At a simple pole the form is locally conformally equivalent to the germ at $(\mathbb{C},0)$ of $\frac{adz}{z}$   Around a point of order $d\geq 0$ of $\omega$ the form is locally conformally equivalent to the exact form $d(z^{d+1})$ in $(\mathbb{C},0)$. 
In fact, the construction can be reverted to associate, to any branched translation structure on a surface of finite type $\Sigma\setminus \PP$ with non-trivial local holonomy around each point of $\PP$ a unique pair $(C,\omega)$ of complex structure $C$ on $\Sigma$ and a meromorphic form $\omega$ on $C$ with simple poles on $\PP$, by pulling back the pair $(\mathbb{C},dz)$ to the surface $\Sigma\setminus \PP$ via the charts of the translation structure, and showing that the we can glue the structure of a simple pole of the unique appropriate local monodromy at each point of $\PP$. Under this equivalence, the period homomorphism of the form corresponds precisely to the holonomy of the translation structure. This dictionary between forms and translation structures will be extensively exploited to construct isoperiodic deformations of forms, via the construction of isoholonomic translation structures. Let us analyze in more detail some further structures preserved by translations that we can pull back to $C$.

The above translation structure induces a flat singular metric on the surface, which is the flat metric of $\mathbb{C}$ pulled back by the charts of the translation structure. It can be described on $C$ by the form $\omega\otimes\overline{\omega}$. Its singular set lies is the support of the divisor $(\omega)$ of the form. 

The chart definition around a simple pole shows that the singularity of the metric  is given by a semi-infinite cylinder $\{z\in\mathbb{C}:\Im(z/2\pi i a)>0\}/2\pi i a\mathbb{Z}$ where (the period of a small circle $\gamma_r$ around the pole satisfies $\int_{\gamma_{r}}\omega=2\pi i a$). At a point  $q\in C$ where $\omega$ has order $d\geq 0$, the metric is equivalent to a conical point of angle $2\pi(d+1)$.  

The length of a piecewise smooth path \(\gamma: I \rightarrow C \) is defined by \( l(\gamma)=\int_I |\omega (\gamma '(t) )|dt\), and the distance between two points \(x\) and \(y\) is the infimum of the lengths of piecewise smooth paths that join \(x \) to \(y\). This metric space is locally isometric to a cone with angle an integer multiple of \(2\pi\). 

At a zero of \(\omega\), there is still a notion of angle between two emanating curves with non zero derivative in the translation chart. Geodesics are the curves that are plane geodesics in the flat coordinates outside the zeroes of the form, and when passing a zero, make an angle on both sides at least \(\pi\).

If the form has simple poles, this metric is complete, and the universal cover \(\widetilde{C}\) of \(C\) equipped with the lift \(\widetilde{\omega}\) of \(\omega\)
is then a complete \(\text{CAT}(0)\) space.

The radius of injectivity of \( (C,\omega) \) is then defined as the smallest real number \(r>0\) so that all geodesics of length \(r\) are embedded.   If the form has simple poles, the metric is complete, and the radius of injectivity is positive.

\subsection{Directional foliations induced by $(C,\omega)$}
\label{ss: singular flat metric and directional foliations}
The directional oriented geodesic foliation of $\mathbb{C}$ associated to a unitary vector $\theta\in\mathbb{S}^1\subset\mathbb{C}^*$  is invariant by translations, and therefore lifts to an oriented singular geodesic foliation $\mathcal{G}_\theta$ on $C$ that has saddle at each  zero of $\omega$,  with $2d$ local oriented separatrices each forming an angle of $\pi$ and opposite direction with the two adjacent separatrices. At a pole, of peripheral period $2\pi i a\in\mathbb{C}^*$, the oriented foliation $\mathcal{G}_{\theta}$ 
has 
\begin{itemize}
\item a center if $2\pi i a$ and $\theta$ are $\mathbb{R}$--collinear, 
\item a source if $\theta$ lies in the right half plane of $\mathbb{C}\setminus 2\pi i a\mathbb{R}$ where the side determined by the direction of $2\pi i a$ and
\item a node if $\theta$ lies in the other (left) half plane of  $\mathbb{C}\setminus 2\pi i a\mathbb{R}$.

\end{itemize}

\begin{definition} A generalized saddle connection of $\mathcal{G}_\theta$ is an oriented geodesic leaf that has $\alpha$-limit (resp. $\omega$-limit) formed by precisely one singular point of the underlying metric (a zero or a pole of the form) and at least one of both is a saddle (zero of the form). If both starting and endpoint are saddles (zeros of $\omega$) we simply call it a saddle connection. If moreover they coincide we call it a closed saddle connection at a zero of the form. 
\end{definition}

We define $ R(\mathcal{G}_{\theta})$ as the set of generalized saddle connections of $\mathcal{G}_\theta$.  The set $C\setminus R(\mathcal{G}_{\theta})$ is a $\mathcal{G}_{\theta}$-invariant open set in $C$. According to \cite[Proposition 5.5]{Tahar1} each of its connected components is in one of the following possibilities: 
\begin{itemize}
    \item a (possibly semi-infinite) cylinder of closed regular geodesics
    \item a minimal component
    \item a band of leaves that join two distinct poles
\end{itemize}

A semi-infinite cylinder of closed geodesics occurs at a pole only when the chosen direction is parallel to the direction of the peripheral period of the form at the pole.

If there is a  geodesic joining two (necessarily different) poles, then there is a band of geodesics joining the two poles. Then, either the form defines a cylinder $(\mathbb{C}^*,\frac{adz}{z})$, or, for each of the two poles, there is a generalized saddle connection reaching the pole. 

The semi-infite cylinders of closed geodesics can only occur for directions parallel to some peripheral period. 
If  all the peripheral periods at the poles are contained in a real line $\mathbb{R}\theta\subset \mathbb{C}$, then there is a semi-infinite cylinder around each pole. Therefore $R(\mathcal{G}_{\theta})$ has only saddle connections (between zeros of $\omega$). The possibility of bands joining poles is ruled out.

If moreover, all the periods on closed curves (not only the peripheral) lie in $\mathbb{R}\theta$ then $\mathcal{G}_{\theta}$ admits a real first integral-- namely $z\mapsto \int_{z_0}^z\text{Im}( \frac{\omega}{\theta})$ -- and the only possibility is the case of cylinders of closed geodesics. This case is impossible for holomorphic forms, as their periods always contain a lattice in $\mathbb{C}$ (due to the positivity of the volume and Riemann's relations). 

If $\omega$ is a stable form on a stable nodal curve $(C,Q)$ with marked points and no zero components, we have, on each component of its normalization, the translation structure and the directional foliations defined above.

At a node the metric is either a union of two semi-infinite half cylinders (with the same peripheral period) or a union of two disjoint conical points of angle multiple of $2\pi$ (if the residue of the node is zero) identifying the two singularities . 

\section{Boundary strata of Torelli space bordification}
\subsection{Bordifications of the pointed Torelli space and their stratification}
\label{ss:bordification}

Let \(\Sigma\) be a closed connected oriented surface of genus \(g\), and \(\PP\subset \Sigma\) be a finite subset of \(\Sigma\) of cardinal \(n\). The Teichm\"uller space \(\mathcal T _{\Sigma, \PP} \) is a complex manifold whose points parametrizes triples \( (C, Q, [f])\)
where \(C\) is a closed Riemann surface of genus \(g\) (the curve), \(Q\subset C\) is a finite subset of cardinal \(n\) (the marked points), and \([f]\) denotes the isotopy class of orientation preserving diffeomorphism \( f : (\Sigma, \PP) \rightarrow (C, Q)\) (the marking).

Recall from the Appendix in \cite{CD} and references therein that $\mathcal{T}_{\Sigma, \PP}$ can be bordified to a topological stratified space $\overline{\mathcal{T}} _{\Sigma, \PP} $ called the augmented Teichm\" uller space formed by homotopy equivalence classes of triples $(C,Q,[f])$ where  $(C,Q)$ is a stable curve of genus $g$ with $n$ marked points and $f:(\Sigma, \PP) \rightarrow (C,Q)$ is a collapse map, i.e. the preimage of each node is a simple closed curve and on the complement of this disjoint union of simple closed curves, it is a homeomorphism onto the complement of the nodes in $C$ sending marked points to marked points. The number of nodes of a stable curve determines the codimension of the stratum where it belongs to.  The open stratum is the dense connected set $\mathcal{T} _{\Sigma, \PP} $. 

For each isotopy class of a system of curves, i.e. a disjoint union $c=c_1\sqcup\cdots\sqcup c_k$ of essential\footnote{not boundary of a disc or a once punctured disc in $\Sigma\setminus\PP$} non isotopic simple closed curves $c_i$ on the punctured surface $\Sigma \setminus \PP$ there is a connected component of the stratum of codimension $k$ in the boundary, i.e. the complement of $\mathcal{T}_{\Sigma, \PP}$ in $\overline{\mathcal{T}}_{\Sigma, \PP}$ formed by marked stable curves whose collapse map is homotopic to the pinching $\Sigma\rightarrow \Sigma/\sqcup c_i$ of each simple closed curve to a distinct point. This stratum is denoted by \( \mathcal B_c\subset \mathcal T _{\Sigma, \PP}\).

It has a complex manifold structure that can be described by using the so-called attaching maps. Namely, associated to the system of curves \(c\), let us consider the surface \( \hat{\Sigma} \) obtained from the geometric completion \(\overline{\Sigma\setminus c}\) of \(\Sigma\setminus c\) by contracting any boundary component to a point. Inside \(\hat{\Sigma}\), one considers the finite subset \( \hat{\PP}\) obtained by taking all the points of \(\PP\) together with all the points obtained by contracting the boundary components of \(\overline{\Sigma\setminus c}\). Letting \( (\hat{\Sigma}_j, \hat{\PP}_j)\) be the connected components of the pointed surface \( (\hat{\Sigma}, \hat{\PP}) \), one then has a natural parametrization 
\[ \mathcal B_c  \simeq \prod _ j \mathcal T _{\hat{\Sigma}_j, \hat{\PP}_j} \]
From left to right, this isomorphism is defined by taking the normalization of a marked pointed curve \( (C, Q, f)\in \mathcal B_c\), and restricting the homeomorphism $\hat{f}:(\hat{\Sigma},\hat{\PP})\rightarrow (\hat{C},\hat{Q})$ to the components.  From right to left, it is just attaching map following the information given by $c$.

The previous isomorphism can be used to arrange the factors in another combination that will be useful for the purpose of extending the definition of the period map to a bordification of $\Omega\mathcal{T}_{\Sigma,\PP}$ as follows:  let $c=c^s\sqcup c^{ns}$ where $c^s$ correspond to the subset of separating curves of $c$. Consider the so-called separating normalization of $\Sigma$ associated to $c^s$, i.e. the surface \( \hat{\Sigma}^s \) obtained from the geometric completion \(\overline{\Sigma\setminus c^s}\) of \(\Sigma\setminus c^s\) by contracting any boundary component to a point. Inside \(\hat{\Sigma}^s\), one considers the finite subset \( \hat{\PP}^s\) obtained by taking all the points of \(\PP\) together with all the points obtained by contracting the boundary components of \(\overline{\Sigma\setminus c^s}\). Letting \( \{(\hat{\Sigma}^s_v, \hat{\PP}^s_v)\}_{v\in V}\) be the connected components of the pointed surface \( (\hat{\Sigma}^s, \hat{\PP}^s) \); $c^{ns}_v$ be  the subset of $c^{ns}$ of curves that lie in \( \hat{\Sigma}^s_v\setminus  \hat{\PP}^s_v\) and \(\mathcal{B}_{c^{ns}_v}\subset \overline{\mathcal{T}}_{\hat{\Sigma}^s_v, \hat{\PP}^s_v}\) be the boundary component corresponding to the non-separating curve system $c^{ns}_v\subset \hat{\Sigma}^s_v\setminus \hat{\PP}^s_v$, one then has a natural isomorphism 
\begin{equation} \label{eq:partial norm teich}
\mathcal B_c  \simeq \prod _ {v\in V}\mathcal B _{c^{ns}_v} 
\end{equation}
that associates to $f:(\Sigma,\PP)\rightarrow (C,Q)$, the restriction of the collapse map $\hat{f}^s:(\hat{\Sigma}^s,\hat{\PP}^s)\rightarrow (\hat{C}^s,\hat{Q}^s)$ to each component. 

The combinatorics of adjacency of the different boundary strata is codified by the curve complex $\mathcal{C}_{\Sigma, \PP}$, that has a vertex for each isotopy class of essential simple closed curve in $\Sigma\setminus \PP$, and we add a $k+1$ simplex to a finite number of vertices associated to simple closed curves $c_1,\ldots, c_k$ if their union forms a curve system up to isotopy. Sometimes it will be useful to work on a bordification that considers only stable curves with separating nodes, i.e. of compact type. The associated curve systems will be referred to as separating curve systems.  In this case we denote the space by $\overline{\mathcal{T}}^{sep}_{\Sigma, \PP}$. The corresponding adjacency combinatorics of boundary strata is codified by the separating curve complex $\mathcal{C}^{sep}_{\Sigma, \PP}$ having a similar definition as the curve complex, but imposing that all vertices correspond to separating curves in $\Sigma\setminus \PP$, i.e. the subcomplex of $\mathcal{C}_{\Sigma, \PP}$ spanned by vertices corresponding to separating curve systems. 

The group $\text{Mod}_{\Sigma, \PP}$ acts on $\overline{\mathcal{T}}_{\Sigma, \PP}$ by pre-composition on the collapse maps and preserves the boundary stratification--an element of $\text{Mod}_{\Sigma, \PP}$  sends curve systems to curve systems -- by possibly exchanging connected components of strata. The quotients 
\[ \overline{\mathcal{M}}_{g,n}=\frac{\overline{\mathcal{T}}_{\Sigma, \PP}} {\text{Mod}_{\Sigma, \PP} }\quad \text{ and its subset }\quad\overline{\mathcal{M}}^{sep}_{g,n}=\frac{\overline{\mathcal{T}}^{sep}_{\Sigma, \PP}} {\text{Mod}_{\Sigma, \PP}}.\] are respectively a (compact) orbifold isomorphic to the Deligne-Mumford compactification of $\mathcal{M}_{g,n}$ (see \cite{ACG}), and a (Zariski) open subset of the latter consisting of a bordification of $\mathcal{M}_{g,n}$ by adding isomorphism classes of stable curves with separating nodes. 
We consider the pointed Torelli group $\mathcal{I}_{\Sigma, \PP}$ defined by the exact sequence \begin{equation}\label{eq:pointedtorelligroup}1\rightarrow \mathcal{I}_{\Sigma, \PP}\rightarrow \text{Mod}_{\Sigma, \PP}\rightarrow \text{Aut}\big(H_1(\Sigma\setminus \PP,\mathbb{Z})\big),\end{equation} 
and define the partial bordifications of the space $\mathcal S _{\Sigma, \PP}$

\[ \overline{\mathcal{S}}_{\Sigma, \PP}=\frac{\overline{\mathcal{T}}_{\Sigma, \PP}} {\mathcal{I}_{\Sigma, \PP} }\quad \text{ and its subset }\quad\overline{\mathcal{S}}^{sep}_{\Sigma, \PP}=\frac{\overline{\mathcal{T}}^{sep}_{\Sigma, \PP}} {\mathcal{I}_{\Sigma, \PP}}.\]
The quotient boundary stratification has adjacency combinatorics given respectively by the quotient of $\mathcal{C}_{\Sigma, \PP}$ (resp. $\mathcal{C}^{sep}_{\Sigma, \PP}$) by the action of $\mathcal{I}_{\Sigma, \PP}$ on the complexes. A connected component of a boundary stratum of $\overline{\mathcal{S}}_{\Sigma, \PP}$ (resp.  of $\overline{\mathcal{S}}^{sep}_{\Sigma, \PP}$) is characterized by a pointed Torelli class of the corresponding curve system.  Hence, the combinatorics of adjacency of boundary strata is codified by  $\mathcal{C}_{\Sigma, \PP}/\mathcal{I}_{\Sigma, \PP}$ (resp.  $\mathcal{C}^{sep}_{\Sigma, \PP}/\mathcal{I}_{\Sigma, \PP}$) and each simplex corresponds to a pointed Torelli class of curve system in $\Sigma\setminus \PP$. 

\subsection{Homological characterization of some boundary strata $\overline{\mathcal{S}} _{\Sigma, \PP}$}

In this subsection we will prove that each pointed Torelli class of either a separating curve system (i.e. simplex of  $\mathcal{C}^{sep}_{\Sigma, \PP}/\mathcal{I}_{\Sigma, \PP}$) or a unique non-separating curve can be characterized by homological information induced by the curve system on $H_1(\Sigma \setminus \PP)$. It will notably allow to characterize homologically the associated strata of the bordification of Torelli space \(\overline{\mathcal{S}}_{\Sigma,\PP}\). The important concept enabling this characterization is the one of homologically marked labelled graph; even if conceptually quite natural and simple, the material needed for its definition is quite technical.


We will need to use decorated graphs associated to stable curves. We start by recalling the definition of a labeled graph from \cite[p. 98]{AC}. 
A graph is the datum $G$ of a non-empty finite set $V=V(G)$ called the vertices, a non-negative integer \(g_v\) for each $v\in V$, a finite set \(L=L(G)\) called the half legs of \(G\), a partition \(\mathfrak{P}\) of \(L\) into subsets with either one or two elements, and a subset \(L_v\subset L\) for every \(v\in V\) satisfying
\( L=\bigsqcup_{v\in V} L_v\). 
The elements of \(\mathfrak{P}\) with one element will be called legs of $G$ and those with two elements correspond to the edges \(E=E(G)\) of the graph. We also set \(l_v=|L_v|\). By a  \(\PP\)-labelled graph we mean the datum of a graph \(G\) with a bijection between the set of its legs and \(\PP\). Such a graph permits to codify the combinatorics of a stable curve with marked points, as follows. 

A \(\PP\)-pointed stable curve is a pointed stable curve \((C, Q)\) with an identification \(r\in \PP\mapsto q_r \in Q\) of the marked points with \(\PP\). In \cite[p. 98]{AC} the authors associate to such a \(\PP\)-pointed stable curve an $\PP$-labelled graph associated to the normalization $\hat{C}\rightarrow C$ of $C$. For our purposes we need just to define it at the level of the partial normalization $\hat{C}^{s}\rightarrow C$ that results from normalizing the \emph{separating} nodes of $C$. We will call such a normalization the separating normalization. It is the same as the full normalization if all nodes are separating.

Given a $\PP$-pointed genus \(g\) stable curve \( (C,Q=\{q_r\}_{r\in\PP})\) we associate a \(\PP\)-labelled graph \(G\) as follows: for its separating normalization \(\pi^s: \hat{C}^s\rightarrow C\) we let \(V(G)\) be the set of components of \(\hat{C}^s\), \(L(G)\) the set of points of \(\hat{C}^s\) that correspond either to nodes or to marked points. Two of those half edges correspond to an edge if they are mapped to the same point (a node) by \(\pi^s\). The indexing of the legs by \(\PP\) is the obvious one and we define \(g_v=\text{genus of } v\) and \(L_v\) the set of elements of \(L\) belonging to \(v\). By construction we have\[ g=\left(\sum_{v\in V(G)} g_v\right)+1-|V(G)|+|E(G)|.\] The fact that we only normalize separating nodes, implies that the graph $G$ is actually a tree.


In the sequel, we explain how to keep track of some homological information of stable pointed curves. We will use the following notation: denote by 
\begin{equation} \label{eq: psi geometric}\psi_{\Sigma, \PP}: \PP\rightarrow H_1 (\Sigma\setminus \PP, \mathbb Z)\end{equation} the map that sends an element \(r\in \PP\) to  the peripheral class $\pi_r$. A homological label of a \(\PP\)-labelled graph is a collection \[\{ (W_v,<,>_v, \psi_v)\}_{v\in V(G)}\]
where for each \(v\in V\)
\begin{itemize}
\item \(W_v\) is a \(\mathbb Z\)-module 
\item \(<,>_v\) is a \(\mathbb Z\)-linear anti-symmetric bilinear form,
\item \(\psi_v: L_v \rightarrow W_v\) is a map with image in the kernel of \(<,>_v\).
\end{itemize}
so that for each \(v\in V\), the triple $(W_v,<,>_v,\psi_v)$ is isomorphic to a triple of the form  $ (H_1 (\Sigma_v\setminus \PP_v, \mathbb Z), \cdot_v, \psi_{\Sigma_v, \PP_v}) $ where \((\Sigma_v, \PP_v)\) is a \(L_v\)-pointed curve of genus \(g_v\), \(\cdot_v\) is the intersection form on \(H_1(\Sigma_v\setminus \PP_v,\mathbb Z)\), and \(\psi_{\Sigma_v, \PP_v}\) is the map defined similarly as \eqref{eq: psi geometric}.

For each such labelled tree \(G\) we associate a triple \( (W(G), <,>, \psi) \) where \(W(G)\) is the module 
\[ W(G) :=\frac{\bigoplus_{v\in V(G)}W_v} {\{\psi_v(e)+\psi_{v'}(e')=0:e\in L_v,e'\in L_{v'}\text{ s.t. } e\cup e' \in E(G)\}},\] \(<,>\) is the anti-symmetric bilinear form induced by the forms \( <,>_v\) on each summand, and  \(\psi_G : \PP\rightarrow W(G)\) is the map that sends an element \(r\in \PP\) to the element \( \psi_{v(r)} (l_r) \) (here \( l_r \) is the half leg corresponding to \(r\) at the vertex \(v_r\)).

A homological $(\Sigma, \PP)$-marking of a homologically labelled \(\PP\)-labelled tree $G$ is an isomorphism 
\begin{equation}
\label{eq:homological marking}m:(H_1(\Sigma \setminus \PP, \mathbb Z),\cdot, \psi_{\Sigma, \PP})\rightarrow (W(G),<,>, \psi_G)
\end{equation}
that associates to each peripheral class of a puncture $r\in \PP$, the class in $W(G)$ of the half leg that corresponds to $r$.



There is a canonical homological \((\Sigma, \PP)\)-marking of the \(\PP\)-labeled tree \(G\) corresponding to an element $[(C,Q,f)]\in \overline{\mathcal{T}}_{\Sigma, \PP}$. Indeed, via a collapse map $f:(\Sigma,\PP)\rightarrow (C,Q)$ representing the point,  the connected components of the  separating normalization surface \((\hat{\Sigma}^s, \hat{\PP}^s) \) are in bijection with the set \(V(G)\) of vertices of \(G\); we denote by \( (\hat{\Sigma}^s_v, \hat{\PP}^s_v)\) the one corresponding to the vertex \( v\). The set of half legs at \(v\) is identified with the set of marked points of \(\hat{\Sigma}^s_v\), namely \(L_v \simeq \hat{\PP}_v\). 
The family \(\{ (H_1(\hat{\Sigma}^s_v \setminus \hat{\PP}^s_v, \mathbb Z), \cdot_v, \psi_v:=\psi_{\hat{\Sigma}^s_v, \hat{\PP}^s_v})\}_{v\in V(G)}\) is a homological label of the graph \(G\). Moreover, the use of Mayer-Vietoris exact sequences (obtained by cutting \(\Sigma\) along the separating curves that are collapsed to the nodes by $f$) shows that the natural inclusion maps \( H_1(\hat{\Sigma}^s_v \setminus \hat{\PP}^s_v, \mathbb Z)\rightarrow H_1(\Sigma \setminus \PP, \mathbb Z)\) can be glued to produce a homological \((\Sigma, \PP)\)-marking of \(G\):
\begin{equation}
\label{eq:homology decomposisiton separating syst}H_1(\Sigma\setminus\PP) \overset{m}{\cong}\frac{\bigoplus_{v\in V(G)}H_1(\Sigma_v\setminus \PP_v)} {\{\psi_v(e)+\psi_{v'}(e')=0:e\in L_v,e'\in L_{v'}\text{ s.t. } e\cup e' \in E(G)\}},\end{equation}
that preserves the intersecion products and the identifications of peripheral classes. This homological $(\Sigma,\PP)$-marking  $(G,m)$ does not depend on the isotopy class of $f$.  Moreover, for all elements in a stratum $\mathcal{B}_{c}$ it is the same.
\begin{proposition}
\label{p:partial norm torelli}
With the notations of the isomorphism \eqref{eq:partial norm teich} , let $[c]$ (resp. $[c^{ns}_v]$) be the $\mathcal{I}_{\Sigma,\PP}$ (resp. $\mathcal{I}_{\hat{\Sigma}^s_j,\hat{\PP}^s_v})$ class of the curve system $c$ (resp. $c^{ns}_v$) in $\Sigma\setminus\PP$ (resp. $\hat{\Sigma}^s_v\setminus \hat{\PP}^s_v$) and consider $(G,m)$ the $(\Sigma,\PP)$-marking of a labelled tree obtained from any element in $\mathcal{B}_{c}$. Then, the isomorphism \eqref{eq:partial norm teich} induces an isomorphism 
\begin{equation} \label{eq:partial norm Torelli}
\mathcal B_{[c]}  \simeq \prod _ {v\in V(G)} \mathcal B _{[c^{ns}_v]} 
\end{equation}
where $\mathcal B_{[c]}\subset \overline{\mathcal{S}}_{\Sigma,\PP}$ is the quotient of $\mathcal B_c$ by $\mathcal{I}_{\Sigma,\PP}$ and $\mathcal B_{[c^{ns}_j]}\subset \overline{\mathcal{S}}_{\hat{\Sigma}^s_j,\hat{\PP}^s_j}$ is the quotient of $\mathcal B_{c^{ns}_j}$ by $\mathcal{I}_{\hat{\Sigma}^s_v,\hat{\PP}^s_v}$ 
\end{proposition}

\begin{proof}[Proof of Proposition \ref{p:partial norm torelli}] We fix a representative $c$ of the pointed Torelli class $[c]$ and consider each point $(C,Q,f)\in \mathcal{B}_c\subset \overline{\mathcal{T}}_{\Sigma,\PP}$ represented by an $f:(\Sigma,\PP)\rightarrow (C,Q)$ that collapses the curves of $c$. The action of $\mathcal{I}_{\Sigma,\PP}$ on $\mathcal{B}_c$ is then by orientation preserving homeomorphisms $\phi:(\Sigma,\PP)\rightarrow (\Sigma,\PP)$ that preserve the curve system $c$. The following Lemma gives some more precise information on this action:

\begin{lemma}
\label{l:torelli preserving scs}
 Let $c^s\subset \Sigma\setminus \PP$ be a curve system of separating simple closed curves. Suppose $\phi:(\Sigma,\PP)\rightarrow(\Sigma,\PP)$ is an orientation preserving homeomorphism such that $\phi(c^s)=c^s$, $\phi_{|\PP}=\text{Id}$   and $\phi_*:H_1(\Sigma\setminus\PP)\rightarrow H_1(\Sigma\setminus\PP)$ is the identity map. Then for each simple closed curve $c_j\subset c^s$, $\phi(c_j)=c_j$ with the same orientations. 
\end{lemma}
Assume Lemma \ref{l:torelli preserving scs} is true. Then the restriction of  $\phi$ to the components of $(\hat{\Sigma}^s,\hat{\PP}^s)$ define automorphisms on each component, and morover, an element of \begin{equation}\label{eq:prod of torelli of parts}
\prod_v\mathcal{I}_{\hat{\Sigma}_v^s,\hat{\PP}^s_v}
\end{equation} 
Reciprocally, an element of \eqref{eq:prod of torelli of parts} produces a unique element in $\mathcal{I}_{\Sigma,\PP}$ that preserves each separating curve $c_j$ of $c^s$. 
With this identification of the pointed Torelli groups, the map \eqref{eq:partial norm teich} is equivariant with respect to the actions on source and target thus inducing the desired isomorphism \eqref{eq:partial norm Torelli} on the quotients. 
\begin{proof}[Proof of Lemma \ref{l:torelli preserving scs}:] By induction on the number of curves of the curve system. Suppose first that $c$ is a single separating curve. The hypothesis on the image of $\phi$ already implies that $\phi(c)=c$.  The restriction of $\phi $ to each component of $\Sigma\setminus c$ induces a homeomorphism from one component to another component (maybe the same). We need to show that it is actually an automorphism of the component. Consider the decomposition \eqref{eq:homology decomposisiton separating syst}  induced by $c$. The restriction of $\phi$ to $\Sigma\setminus c$ induces a map in homology that might permute the factors of the decomposition.  If we show that each factor in homology is preserved we will be done.  If $\PP$ is empty then there are no relations in the quotient decomposition, and the sum is a direct sum of the two modules. As $\phi_{*}$ is the identity in homology, it is also the identity by restriction to each factor and we are done. If $\PP$ is nonempty, then, as $\phi$ is the identity restricted to $\PP$, we have that the peripheral $\pi_{r}$ is fixed by $\phi_{*}$. This implies that one component of $\Sigma\setminus c$ is preserved by $\phi$ and therefore both are. 
Suppose we have proved the statement of Lemma \ref{l:torelli preserving scs} for any separating curve system with at most $k-1\geq 1$ elements. Take a curve system $c$ with $k$ separating curves. To prove the inductive step it suffices to find an element $c_{k}$ of the curve system such that $\phi(c_{k})=c_{k}$ with the same orientation.  Indeed, we can then apply the inductive hypothesis to each component of $\Sigma\setminus c_{k}$ to conclude (as each component has at most $k-1$ separating curves that lie in $c\setminus c_{k}$ and $\phi_{*}$ restricted to the factors of the associated decomposition is the identity). As $c_{k}$ we take a component $(\Sigma_{v},\PP_{v})$ of $\Sigma\setminus c$ that has as unique boundary component in a curve $c_{k}$. If $\PP_{v}$ contains at least two points, then the peripheral class of a point $r\in\PP_{v}$ that does not correspond to $c_{k}$ is sent by $\phi_{*}$ to itself. This implies that $\phi$ restricted to   $\Sigma\setminus c_{k}$ preserves each component. Since it also sends the boundary component $c_{k}$ of  $\Sigma_{v}$ to some boundary component of $\Sigma_{v}$ the only option is $\phi(c_{k})=c_{k}$ with the same orientations. If $\PP_{v}$ is a set with a single point (corresponding to the boundary component $c_{k}$) then the decomposition \eqref{eq:homology decomposisiton separating syst} can be expressed as a direct sum of $H_{1}(\Sigma_{v}\setminus \PP_{v})\cong H_{1}(\Sigma_{v})$ with another module. As $\phi_{*}$ is the identity, every element in this (non-trivial) module is fixed by $\phi_{*}$. Therefore $\phi$ restricted to $\Sigma_{v}$ has image in $\Sigma_{v}$ and $\phi$ preserves one of the components of $\Sigma\setminus c_{k}$. Hence it preserves  each of the two components. 
\end{proof}
\end{proof}

\begin{remark}\label{rem:homol non-injectivity } The natural inclusion $\hat{\Sigma}^s_v\setminus \hat{\PP}^s_v\rightarrow \Sigma\setminus \PP$ does not always induce an injective map at the homological level. When it is injective,  the identification can be used to mark the corresponding part homologically with the image submodule in $H_1(\Sigma\setminus\PP)$. 
\end{remark}

We will use a pictorial representation for  a  homologically \((\Sigma, \PP)\)-marked labelled tree by drawing the  $\PP$-labeled tree, adding the module corresponding to each vertex (represented as a circle surrounding the module) and adding the information of the peripheral classes (belonging to the module of the corresonding vertex) at the end of each half leg belonging to the corresponding module. The classes corresponding to edges can be of the tree can be deduced from the ones of the half legs, since in each component, the sum of all half legs is the zero class.

\begin{example}
An example of the decomposition associated to a marking $f:(\Sigma,\PP)\rightarrow (C,Q)$ of a genus $g$ stable curve $C$ with two separating nodes, three parts of genus \(0,g-1\) and $1$ and seven  marked points $Q=\{q_1,\ldots, q_7\}$
is depicted in Figure \ref{fig:decomposition example}.

\begin{figure}[h]
     \centering
     \includegraphics[height=5cm]{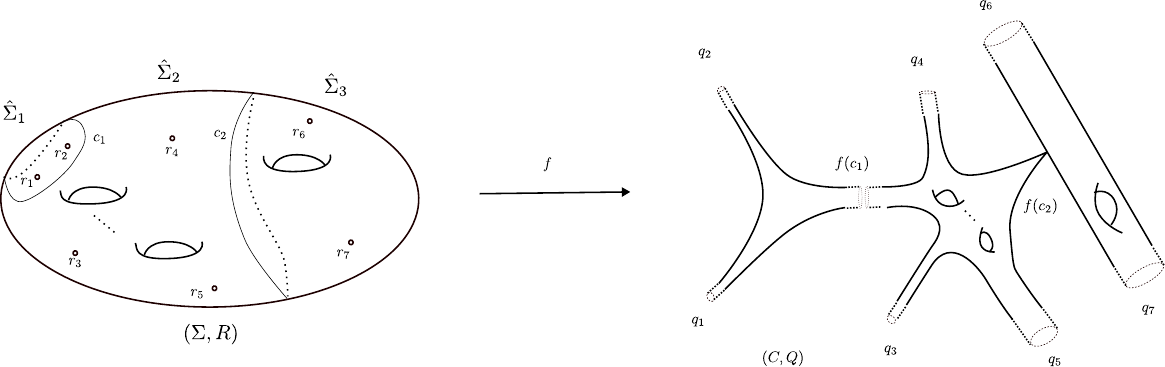}
     \caption{Geometric interpretation of the $(\Sigma,\PP)$-homological marking induced by a marking $f:(\Sigma,\PP)\rightarrow (C,Q)$}
     \label{fig:decomposition example}
 \end{figure}
Each curve $c_1,c_2\subset \Sigma\setminus \PP$ is collapsed by $f$ to a node of $C$. Each $\hat{\Sigma}_i$ is a connected component of $\Sigma\setminus \{c_1,c_2\}$ and contains a subset of points $\hat{R}_i\subset \PP$. Denote $W_i=H_1(\hat{\Sigma}_i\setminus \hat{\PP}_i)$ and $\pi_{r}\in W_i$ the peripheral class around $r\subset \hat{\PP}_i$. With these notations the $(\Sigma,\PP)$-homological marking is represented by the following decorated tree:
\begin{figure}[h]
     \centering
\begin{center}
\begin{tikzpicture}[node distance = {20mm},thick, main/.style = {draw,circle}]

\node[main] (1) [label=below:$0$]{$W_1$};
\node[main] (2) [right of = 1][label=below:$g-1$]{$W_2$};
\node[main] (3) [right of= 2][label=below:$1$]{$W_3$};
\node (4) [above left of = 1]{$\pi_{r_1}$};
\node (5) [below left of =1]{$\pi_{r_2}$};
\node (6) [above right of =3]{$\pi_{r_6}$};
\node (7) [below right of =3]{$\pi_{r_7}$};
\node (8) [above of =2]{$\pi_{r_4}$};
\node (9) [above left of =2]{$\pi_{r_3}$};
\node (10) [above right of =2]{$\pi_{r_5}$};
\draw (3) -- (6);
\draw (1) -- (2);
\draw (2) -- (3);
\draw (1) -- (4);
\draw (1) -- (5);
\draw (3) -- (7);
\draw (2) -- (9);
\draw (2) -- (10);
\draw (2) -- (8);

\end{tikzpicture}
\end{center}
\caption{The decorated tree associated to the marking $f$ of Figure \ref{fig:decomposition example} where $W_i:=H_1(\hat{\Sigma}_i\setminus \hat{\PP}_i)$}
     \label{fig:labelling}
 \end{figure}
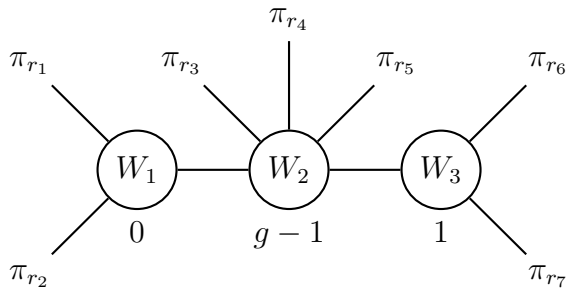
\end{example}

 Two homologically labelled graphs $G$ and $G'$ are isomorphic if there exists a homeomorphism of $\PP$-labelled graphs $\phi:G\rightarrow G'$ and a family of isomophisms $$\varphi_v:(W_v,<,>_v, \psi_v)\rightarrow (W'_{\phi(v)},<,>_{\phi(v)}, \psi_{\varphi_v}).$$ 
Two homologically \((\Sigma, \PP)\)-marked  $\PP$-labelled trees $(G,m)$ and $(G',m')$ are isomorphic if there exists an isomorphism $(\phi,\{\varphi_v\}_{v\in V(G)})$ of homologically marked $\PP$-labelled graphs between them such that in $H_1(\Sigma\setminus\PP)$ the following equations hold for each $v\in V(G)$
\begin{equation}\label{eq:same submodule}
m^{-1}(W_v)=(m')^{-1}(W'_{\phi(v)}) ,   
\end{equation} \begin{equation}\label{eq:compatibility for H_1 marking} m^{-1}(\varphi_v(\psi_v(l))=(m')^{-1}(\psi'_{\phi(v)}(\phi(l)))
\quad\forall l\in L(v)\text{ not belonging to an edge.}\end{equation} 

 The tree condition together with the fact that each module has one boundary relation, implies that condition \eqref{eq:compatibility for H_1 marking} is valid also for all $l\in L(v)$ and each $v\in V(A)$, whenever it is valid for the half legs that are not part of edges.
\begin{proposition}
\label{p:Torelli equivalence sesc systems}
Two curve systems $c$ and $c' $ on $\Sigma\setminus \PP$ formed by separating simple closed curves are pointed Torelli equivalent if and only if their induced $(\Sigma, \PP)$-homologically marked labeled trees are isomorphic. Moreover, any $(\Sigma,\PP)$-homologically marked labeled tree is isomorphic to one induced by some separating curve system $c$ in $\Sigma\setminus\PP$.
\end{proposition}

\begin{proof}
Suppose first that there exists an element in the pointed  Torelli group $h:(\Sigma,\PP)\rightarrow(\Sigma,\PP)$ satisfying $h(c)=c'$. Then \(h\) induces a natural isomorphism from the \(\PP\)-labelled graph \(G\) associated to  \(c\) and the one \(G'\) associated to \( c'\). We abusively denote by \( \phi : V(G) \rightarrow V(G') \) and \( \phi : E(G)\rightarrow E(G')\) the induced bijections. Remark that the separating normalization and the full normalization coincide because the systems are of separating curves. Denoting by \((\hat{\Sigma}, \hat{\PP}) \) and \((\hat{\Sigma}', \hat{\PP}')\) the two models of normalization associated to the respective systems of curves \(c\) and \(c'\), we observe that the homeomorphism \(h\) also induces a homeomorphism \( \hat{h} : (\hat{\Sigma}, \hat{\PP}) \rightarrow (\hat{\Sigma}, \hat{\PP}') \), and in particular isomorphisms \( \varphi_ v : H_1 (\hat{\Sigma}_v\setminus \hat{\PP} _v, \mathbb Z) \rightarrow H_1 (\hat{\Sigma}_{\phi(v)} ' \setminus \hat{\PP}'_{\phi(v)}, \mathbb Z)  \) that preserve intersection forms and satisfies \( \varphi _ v \circ \psi_v = \psi _{\phi(v)} \circ \phi_{L(v)}\). 
If we denote by \(\iota_v\) (resp. \(\iota '_v\)) the inclusion \( \hat{\Sigma}_v \setminus \hat{\PP}_v\rightarrow \Sigma\setminus \PP\) (resp. \( \hat{\Sigma}_v' \setminus \hat{\PP}_v' \rightarrow \Sigma\setminus \PP\)), these isomorphisms satisfy \( (\iota_v ') _* \circ \varphi_v = (i_v )_*  \) and so promotes the isomorphism \(\phi \) of \(\PP\)-labelled graphs associated to \(c\) and \(c'\) to an isomorphim of homologically marked \( (\Sigma, \PP)\)-labelled graphs. 



Conversely, assume that we have an isomorphism  between the homologically marked \((\Sigma, \PP)\)-labelled graphs associated to \(c\) and \(c'\)  respectively. This means that we have an isomorphism of graph \(\phi : G \rightarrow G'\) together with a family of isomorphisms \( \varphi_ v : H_1 (\hat{\Sigma}_v\setminus \hat{\PP} _v, \mathbb Z) \rightarrow H_1 (\hat{\Sigma}_{\phi(v)} ' \setminus \hat{\PP}'_{\phi(v)}, \mathbb Z)  \) that preserve intersection forms and satisfy \( \varphi _ v \circ \psi_v = \psi _{\Phi(v)} \circ \phi_{L(v)}\) for each \(v\in V(G)\). Recording that \(\phi\) induces a bijection of the set of components of \(c\) to the one of \(c'\) (since those sets identify with the set of vertices of \(G\) and \(G'\) respectively), one can choose 
some  homeomorphism  $k:(\Sigma,\PP)\rightarrow (\Sigma,\PP)$ such that $k(c_i)=c_i'$ that induces the isomorphism $\phi$ at the level of labelled trees.
If $k_*$ is the identity at the level of $H_1(\Sigma\setminus\PP)$ we are done. We claim that, up to precomposing $k$ by some homeomorphism $f:(\Sigma,\PP)\rightarrow (\Sigma,\PP)$ that induces an automorphism of the homologically marked $(\Sigma,\PP)$-labelled tree $G$ induced by  $c$ we can suppose that $h=k\circ f$ acts trivially on $H_1(\Sigma\setminus\PP)$.

In order to construct $f$ we will construct  separately homeomorphisms \(\hat{f} _v \) of \(\hat{\Sigma}_v\) that are the identity on a neighborhood of \(\PP_v\). Such a family induces naturally a homeomorphism \(f\) of \(\Sigma\) that is the identity on the neighborhood of the union of \(\PP\) and \(c\). First remark that we can use the family of isomorphisms  $\{\varphi_v\}$ between the homologies of the components to lift the action of $k_*$ to a family of automorphisms $\{M_v\}$ where $M_v:H_1(\hat{\Sigma}_v\setminus \hat{\PP} _v)\rightarrow H_1(\hat{\Sigma}_v\setminus \hat{\PP}_v)$. If  the family of homeomorphisms \(\{f_v\}_{v\in V(G)} \)  satisfies $(f_v)_*=M^{-1}_v$ for each \(v\in V(G)\), then the homeomorphism \(f\) satisfies $(k \circ f)_*=id$. The following lemma guarantees the existence of the $f_v$'s.


\begin{lemma}\label{l:image of mod in homology automorphisms} Let $\PP\subset \Sigma$ be a finite set in a compact orientable surface $\Sigma$. An automorphism $M:H_1(\Sigma\setminus\PP)\rightarrow H_1(\Sigma\setminus \PP)$ satisfies $M=f_*$ for some orientation preserving diffeomorphism  $f:(\Sigma,\PP)\rightarrow(\Sigma,\PP)$ which is the identity in a neighborhood of $\PP$ if and only if $M$ preserves the intersection form and fixes every peripheral class.
\end{lemma}

\begin{proof}
The kernel $K$ of the intersection form in $H_1(\Sigma\setminus \PP)$ is precisely the peripheral module $\Pi=\Pi_R$. The quotient by $K$ carries an intersection product and it is isomorphic to $H_1(\Sigma)$ endowed with its usual intersection product, a symplectic group. The map $M$ induces an automorphism $\overline{M}$ on $H_1(\Sigma)$, a symplectic group. It is known that $\overline{M}$ is induced by some diffeomorphism  $\overline{f}:\Sigma\rightarrow \Sigma$.  Up to applying an isotopy we can suppose that $\overline{f}$ fixes every point in $\PP$. The composition $M\circ (\overline{f}_*)^{-1}$ induces a homomorphism that fixes every peripheral class and induces the identity on $H_1(\Sigma)$. Therefore it can be written as $\text{Id}+\psi$ where $\psi:H_1(\Sigma\setminus \PP)\rightarrow H_1(\Sigma\setminus \PP)$ is a homomorphism with image in \(\Pi\) and whose kernel contains \(\Pi\). On the other hand $\text{Id}+\psi$ can be realized by using the pure braid group acting on the marked points. 


Indeed, let \(\{f_t \}_{t\in [0,1]}\) be an isotopy of \(\Sigma\) from the identity to a map \(f=f_1\) that fixes each point of \(\PP\). For each  \(p\in \PP\), denote by \( \beta_ p \) the homology class of the closed curve \(\{f_t (p)\}_{t\in [0, 1]}\). Then the action of \( f \) on \(H_1 (\Sigma \setminus \PP) \) is given by \(\text{id} +\psi \) with 
\[ \psi (\gamma ) = \sum _p (\beta_p \cdot \gamma) \pi_p  \text{ for every } \gamma \in H_1(\Sigma\setminus \PP).\] 
Each homeomorphism of \(\Sigma\) that fixes each point of \(\PP\) is isotopic among isotopies fixing \(\PP\) to a homeomorphism which is the identity on a neighborhood of \(\PP\) so we are done. 
\end{proof}
To prove the last statement of Proposition \ref{p:Torelli equivalence sesc systems} we remark that for each vertex $v$ of a $(\Sigma,\PP)$-marking of a labeled tree we choose an isomporphism between the module $W_v$ and the homology of a marked surface $(\Sigma_v,\PP_v)$ that preserves the intersection forms and the correspondence of half legs and marked points. The boundary points corresponding to edges are blown up to form boundary circles. Gluing these circles according to the information of the tree provides a surface with marked points homeomorphic to $(\Sigma,\PP)$ endowed with a family of separating simple closed curves $c$, each obtained from an edge of the graph. By construction the $(\Sigma,\PP)$-homologically marked labeled tree associated to $c$ is isomorphic  to the initially given tree. 
\end{proof}

We finish this section with a useful tool enabling to lift points in  \(\overline{\mathcal M}^{sep}_{g,n}\) to elements of the Torelli space \(\mnc^{sep}_{\Sigma, \PP}\). 

Given \( (C,Q,[f])\in \mrs _{\Sigma, \PP}\), the marking \(f\) induces an isomorphism \( f_* : H_1 (\Sigma\setminus\PP) \rightarrow H_1(C\setminus\PP)\) that preserves the intersection form and maps peripheral class \(\pi_r\in H_1( \Sigma\setminus \PP)\) around a point \(r\in\PP\) to the peripheral class in \(H_1(C\setminus Q)\) around the corresponding point \(q_r\).

\begin{corollary} \label{c: marking homologically smooth compact curves}
    Given any point \((C,Q)\in \mathcal M_{g,n}\), and a linear isomorphism \( m: H_1 (\Sigma\setminus\PP) \rightarrow H_1(C\setminus Q)\) that maps the intersection form on \(H_1 (\Sigma\setminus\PP)\) to the one on \(H_1(C\setminus Q)\), and  the peripheral class \(\pi_r\in H_1( \Sigma\setminus \PP)\) around \(r\in\PP\) to the peripheral class \(\pi_{q_r}\in H_1(C\setminus Q)\) around the corresponding element \(q_r\in Q\), there exists a unique element \((C, Q, [f]) \in \mrs_{\Sigma, \PP}\) such that \(m = f_{*}\). \end{corollary}
   
   \begin{definition}
  A point in $\mrs_{\Sigma,\PP}$ will sometimes be denoted by the triple $(C,Q,m)$ where $(C,Q)$ is a pointed Riemann surface of the same type as $(\Sigma,\PP)$ and  \( m: H_1 (\Sigma\setminus\PP) \rightarrow H_1(C\setminus Q)\) is an isomorphism as in Corollary \ref{c: marking homologically compact type curves} (the correspondence between the points of $\PP$ and $Q$ is given by the isomorphism $m$).
   \end{definition}
   For boundary points in $\overline{\mathcal{M}}^{sep}_{g,n}$ the analogous statement reads
\begin{corollary}\label{c: marking homologically compact type curves}
    Given a point  \((C,Q)\in \overline{\mathcal M}^{sep}_{g,n}\)  and a \((\Sigma,\PP)\)-homological marking $m$ of its homologically labelled $\PP$-labelled tree $G$ (as in \eqref{eq:homological marking}), there exists a unique $(C,Q,[f])\in\mnc^{sep}_{\Sigma,\PP}$ such that the homological $(\Sigma,\PP)$-marking induced by the collapse map $f:(\Sigma,\PP)\rightarrow (C,Q)$ is isomorphic to $(G,m)$.
\end{corollary}
\begin{proof}
 To prove existence of the collapse map $f$, use the last statement of Proposition \ref{p:Torelli equivalence sesc systems} to produce a separating curve system $c\subset\Sigma\setminus\PP$ that induces a \((\Sigma,\PP)\)-homologically marked tree $(G',m')$ that is isomorphic to the given $(G,m)$ for $(C,Q)$. A choice of isomorphism between $(G,m)$ and $(G',m')$ can be realized by a topological marking in each vertex separately by using Corollary \ref{c: marking homologically smooth compact curves}. Moreover, the union of the markings of the parts induces a well defined collapse map (that collapses precisely $c$). 
 Given two collapse maps $f,f':(\Sigma,\PP)\rightarrow (C,Q)$ inducing isomorphic $(\Sigma,\PP)$-homologically marked trees, Proposition \ref{p:Torelli equivalence sesc systems} proves that they are equivalent by an element of the pointed Torelli group $\mathcal{I}_{\Sigma,\PP}$, thus proving that they define the same point in $\mnc^{sep}_{\Sigma,\PP}$. 
\end{proof}
A point in $\mnc^{sep}_{\Sigma,\PP}$ will be defined by simply presenting the $(\Sigma,\PP)$-homologically marked labeled tree of the $\PP$-labeled tree associated to the underlying curve that has in each vertex a submodule of $H_1(\Sigma\setminus \PP)$.






\subsection{Spherical boundary components and modules} \label{ss: E-modules}
In this section we introduce a family of boundary strata of $\mnc^{sep}_{\Sigma,\PP}$ that are especially useful to prove the inductive step in Theorems \ref{t:Rgn} and \ref{t:Cgn}. They are associated to separating simple closed curves in $\Sigma\setminus \PP$ that separate the surface in part of genus zero with at least two marked points and another of genus $g$ with the rest of marked points. We will call such a boundary stratum a spherical boundary stratum. Let us analyze the $(\Sigma,\PP)$-marking associated to such strata. 

Given a subset $E\subset \PP$ we denote $\Pi_E\subset H_1(\Sigma\setminus \PP)$  the submodule generated by the peripherals $\{\pi_e : e\in E\}$ around points of $E$. We have 
\begin{equation}
\label{eq:peripheral ex seq}
0\rightarrow\Pi_E\rightarrow H_1(\Sigma\setminus\PP) 
\rightarrow \frac{H_1(\Sigma\setminus\PP)}{\Pi_E}\cong H_1(\Sigma\setminus(\PP\setminus E))\rightarrow 0 .
\end{equation}

\begin{definition}\label{d: E-module} An \( E\)-module is a submodule \( M \subset H_ 1 (\Sigma\setminus\PP, \mathbb Z) \) such that \begin{equation}\label{eq:Emodule} H_1 (\Sigma\setminus\PP, \mathbb Z) = M+ \Pi _E\text{ and }M \cap \Pi_E = \mathbb Z \pi_E =\mathbb Z \pi_{P\setminus E}\text{ where }\pi_E := \sum_{e\in E} \pi_e.\end{equation} 
\end{definition}
Whenever we take a simple closed curve $c$ in \(\Sigma\setminus\PP\) that separates the surface in a disc marked by the points of \(E\), and a genus \(g\)-domain marked with the points of \( \PP \setminus E\): the homology of this latter defines a \(E\)-module $M_c\subset H_1 (\Sigma\setminus\PP, \mathbb Z)$.  By Proposition \ref{p:Torelli equivalence sesc systems} applied to this special case we deduce that two separating simple closed curves $c,c'$ in $\Sigma\setminus\PP$ defning the $E$-modules $M_c$ and $M_{c'}$ respectively are pointed Torelli equivalent if and only if $M_c=M_{c'}$. In other words, we can associate the module $M_c$ to the  vertex associated to the class of $c$ in  $\mathcal{C}^{sep}(\Sigma, \PP)/\mathcal{I}_{\Sigma,\PP}$, and talk about the boundary stratum $\mathcal{B}_{M_c}=\mathcal{B}_{c}$ of  $\mnc^{sep}_{\Sigma,\PP}$  corresponding to the module $M_c$. In fact, by Corollary \ref{c: marking homologically compact type curves}, any $M$ satisfying \eqref{eq:Emodule} can be written as $M=M_c$ for some simple closed curve $c$ that separates the surface in a component of genus zero (the spherical component) containing the points of $E$ and another component of genus $g$ with the rest of points.

\begin{definition}
Given a subset $E\subset \PP$, with at least two points and  an $E$-module $M\subset H_1(\Sigma\setminus\PP)$ we define $\mathcal{B}_M\subset \mnc^{sep}_{\Sigma,\PP}$ as the spherical boundary stratum formed by marked stable curves that have $M$ and $\Pi_E$ as the only factors in their associated $(\Sigma,\PP)$-homologically marked labelled tree. We will sometimes refer to $\mathcal{B}_M$ as an
    $E$-boundary component.
\end{definition}

For a $E$-boudary     component characterized by  an $E$-module \(M\subset H_1(\Sigma\setminus\PP)\) the associated diagram has two vertices one of which has genus zero, $|E|$ legs and $\Pi_E$ as the module marking the part,  and the other has $|\PP\setminus E|$ legs and genus $g$ with homology marked by $M$.  See a case with $E=\{r_1,r_2,r_3\}$ and $n=5$ with the notations introduced: 
\begin{center}
\begin{tikzpicture}[node distance = {20mm},thick, main/.style = {draw,circle}]
\node[main] (1) [label=below:$0$]{$\Pi_E$};
\node[main] (2) [right of = 1][label=below:$g$]{$M$};
\node (3) [above left of = 1]{$\pi_{r_1}$};
\node (4) [below left of =1]{$\pi_{r_3}$};
\node (5) [left of =1]{$\pi_{r_2}$};
\node (6) [above right of =2]{$\pi_{r_4}$};
\node (7) [below right of =2]{$\pi_{r_5}$};

\draw (2) -- (7);
\draw (2) -- (6);
\draw (1) -- (2);
\draw (1) -- (3);
\draw (1) -- (4);
\draw (1) -- (5);

\end{tikzpicture}
\end{center}
\begin{definition} The quotient \(E\)-module $\overline{M}$ associated to an $E$-module \(M\) is by definition its quotient by \( \mathbb Z \pi_E=\mathbb Z \pi_{\PP\setminus E}\), it fits in \(H_1 (\Sigma\setminus\PP) /\mathbb Z \pi_E\)
and is naturally isomorphic to \( H_1 (\Sigma\setminus(\PP\setminus E))\).
\end{definition}

One can retrieve the \(E\)-module \(M\) from its associated quotient \(E\)-module \(\overline{M}\): \(M\) is the preimage of \(\overline{M}\) by the quotient map \( H_1 (\Sigma\setminus\PP)\rightarrow H_1 (\Sigma\setminus\PP)/ \mathbb Z \pi _E\). 

Thanks to the exact sequence \eqref{eq:peripheral ex seq}, the set of \(E\)-modules (or equivalently of quotient \(E\)-modules) has the structure of an affine space directed by the module 
\[ \text{Hom} \big( H_1 (\Sigma\setminus(\PP\setminus E));\frac{ \Pi_E}{ \mathbb Z \pi _E}\big) \]
\section{Isoperiodic foliations}
\subsection{The Hodge bundle}
 Let us denote by $$\Omega\overline{\mathcal{M}}_{g,n}\rightarrow\overline{\mathcal{M}}_{g,n}$$ the bundle whose fiber over a point $(C,Q)$ is the vector space of meromorphic stable forms on the marked stable curve $(C,Q)$ having at worst simple poles on the $n$ distinct ordered points of $Q$. In other words the space of meromorphic sections of the twisted line bundle $K_C(-(q_1+\ldots+q_{n}))$. A point in $\Omega\overline{\mathcal{M}}_{g,n}$ will be denoted by $(C,Q,\omega)$. When $(\omega)_{\infty}=q_1+\ldots+q_{n_1}$ we will omit $Q$ and write  $(C,\omega)$.

\subsection{The period map}
The pull back of the bundle $\Omega\overline{\mathcal{M}}_{g,n}$  by the (branched) cover $\overline{\mathcal{S}}_{\Sigma,\PP}\rightarrow \overline{\mathcal{M}}_{g,n}$ will be denoted  $$\Omega\overline{\mathcal{S}}_{\Sigma,\PP}\rightarrow\overline{\mathcal{S}}_{\Sigma,\PP} .$$
 A point in $\Omega\overline{\mathcal{S}}_{\Sigma,\PP}$ will be denoted by $(C,Q,[f],\omega)$ where $[f]$ denotes the pointed Torelli class of the marking $f$. 
 
 \begin{definition}
Given a subset $\mathcal{K}$ of $\overline{\mathcal{S}}_{\Sigma,\PP}$ or any of its  quotients, the set $\Omega^{*}\mathcal{K}$ denotes the subset of forms of $\Omega\mathcal{K}$ that have isolated zeros (i.e. no zero components) and $\Omega_{0}^{*}\mathcal{K}$ the subset of  $\Omega^{*}\mathcal{K}$ having zero residues at non-separating nodes.  
\end{definition}
 
 We want to describe the set of forms in $\Omega\overline{\mathcal{S}}_{\Sigma,\PP}$ for which all integrals on classes of $H_1(\Sigma\setminus\PP)$ are well defined complex numbers that coincide with those of some form on a \textit{smooth} curve. A problem that arises is that stable forms can have non-zero residues at nodes, and this does not allow to integrate any path passing through the node. We will show that this problem can be by-passed by taking representatives in homology classes that avoid the separating nodes, but when the node is non-separating there are classes in $H_{1}(\Sigma\setminus\PP)$ that do not admit a representative that does not intersect the node. For instance, take the class of a closed curve with basepoint at the node, starting on one branch and ending at the other. We therefore restrict our attention to a union of strata where we can integrate on some representative in every class of $H_1(\Sigma\setminus\PP)$, namely $$\Omega_0\overline{\mathcal{S}}_{\Sigma,\PP}=\{(C,Q,[f],\omega)\in\Omega\overline{\mathcal{S}}_{\Sigma,\PP}: \text{Res}_q(\omega)=0\quad\forall \text{ non-separating node } q \text{ of } C\}.$$
Unfortunately this set is neither open (its interior is formed by stable forms on curves having only separating nodes) nor closed (it is dense) in $\Omega\overline{\mathcal{S}}_{\Sigma,\PP}$.  Let us prove that we can extend the period map to it. 

Remark that if $c$ is a system of curves, each of whose elements is a \emph{non-separating curve} in $\Sigma\setminus \PP$ we can extend the definition of the map $\per_{\Sigma,\PP}$ to $\Omega_0\mathcal{B}_{[c]}$ with exactly the same definition: 
$$H_1(\Sigma\setminus\PP)\ni[\gamma]\mapsto\int_{f(\gamma)}\omega=:\percompact(C,Q,[f],\omega)([\gamma])\in\mathbb{C}
$$
When $c\subset \Sigma\setminus \PP$ is a curve system that contains separating curves, write $c=c^{s}\sqcup c^{ns} $   where $c^{s}$ correspond to the separating curves of $c$,  and let $(G,m)$ be the $(\Sigma,\PP)$- marking of a labelled tree corresponding to any element in the boundary stratum $\mathcal{B}_{[c]}\subset \mnc_{\Sigma,\PP}$ corresponding to the Torelli class of $c$.   The isomorphism of Proposition \ref{p:partial norm torelli} allows to define a separating normalization map at the level of stable forms by restriction of the forms to the components
\begin{equation}\label{eq:normaliza separantes}\Omega_0\mathcal B_{[c]}\rightarrow\prod_{v\in V(G)}\Omega_0\mathcal{B}_{[c^{ns}_v]}\subset\prod_{v\in V(G)}\Omega_0\overline{\mathcal{S}}_{\hat{\Sigma}^{s}_v,\hat{\PP}^s_v}
    \end{equation}
 The map \eqref{eq:normaliza separantes} is not surjective (the zero residue sum condition at nodes imposes a residue relation at pairs of poles of different components of the separating normalization that correspond to a node).

   On the $v$'th factor of the target space, the associated system of curves has only non-separating curves, and we have a well defined period map $\percompact_{\hat{\Sigma}_v^s,\hat{\PP}^s_v}$ with image in  $\text{Hom}(H_1(\hat{\Sigma}^s_v\setminus\hat{\PP}^s_v),\mathbb{C})$. Consider the map \begin{equation}
     \label{eq:sum periods of parts}   
\sum_{v\in V(G)}\percompact_{\hat{\Sigma}_v^s,\hat{R}^s_v}:\prod_{v\in V(G)}\Omega_0\mathcal{B}_{[c^{ns}_v]}\rightarrow\text{Hom}\left(\bigoplus_{v\in V(G)} H_1(\hat{\Sigma}^s_v\setminus\hat{\PP}^s_v),\mathbb{C}\right).\end{equation}

    We claim that the composition of the map \eqref{eq:normaliza separantes} with the map \eqref{eq:sum periods of parts} induces an element in $\text{Hom}(H_1(\Sigma\setminus\PP),\mathbb{C}$). 
    Indeed, if $(C,Q,[f],\omega)\in\Omega_0\mathcal{B}_{[c]}$ we know that at each separating node, the sum of the residues of the two branches is zero. Denote by $(\hat{C}^s_v,\hat{Q}^s_v,[\hat{f}^s_v],\hat{\omega}^s_v)$ the $v$'th component of the image of that point by the map \eqref{eq:normaliza separantes} in $\Omega_0\mathcal{B}_{[c^{ns}_v]}$ and $$p_v=\percompact_{\hat{\Sigma}^s_v,\hat{\PP}^s_v}(\hat{C}^s_v,\hat{Q}^s_v,[\hat{f}^s_v],\hat{\omega}^s_v).$$  To each separating curve $c^s_j\in c^s$ there corresponds a pair of points $n_{j1}$ and $n_{j2}$  in $\hat{R}^s\subset \hat{\Sigma}^s$ lying in different connected components of $\hat{\Sigma}^s$. Let $v_{j1}$ (resp. $v_{j2}$) be such that  $n_{j1}\in \hat{\Sigma}^s_{v_{j1}}$ (resp. $n_{j2}\in \hat{\Sigma}^s_{v_{j2}}$). The residue condition can be read as follows: for the two peripherals $\pi_{n_{j1}}\in H_1(\hat{\Sigma}^s_{v_{j1}}\setminus \hat{R}^s_{v_{j1}})$ and $\pi_{n_{j2}}\in H_1(\hat{\Sigma}^s_{v_{j2}}\setminus \hat{R}^s_{v_{j2}})$ we have 
    $$(\sum_{v\in V(G)}p_v)(\pi_{n_{j1}}+\pi_{n_{j2}})=p_{v_{j1}}(\pi_{n_{j1}})+p_{v_{j2}}(\pi_{n_{j2}})=\int_{\hat{f}^s_{v_{j1}*}(\pi_{n_{j1}})}\hat{\omega}^s_{v_{j1}}+\int_{\hat{f}^s_{v_{j2}*}(\pi_{n_{j2}})}\hat{\omega}^s_{v_{j2}}=0$$
    Therefore, supposing there are $k\geq 1$ separating curves in $c$, the map $\sum_v p_v$ induces a homomorphism on $$\frac{\oplus H_1(\hat{\Sigma}^s_v\setminus\hat{R}^s_v)}{\{\pi_{n_{j1}}+\pi_{n_{j2}}=0: j=1,\ldots,k\}}\rightarrow\mathbb{C}$$
   Its composition with the isomorphism $m$ determined by the $(\Sigma,\PP)$ -marking (defined in \eqref{eq:homology decomposisiton separating syst} and which is independent of the point of $\mathcal{B}_{[c]}$) gives the desired period homomorphism $$p=:\percompact_{\Sigma,\PP}(C,Q,[f],\omega): H_1(\Sigma\setminus \PP)\rightarrow \mathbb{C}$$ that satisfies $p([\gamma])=\int_{f_{*}([\gamma])}\omega$ for all closed paths $\gamma$ in $\Sigma\setminus \{c^s,\PP\}$.

\begin{definition}
The period map on $\Omega_0\overline{\mathcal{S}}_{\Sigma,\PP}$ is the map $$\percompact_{\Sigma,\PP}:\Omega_0\overline{\mathcal{S}}_{\Sigma,\PP}\rightarrow \text{Hom}(H_1(\Sigma\setminus\PP);\mathbb{C})$$
sending each point to its period homomorphism. 
\end{definition}

The period map $\percompact_{\Sigma,\PP}$ is globally continuous, and holomorphic in restriction to any stratum of the boundary stratification of $\Omega_0\overline{\mathcal{S}}_{\Sigma,\PP}$.

In the sequel, we will also use the following notation.

\begin{definition} Given  $p\in\text{Hom}(H_1(\Sigma\setminus\PP);\mathbb{C})$ and a subset $\mathcal{K}\subset \Omega_0\overline{\mathcal{S}}_{\Sigma,\PP}$ we define $$\mathcal{K}(p)=\{(C,Q, m,\omega)\in\mathcal{K}: \percompact_{\Sigma,\PP}(C,Q, m,\omega)=p\},$$ i.e. the set of elements in $\mathcal{K}$ having periods $p$. 
\end{definition}

\subsection{Local description of isoperiodic sets}\label{ss:local description} We recall two results proven in the Appendix of \cite{CD} that describe the local properties of a fiber of the period map at a point with no zero components: 

\begin{theorem}\label{t:local structure}
The local fiber of the period map in $\Omega_{0}\overline{\mathcal{S}}_{\Sigma,\PP}$  at a point $(C,Q,[f],\omega)\in \Omega^{*}_0\overline{\mathcal{S}}_{\Sigma,\PP}$ with no zero components and zero residues at non-separating nodes, projects to the orbifold chart of $\Omega\overline{\mathcal{M}}_{g,n}$ as a complex manifold transverse to every (smooth) component of the boundary divisor through the point. Therefore, it is  an abelian ramified cover of a normal crossing divisor in $(\mathbb{C}^{2g+n-3},0)$ having precisely one component of codimension one in each boundary component of codimension one of the ambient space through the point. \end{theorem}

\begin{corollary}\label{c:closure}
Let $(C,Q,m,\omega)\in\Omega^{*}_0\mrs_{\Sigma,\PP}$ be a form of period $p\in\text{Hom}(H_1(\Sigma\setminus\PP);\mathbb{C})$. Then, the closure of $\Omega\mrs_{\Sigma,\PP}(p)$ in $\Omega^{*}_{0}\mnc_{\Sigma,\PP}$ (resp. $\Omega^{*}_{0}\mnc^{c}_{\Sigma,\PP})$ is the set $\Omega^{*}_{0}\mnc_{\Sigma,\PP}(p)$ (resp. $\Omega^{*}_{0}\mnc^{c}_{\Sigma,\PP}(p)$). The sets $\Omega\mrs_{\Sigma,\PP}(p)$ and $\Omega^{*}_{0}\mnc_{\Sigma,\PP}(p)$ (resp. $\Omega^{*}_{0}\mnc^{c}_{\Sigma,\PP}(p)$) have the same number of connected components
\end{corollary}
\begin{proof}
Apply Theorem \ref{t:local structure} to any point in the closure to prove that the sets coincide. The local structure of abelian ramified cover of normal crossing divisor implies that the boundary in $\Omega^{*}_{0}\mnc_{\Sigma,\PP}(p)$ do not locally separate $\Omega^{*}_{0}\mnc_{\Sigma,\PP}(p)$. Since $\Omega\mrs_{\Sigma,\PP}(p)$ is the complement of the boundary points, we conclude. The same argument applies for $\Omega^{*}_{0}\mnc^{c}_{\Sigma,\PP}(p)$.
\end{proof}

\begin{definition}
Two points $(C,Q,m,\omega)$ and $(C',Q',m',\omega')$ in $\Omega\mrs_{\Sigma,\PP}$ (resp. $\Omega\mnc_{\Sigma,\PP}$) are said to be isoperiodically equivalent and denoted $(C,Q,m,\omega)\sim (C',Q',m',\omega')$ if there exists an isoperiodic path joining them in $\Omega\mrs_{\Sigma,\PP}$ (resp. $\Omega\mnc_{\Sigma,\PP}$)
 \end{definition}

 Corollary \ref{c:closure} implies that the two equivalence relations restricted to  $\Omega\mrs_{\Sigma,\PP}$ coincide. 
 
As part of the tools for its proof, and some of the results in this paper, we recall the definition of the period coordinates on strata of meromorphic differentials with simple poles. 

The space $\Omega\mathcal{M}_{g,n}$ is stratified by the properties of the zeros and poles of the forms. Apart from the stratum of zero forms, each stratum is determined by a partition the number $2g-2-l$, where $0\leq l\leq n$ is the number of simple poles, into a sum $n_1+\ldots+n_k$ of positive integers. The stratum corresponding to forms with $n$ simple poles and the partition $(n_1,\ldots,n_k)$, will be denoted by $\Omega\mathcal{M}_{g,n}(n_1,\ldots,n_k)$. It is a complex suborbifold.
The stratum with the maximal number $n$ of poles and the maximal number of zeros i.e. each $n_i=1$ and $k=2g-2+n$, is open and called the generic stratum. The stratum with a single zero and $l$ poles, $\Omega\mathcal{M}_{g,n}(2g-2+n)$, is called the minimal stratum. 

Around a point $(C,Q,\omega)\in\Omega\mathcal{M}_{g,n}(n_1,\ldots,n_k)$ with $n$ simple poles, the stratum  is locally parametrized by the complex vector space of dimension $2g+(n-1)+(k-1)$ \begin{equation}\label{eq:period coord on strata}\text{Hom}(H_1(C\setminus Q, Z;\mathbb{Z}),\mathbb{C})\end{equation} (see \cite{boissy}) where $Z\subset C$ is the set of zeros of $\omega$. The construction of those coordinates around $(C,Q,\omega)$ is carried by using the Gauss-Manin connection to identify the homology group of the central curve $(C,Q)$ with the homology group of a nearby point $(C',Q')$, and apply integration of $\omega'$ on any relative homology class to define the associated (period) homomorphism that corresponds to the coordinates of the point. In these coordinates the restriction of the period map to the stratum is just the linear projection 
$$\text{Hom}(H_1(C\setminus Q, Z;\mathbb{Z}),\mathbb{C})\rightarrow \text{Hom}(H_1(C\setminus Q;\mathbb{Z}),\mathbb{C}).$$
which is submersive with fiber of dimension $k-1$. This local model can be lifted to any covering of $\Omega\mathcal{M}_{g,n}$ and implies, in particular,  that the map $\per_{\Sigma,\PP}$ is a submersion around $(C,Q,[f],\omega)\in\Omega^*\mrs_{g,n}$ with fibers transverse to every non-generic stratum. In particular the intersection of a fiber of $\per_{\Sigma,\PP}$ with the minimal stratum is an isolated set of points. 

A similar stratification for $\Omega\overline{\mathcal{M}}_{g,n}$-- that substratifies the boundary stratification-- can be defined by using the stratification on each of the parts of the normalization of a given point. It is used to prove the local properties of the isoperiodic foliation on boundary points presented in Theorem \ref{t:local structure}, but we will not use it in the sequel. 

The relative period coordinates around points in the generic stratum give us coordinates in the ambient space $\Omega\mathcal{M}_{g,n}$ and allow us to locally parametrize isoperiodic sets, but for points in non-generic strata it will be convenient to have local parametrizations of isoperiodic sets. This can be done by using the interpretation of forms as branched translation structures, and  Hurwitz spaces of coverings on discs. The basic idea is to deform the atlas of the branched structure locally around branch points. 

Choose $(C,Q,\omega)\in \Omega^*\mathcal{M}_{g,n}(n_1,\ldots, n_k)$ and denote $(L,(C,Q,\omega))$ the germ of isoperiodic set (leaf) through $(C,Q,\omega)$. Consider the local $n_i+1$-branched cover $\phi_i:(C,z_i)\rightarrow(\mathbb{C},0)$ defined by integration of $\omega$ around the zero $z_i\in C\setminus Q$ of order $n_i$.    As shown in the appendix in \cite{bps}  we can construct a  germ of continuous  map \begin{equation}
    \label{eq:hurwitz parametrization o isoperiodic}
\big(H(\phi_1)\times\cdots\times H(\phi_{d_k}),(\phi_1,\ldots,\phi_k)\big)\rightarrow (L,(C,Q,\omega))
\end{equation}
where each $H(\phi_i)$ is a Hurwitz space of coverings of degree $d_i$ over a disc $\phi_i(D_i)$ up to an equivalence in the boundary that allows to glue each element with the translation surface defined by $\omega$ on $C\setminus (\cup D_i)$ to obtain a branched translation structure on a closed marked surface of genus $g$.  Since we can keep the marked points and a marking of the homology by avoiding the discs $D_i$, the gluing preserves all integrals over cycles of $H_1(C\setminus Q;\mathbb{Z})$ so the image of any point is in $L$. 

By \cite{bps}[Lemma A7, p. 439], each $H(\phi_i)$ can be parametrized by the space of polynomials $R_a(z)=z ^{d_i}+a_{d_i-1}z^{d_i-1}+\ldots+a_0$ with $\sum a_j=0$ having critical values in the unit disc. The point with all coordinates $a_j=0$ corresponds to the initial point. Any choice on each Hurwitz space of a point that has $n_i$ distinct critical values $\{v_1,\ldots,v_{n_i}\}$ (therefore determining simple critical points on some set $Z_i$) has as image in $L$ a point with simple zeros.  On the other hand, the set of critical values determines the values of the $a_i$'s in the parameter space.

Using the last idea, we can relate some particular isoperiodic deformations with particular deformations of branched coverings constructed by a cut and paste procedure called Schiffer variations.

The Schiffer variation is a local continuous  deformation of branched translation structures that preserves the holonomy. We then interpret it as a deformation of homologically marked stable form. We refer the reader to Section 5.1 of the paper \cite{CDF2} for the related history and definitions. It allows to define isoperiodic paths in $\Omega\overline{\mathcal{M}}_{g,n}$ (and its covers) by deforming any stable meromorphic form having an isolated zero or a node continuously and isoperiodically. It is defined by a continuous surgery for any given family of twin paths, i.e. distinct paths $\gamma_1,\gamma_2, \ldots,\gamma_k$ with parameter space $[0,1)$  of a stable meromorphic form $(C,\omega)$, that are disjoint apart from their initial value, where they coincide, and such that the $k$ paths in $\mathbb{C}$ defined by integration \begin{equation}\label{eq:integral} t\mapsto\int_{\gamma_{i|[0,t]}}\omega\quad\text{ are non-constant and coincide for }\quad i=1,\ldots,k.\end{equation}
Condition \eqref{eq:integral} imposes that the starting point is either an isolated zero of $\omega$ or a node of zero residue and the paths $\gamma_i$ do not lie in a zero component of the form $\omega$. The associated local deformation will be considered as a continuous deformation of the underlying  branched translation structure continuously in terms of the variable $t$ by doing the following surgery for each fixed time $t\in [0,1)$: order $\gamma_1,\ldots,\gamma_k$ in the cyclic ordering at the zero. Slit each oriented $\gamma_{i}([0,t])$ and glue the left side of the slit of $\gamma_{i}([0,t])$ with the right side of the slit of $\gamma_{i+1}([0,t])$ (where we write $\gamma_{k+1}=\gamma_1$) by respecting the values of the function given in \eqref{eq:integral}. This reorganizes the angle sectors at the different endpoints that appear after the surgery. It also produces a family of paths in the surface obtained after the surgery that. indeed, by construction, all points $\gamma_i(t)$ for $i=1,\ldots, k$ form a unique point in the obtained (possibly nodal) surface. When we invert the orientation of the obtained paths they form twin paths (possibly starting at a singular point). Applying the Schiffer variation to these twin paths recovers the original flat surface. In Figures \ref{fig:schiffer_regular}, \ref{fig:schiffer_one_closed_twin}, \ref{fig:schiffer_two_closed_twins}, \ref{fig:non_sep_shiffer}, \ref{fig:schiffer_twins_2_2_poles} and  \ref{fig:twins_2_one_pole} there are several examples of such construction with two twins. 

The Schiffer variation applied to families of small parallel geoedsic segments of the same length starting at a zero or a node allow to continuously deform a form by preserving the holonomy (or the period homomorphism for branched translation structures) and some further properties. Some of the possibilities are:
\begin{itemize}
\item  change the position of the branching value locally to any point without changing its order (by choosing the full set of twins at branch point) 
\item split a branch point into several branch points in any combination that is compatible with the total angle condition and change the position of the branching values locally. (by choosing subfamilies of twins at the branch point)
\item smoothing a node of zero residue but without zero components of a stable form (by choosing two small twins that lie in the two different branches of the node)
\item smoothing a node of non-zero residue (inverse operation of the one described in Figure \ref{fig:schiffer_twins_2_2_poles})
\end{itemize}

Under the light of the properties of the  map \eqref{eq:hurwitz parametrization o isoperiodic}, this implies that all the neighbouring points in the isoperiodic set $(L,(C,Q,\omega))$ can be attained by Schiffer variations regardless of the stratum of $(C,Q,\omega)$ in $\Omega^*\mathcal{M}_{g,n}$. Morover, we can also deform locally from one boundary stratum of $\Omega^*_0 
\overline{\mathcal{M}}_{g,n}$ to a neighbouring stratum by smoothing nodes isoperiodically. 

One advantage of Schiffer variations is that they allow to define not only local deformations, but also long isoperiodic paths that change the number of zeros, and also allow to reach stable forms with some nodes. Let us analyze some examples of those phenomena that will be useful in the sequel. 

In Figures \ref{fig:schiffer_regular}, \ref{fig:schiffer_one_closed_twin}, \ref{fig:schiffer_two_closed_twins} and \ref{fig:non_sep_shiffer} we develop the classical case (see \cite{CDF2}) of two twins whose endpoints are not at poles. Remark that, when the endpoints of some distinct $\gamma_i$'s coincide there might be two discs formed by sectors that are glued at a point. In this case we obtain a stable form with a node of zero residue.
\begin{figure}[h]
     \centering
     \includegraphics[height=5cm]{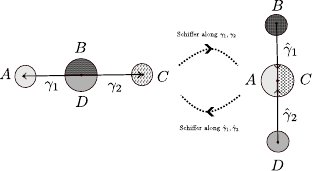}
     \caption{Schiffer variations along two paths with distinct endpoints at saddle points. If $\{B,D\}=\{A,C\}$ it joins two forms in the same stratum. If $A,C>2\pi$ and $B$ or $D$ are $2\pi$ the resulting form has one saddle less.}
     \label{fig:schiffer_regular}
 \end{figure}
 \begin{figure}[h]
     \centering
     \includegraphics[height=5cm]{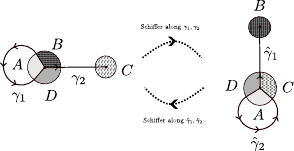}
     \caption{Schiffer variations along two twin paths, one of which is closed. If $B=C$ it joins two forms in the same stratum}
     \label{fig:schiffer_one_closed_twin}
 \end{figure}
 \begin{figure}[h]
     \centering
     \includegraphics[height=3cm]{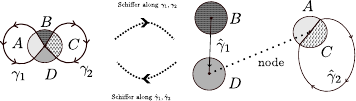}
     \caption{The Schiffer variation along two twin closed paths of zero intersection at a regular point leads to a pair of twins at a node with zero residue. One of the twins is closed and starts and ends at the same branch of the node}
     \label{fig:schiffer_two_closed_twins}
 \end{figure}

 \begin{figure}[h]
     \centering
 \includegraphics[height=4cm]{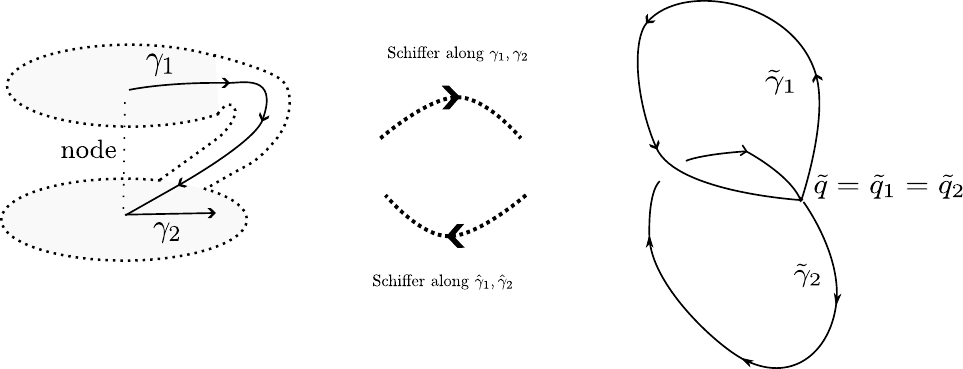}
     \caption{The Schiffer variation along two twin closed paths of $\pm 1$ intersection at a regular point leads to a pair of twins at a node with zero residue. One of the twins is closed and starts and ends at different branches of the node.}
     \label{fig:non_sep_shiffer}
 \end{figure}
\begin{figure}[h]
     \centering
     \includegraphics[height=7cm]{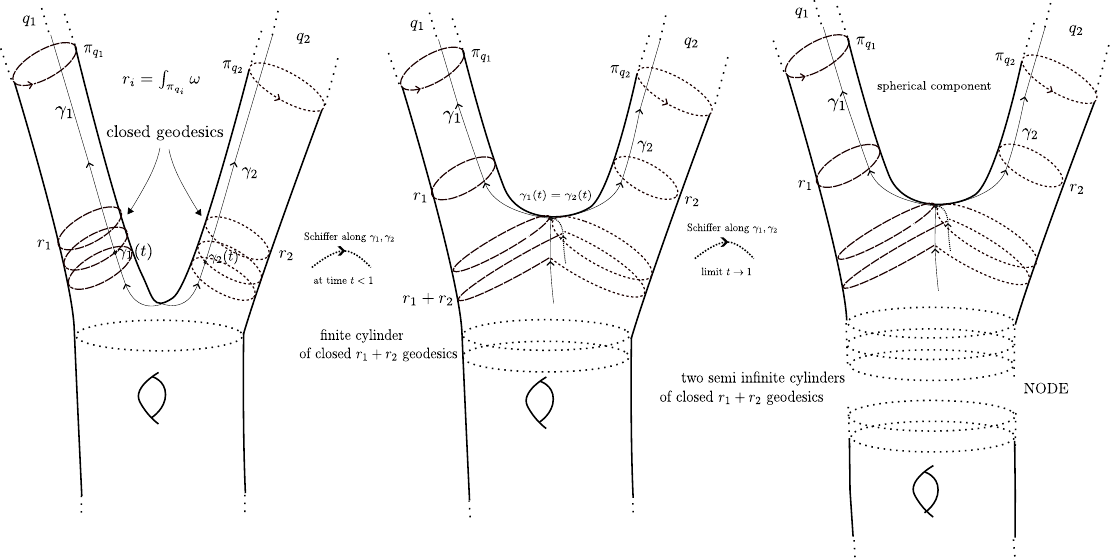}
     \caption{Schiffer variations along pairs of twins leading to two distinct poles}
     \label{fig:schiffer_twins_2_2_poles}
 \end{figure}
\begin{figure}[h]
     \centering
     \includegraphics[height=7cm]{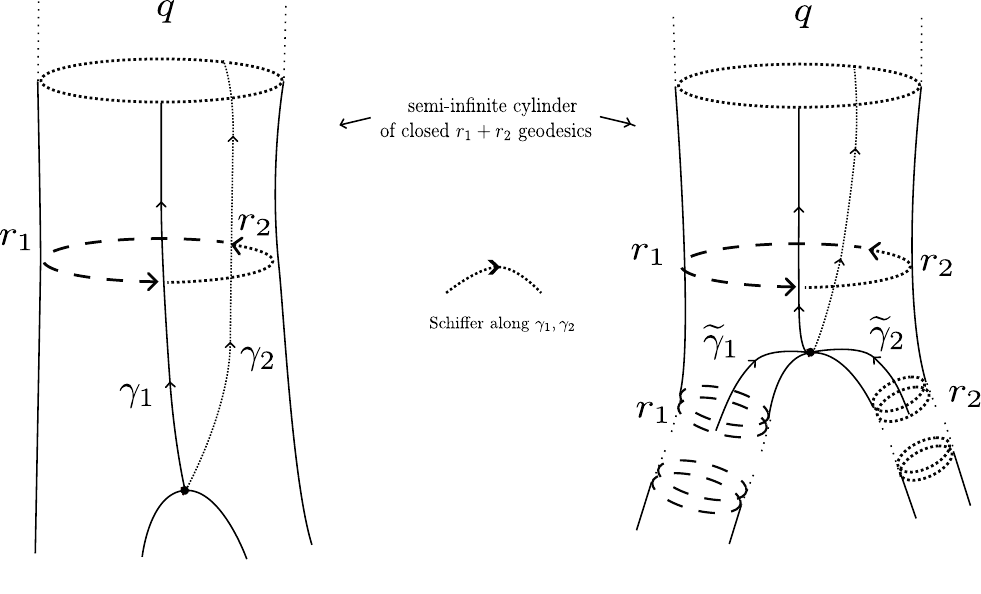}
     \caption{Schiffer variations along a pair of twins leading to one pole}
     \label{fig:twins_2_one_pole}
 \end{figure}

In Figure \ref{fig:schiffer_twins_2_2_poles} we develop the Schiffer variations along two semi-infinite geodesic twins $\gamma_1$ and $\gamma_2$ starting at a saddle and leading to distinct simple poles $q_1$ and $q_2$ of peripheral periods $r_1$ and $r_2$ respectively. If the peripheral periods satisfy $\mathbb{R}r_1=\mathbb{R}r_2$ then necessarily $r_1/r_2>0$. Indeed, the oriented  directional foliation determined by the peripheral period $r_1$ has a cylinder of closed geodesics around both $q_1$ and $q_2$.  On the other hand, the angle of the geodesic twins with the given oriented foliation is constant. This implies that the orientation of the geodesic in both cylinders is the same, or, in other words, $r_1/r_2>0$. In particular we deduce that $r=r_1+r_2$ does not vanish. 

For each large enough $t$ consider the union of the two closed geodesic around the poles $q_1$ and $q_2$ passing through $\gamma_1(t)$ and $\gamma_2(t)$ respectively. As we do the Schiffer variation the pairs of closed curve become each a single closed curve having developed image the complex number $r$ for the new developing map, and their union form a cylinder with modulus tending to \(+\infty\) as \(t\) tends to \(1\).
In the limit the cylinder converges to a node. On one side of the node we have a meromorphic differential on a genus zero curve with three poles having peripheral periods \(r_1, r_2\) and \( -r\). On the other side we have a curve of genus $g$ with a pole less (the two poles of period \(r_1\) and \(r_2\) have been replaced by a single pole of period \( r\), the other set of peripheral periods remains unchanged).

In Figure \ref{fig:twins_2_one_pole} we develop the Schiffer variation along two semi-infinite geodesic twins $\gamma_1$ and $\gamma_2$ of a stable form $(C,\omega)$ of type $(g,n)$ starting at a saddle and leading to a fixed pole $q\in C$. The peripheral period $r$ of $q$ is partitioned into a sum of two numbers $r=r_1+r_2$ each corresponding to the integral of $\omega$ along the peripheral curve in one of the sectors defined by the germs of twins at the pole. As the parameter of the Schiffer variation tends to infinity the resulting translation surface has two cylinders of peripheral period $r_1$ and $r_2$ respectively, whose height tends to infinity. In the limit they correspond each to a node of the limit nodal curve of peripheral period $r_1$ and $r_2$ respectively having each a  branch in a genus zero component with three poles. Let us analyze the topology of the limit nodal curve.  Let $\gamma=\gamma_2^{-1}\star\gamma_1$ be the concatenation of the paths in the surface. 
If $\gamma$ separates the surface in two components of types $(g_1,n_1+1)$ and $(g_2,n_2+1)$ respectively, where \( g_1+g_2=g\) is the genus of \(C\) and \(n=n_1+n_2\) is the number of poles of $\omega$, then the Schiffer variation produces a (possibly non-stable) nodal curve with two separating nodes, whose normalization has a part of type $(0,3)$ containing two of the points corresponding to the nodes with peripheral periods $-r_1$ and $-r_2$ respectively,  and two parts of type $(g_1, n_1+1)$ and $(g_2,n_2+1)$ respectively each containing a point corresponding to a node (of peripheral period $r_1$ and $r_2$ respectively). If one of the parts has non-trivial group of automorphisms -- which in the present context corresponds to having genus zero and precisely two marked points-- the corresponding point  in the moduli space of stable forms is the stabilization of the obtained form on the nodal surface. In the latter case, the stable curve has less than two nodes.
If $\gamma$ is a non-separating simple closed curve in $C$, then the Schiffer variation produces a curve with two non-separating nodes that together separate the surface in two components, one of type $(0,3)$ containing two points corresponding to distinct nodes of peripheral period $r_1$ and $r_2$ and another of type $(g,n+1)$ having two points corresponding to the node.

For more details and examples of deformations that include nodal curves, we refer to \cite{CDF2}. If the form is marked, the surgery allows to follow the marking along the surgery. 

\subsection{Review of known connected components in isoperiodic Torelli spaces} 

\begin{lemma}[Genus zero case]\label{l: connectedness spherical case}
Let $\Sigma$ be a sphere  and $\PP\subset\Sigma$ a set with $n$ points. For any period homomorphism \(p \in H^1 (\Sigma\setminus\PP , \mathbb C)\), the natural map \[\Omega\mrs_{\Sigma,\PP}(p)\rightarrow  \mathcal M_{0,n}\] to the moduli space of configurations of \(n\) marked distinct points in the sphere  is a biholomorphism; in particular, the isoperiodic set \(\Omega\mrs_{\Sigma,\PP}(p)\) is connected and non-empty.  
\end{lemma}

\begin{proof}
Given a configuration of \(n\) distinct points \(z_1, \ldots, z_n\) on the Riemann sphere, and numbers \(\alpha_1, \ldots, \alpha_n \in \mathbb C\) with zero sum, there is a unique meromorphic differential on the Riemann sphere having simple poles at the \(z_i\)'s, with peripheral periods given by the \(\alpha_i\)'s: this is the form 
\[ \frac{1}{2\pi i}\sum_j\frac{\alpha_j dz}{z-z_j} . \]  
\end{proof}

\begin{definition}
Let $\Sigma$ be a compact oriented surface. The volume of $p\in H^1(\Sigma,\mathbb{C})$ is defined as the intersection number $$\text{vol}(p)=\Re(p)\cdot\Im (p)$$ 
in the (symplectic) cohomology group $H^1(\Sigma,\mathbb{R})$.
\end{definition}

The following criterion will be useful to detect non empty boundary strata in the closure of the fiber of \(p\):

\begin{lemma}\label{l: Haupt}
  Let $g\geq 1$ be the genus of $\Sigma$ and $n$ be the cardinality of $\PP$. A cohomology class \( p \in H^1 (\Sigma\setminus\PP, \mathbb C)\setminus 0 \) not vanishing on any peripheral class is the period of a meromorphic form with simple poles  \textit{ on a smooth curve} in any of the following cases: 
\begin{itemize}
 \item if \(n=0\),  \(g=1\) and \( \text{vol} (p) >0 \) 
 \item if \(n=0\), $g\geq 2$ and \( \text{vol} (p) > \text{vol} \left(\mathbb C / p (H_1(\Sigma,\mathbb Z)) \right) \).\footnote{if \(p (H_1(\Sigma_g,\mathbb Z))\) is not discrete, then \(\text{vol} \left(\mathbb C / p (H_1(\Sigma,\mathbb Z)) \right) =0\) by convention.}
    \item if \(n=2\) and \( p(\Pi) \neq p (H_1(\Sigma\setminus\PP), \mathbb Z))\).
    \item If $n\geq 3$ and $p (H_1(\Sigma\setminus\PP))\subset \mathbb{C}$ is not a cyclic submodule.
\end{itemize}
\end{lemma}

\begin{proof} The statements on meromorphic forms can be found in \cite{CFG}. We include some other references that use the present context and a short proof of the fourth item.
The first item is obvious. The second item is a theorem due to Haupt, see \cite{Haupt} and \cite{Kapovich, CDF2} for alternative proofs. The third item has been proved in \cite{CD} in the case \(n=2\). 
    
    Let us now assume \(n\geq 3\). By hypothesis we have \( p (H_1(\Sigma\setminus\PP ,\mathbb Z) ) \neq \mathbb Z p (\pi_1) \). Let \( E=\{2,\ldots, n\} \); we can find a \(E\)-module \(M\) in restriction to which \(p\) does not take values in \( \mathbb Z p (\pi_1)\). So the restriction of \(p\) to \(M\) is the period of a meromophic differential with two poles marked by \(\pi_1\) and \(\pi_E\) by case \(n=2\), see \cite{CD}. We glue this differential at the pole \(\pi_E=-\pi_1\) with a form on the sphere with \(n\) poles marked by \(\pi_1, \pi_2,\ldots, \pi_n\) constructed via Lemma \ref{l: connectedness spherical case}. The obtained point lies in a boundary stratum of $\Omega\mnc_{\Sigma,\PP}$ and has period homomorphism $p$. By Theorem \ref{t:local structure} there are points in the local fiber of $\percompact$ around that point that lie outside the boundary.

    \end{proof}

When \( \text{vol} (p) = m \text{vol} \left(\mathbb C / p (H_1(\Sigma,\mathbb Z)) \right) \) for some $m\geq 1$, the holomorphic forms of periods $p$ are branched coverings of degree $m$ over the  elliptic curve \(\mathbb C / p \big(H_1(\Sigma,\mathbb Z)\big)\). Of course $m=1$ can only occur for $\Sigma$ of genus $g=1$. 
A similar phenomenon occurs for the case with some simple poles: If $p(H_1(\Sigma\setminus\PP))$ is cyclic and there exists $m\in \mathbb{Z}$ such that $p(H_1(\Sigma\setminus\PP))=\frac{1}{m}p(\Pi)$, the meromorphic form of period $p$ is a degree $m$ branched covering of the form $(\mathbb{P}^1,\frac{dz}{z})$. Again, the case $m=1$ can occur only for genus zero. In the next two theorems the cases of degree $m=2$ are excluded.

\begin{theorem}[Holomorphic case, \cite{CDF2}]
    \label{t:isoperiodic sets of holo diff}  Let $\Sigma$ be a compact oriented surface. If $p\in H^1(\Sigma,\mathbb{C})$ has $\text{vol}(p)>0$ and either $\Sigma$ is of genus one or \( \text{vol} (p) > 2\, \text{vol} \left(\mathbb C / p (H_1(\Sigma,\mathbb Z)) \right) \), then $\Omega\mrs(\Sigma)(p)$ is connected (and non-empty). 
\end{theorem}

\begin{theorem}[Two poles case]{\cite{CD}}
    \label{2poles} Let $\Sigma$ be a genus $g$ compact surface and $\PP\subset\Sigma$ a subset of two points. If $p\in\text{Hom}(H_1(\Sigma\setminus\PP),\mathbb{C})$ has non-zero peripheral periods and \(p (H_1(\Sigma\setminus\PP), \mathbb Z))\) is neither \(p(\Pi) \) nor \(\frac{1}{2}p(\Pi) \) then $\Omega\mrs_{\Sigma,\PP}(p)$ is connected and non-empty. 
\end{theorem}

\section{Connected components of some boundary strata of isoperiodic forms}

\subsection{Attaching map for compact type boundary strata in Torelli bordification}


\begin{lemma}\label{l: attaching map}
\label{p:connected isoperiodic strata}
    Let $(\Sigma,\PP)$ be of type $(g,n)$ and $[c]$ be a pointed Torelli class of a separating curve system  $c\subset \Sigma\setminus \PP$.   Then for every $p\in\text{Hom}(H_1(\Sigma\setminus\PP);\mathbb{C})$ one has 
    \[ \Omega\mathcal{B}_{[c]}(p) \simeq \prod _{v\in V(G)} \Omega \mathcal S _{ \hat{\Sigma}_v, \hat{\PP}_v } (p_v) \]
    where \(\simeq \) is a biholomorphism, \(G\) is the \((\Sigma,\PP)\) marking of a \(\PP\)-labelled graph associated to the boundary stratum \(\mathcal{B}_{[c]}\), and \(p_v = (\iota_v )^* p \) with \(\iota_v: \hat{\Sigma}_v \setminus \hat{\PP}_v\rightarrow \Sigma \setminus \PP\) is the natural inclusion of a component of the (separating) normalization of \((\Sigma,\PP)\) along \(c\). In particular, the stratum \(\Omega \mathcal B_{[c]}(p) \) is non-empty (resp. connected) iff each of the isoperiodic moduli spaces \( \Omega \mathcal S_{\hat{\Sigma}_v, \hat{\PP}_v}(p_v)\) are non-empty (resp. connected). 
\end{lemma}

\begin{proof} As $c$ has no curves that are non-separating the separating normalization coincides with the full normalization and we use the notations of the latter. Moreover, the holomorphic map \eqref{eq:normaliza separantes} can be rewritten in a simpler way:
\begin{equation}\label{eq:normalization in Torelli 2}
\Omega\mathcal{B}_{[c]}\rightarrow \prod _{v\in V(G)} \Omega \mathcal S_{\hat{\Sigma}_v, \hat{\PP}_v}
\end{equation}

By Proposition \ref{p:Torelli equivalence sesc systems}, this map is injective, and hence an isomorphism onto its image.  

Since an element \( (C,Q,m,\omega)\in \Omega\mathcal B_c \) has period homorphism  \(p\) iff the periods of the pull back \( \iota_{C,v}^* \omega\) by the natural injective holomorphic map \(\iota_{C,v} = \hat{C}_v\setminus \hat{Q}_c \rightarrow C\setminus Q\) have period homomorphism \(p_v\)  for all \(v\in V(G)\), the restriction of the map \eqref{eq:normalization in Torelli 2} to $\Omega\mathcal{B}_c(p)$ induces the desired isomorphism.
\end{proof}

\subsection{A connectedness property for boundary strata of isoperiodic forms with one node, non-separating case}
\label{ss:boundary strata of non-separating type}
Let \([c]\) be a Torelli class of non separating simple closed curve \(c\) in \(\Sigma\setminus \PP\), and \(p\in H^1 (\Sigma\setminus \PP,\mathbb C)\) that contains the homology class of \(c\) in its kernel. It determines an isoperiodic set in a the boundary stratum \(\Omega_0 \mathcal B_{[c]}(p) \subset \Omega_0\overline{\mathcal S}_{\Sigma, \PP}(p) \). 

We denote by \(\hat{p}: H_ 1 (\hat{\Sigma}\setminus \hat{\PP}, \mathbb Z) \rightarrow \mathbb C\) the period \( \hat{p}:= p \circ \iota \), where \( \iota: \hat{\Sigma}\setminus \hat{\PP}\rightarrow \Sigma \setminus \PP\) is the natural immersion.  As in Lemma \ref{l: attaching map}, one can define a normalization map 
\begin{equation}\label{eq: normalisation map} \Omega _{0} \mathcal B_{[c]}(p) \rightarrow \Omega \mathcal S _{\hat{\Sigma}, \hat{\PP}} (\hat{p})   \end{equation} 
that sends an element \( (C, Q, m,\omega) \in \Omega_{0}  \mathcal B_{[c]}(p)\) to the element \( (\hat{C}, \hat{m}, \hat{\omega})\in \Omega \mathcal S _{\hat{\Sigma}, \hat{\PP} } (\hat{p})\) with \(\hat{C}\) the normalization of \(C\), \(\hat{m}\) the isotopy class of the homeomorphism \( \hat{f}: (\hat{\Sigma}, \hat{\PP}) \rightarrow (\hat{C}, \hat{Q}) \) naturally induced by a map \( (\Sigma, \PP) \rightarrow ( C, Q) \) pinching \(c\) in the class of \(m\), and where \(\hat{\omega} := \iota_C ^* \omega\) with \(\iota: \hat{C} \setminus \hat{Q} \rightarrow C \setminus Q\) is the natural holomorphic immersion.

\begin{remark}
    Notice that if \( n_1,n_2\in \hat{\PP}\) are the two points obtained by contracting the two  components of the boundary of \( \overline{\Sigma\setminus c}\), then  \(\hat{p}(\pi_{n_{1}})=0=\hat{p}(\pi_{n_{2}}) \). 
\end{remark}

\begin{remark}
    Notice that  \(f\) is well defined up to precomposing it by a homeomorphism \(h\) of \((\Sigma, \PP)\) fixing \(c\) and whose isotopy class belongs to the Torelli group of \((\Sigma, \PP)\). This map \(h\) produces a homeomorphism \(\hat{h} \) of \((\hat{\Sigma},\hat{\PP})\). The action of \(\hat{h}\) on \(H^1(\hat{\Sigma}\setminus \hat{\PP}, \mathbb Z) \) is equal to the identity. So the map \eqref{eq: normalisation map} is well-defined.
\end{remark}

We denote by \( (\overline{\Sigma}, \overline{\PP})\) the pointed surface defined by 
\[ \overline{\Sigma} := \hat{\Sigma} \text{ and } \overline{\PP}:= \hat{\PP} \setminus \{n_1,n_2\}.\]
Since the period \(\hat{p}\) vanishes on \( \pi_{n_1}\) and \(\pi_{n_2}\), there exists a  homomorphism \(\hat{p} : H_1(\overline{\Sigma} \setminus \overline{\PP} , \mathbb Z) \rightarrow \mathbb Z) \) such that \( \hat{p} \) is the composition of the natural inclusion \( H_1 (\overline{\Sigma} \setminus \overline{\PP}, \mathbb Z)\rightarrow  H_1 (\hat{\Sigma} \setminus \hat{\PP}, \mathbb Z)\) with \(\hat{p}\), and one can define a forgetful map 
\begin{equation} \label{eq: forgetful map} \Omega \mathcal S _{\hat{\Sigma},\hat{\PP}} (\hat{p})  \rightarrow \Omega \mathcal S_{\overline{\Sigma}, \overline{\PP}}(\overline{p}) .  \end{equation}
that sends a point \( (\hat{C}, \hat{Q}, \hat{m}, \hat{\omega})\in \Omega \mathcal S _{\hat{\Sigma},\hat{\PP}} (\hat{p})\) to the point 
\((\overline{C}, \overline{Q}, \overline{m}, \overline{\omega})\in \Omega \mathcal S_{\overline{\Sigma}, \overline{\PP}}(\overline{p}) \) defined by 
\(\overline{C}:= \hat{C}\), \(\overline{Q} = \hat{Q}\setminus \{q_1,q_2\}\) where \(q_1,q_2\) are the two points of \(\hat{C}\) marked by \(n_1\) and \(n_2\) respectively, \(\ \overline{m}\) is the obvious map induced by \(\hat{m}\),  and 
\(\overline{\omega}\) is the holomorphic extension of \(\hat{\omega} \) on \( \overline{C}\setminus \overline{\PP}\).

\begin{proposition} 
\label{p:connectedness of non compact boundary stratum}For every period \(p\in H^1 (\Sigma\setminus \PP, \mathbb C ) \) vanishing on the homology class of \(c\), the map 
\[ f : (C,Q,m,\omega) \in \Omega_0 \mathcal B _ c (p) \rightarrow (\overline{C}, \overline{Q}, \overline{m}, \overline{\omega} ) \in \Omega \mathcal S _{\overline{\Sigma}, \overline{\PP}} (\overline{p}) \]
obtained as the composition of the map \eqref{eq: normalisation map} with the map \eqref{eq: forgetful map}, sends every connected component of \(\Omega \mathcal B _ c (p)\) onto a connected component of \(\Omega \mathcal S _{\overline{\Sigma}, \overline{\PP}} (\overline{p})\). 
\end{proposition}

\begin{proof}
 We claim that the map \(f\) has the following property: for any continuous path \(\gamma : [0,1] \rightarrow \Omega \mathcal S_{\overline{\Sigma}, \overline{\PP}  } (\overline{p} ) \), and any \((C_0,Q_0 , m_0,\omega_0)\in f^{-1} (\gamma(0))\), there exists a monotone surjective continuous map \(\theta: [0,1]\rightarrow [0,1]\) such that the path \(\gamma\circ \theta \) can be lifted to a path \(\eta: [0,1]\rightarrow \Omega \mathcal B _ c (p)\) satisfying \( f\circ \eta= \gamma \circ \theta\) and \(\eta(0)= (C_0,Q_0, m_0,\omega_0)\). This (weak) lifting property of paths implies the lemma since connected components of \(\Omega \mathcal S_{\overline{\Sigma}, \overline{\PP} } (\overline{p})\) are path connected (they are manifolds). 

 Let us now prove the claim. Let \(a\in H_1(\Sigma\setminus \PP, \mathbb Z)\) be the homology class of \(c\), and  \( b\in H_1 (\Sigma\setminus\PP) \) be a class such that \( a\cdot b= 1\). Given \( (\overline{C}, \overline{Q}, \overline{m}, \overline{\omega}) \in \Omega \mathcal S_{\overline{\Sigma}, \overline{\PP} } (\overline{p})  \), elements of \( f^{-1} (\overline{C}, \overline{Q}, \overline{m}, \overline{\omega}) \) are in correspondance with the moduli of homology classes \(\overline{\kappa}\) with {\bf different} fixed extremities in \( \overline{C}\setminus \overline{Q} \), satisfying 
 \begin{equation} \label{eq: isoperiodic constraint} \int _{\overline{\kappa}} \overline{\omega} = p ( b) . \end{equation}
 Indeed, given such a homology class \(\overline{\kappa}\), one associates a pointed nodal curve \((C,Q)\) by identifying the endpoints of \(\overline{\kappa}\) and by letting \(Q\) being the image of \(\overline{Q}\) after the identification of the points, a meromorphic form \(\omega\) on \(C\) which is the extension of the naturally defined form  \(\overline{\omega}\) outside the node, and the marking \(\overline{m}\) is defined by the class of pinching maps \( (\Sigma, \PP) \rightarrow (C, Q) \) that pinch \(c\) to the node (and is a homeomorphism outside \(c\)), send the homology class \(b\) to the homology class of the image of \( \overline{\kappa}\) in \(C\) (this is a closed loop) and satisfy that the induced map\( (\overline{\Sigma}, \overline{\PP})\rightarrow (\overline{C}, \overline{Q} )  \) belongs to the Torelli class \(\overline{m}\). Such markings exist, and are unique up to the group generated by the Torelli group and the Dehn twist around the curve \(c\), which stabilizes the stratum \(\mathcal B_c\). So the element \((C, Q,m,\omega) \) is well-defined and by construction its period is \(p\). Reciprocally, given an element \((C,Q, m,\omega) \in  f^{-1} (\overline{C}, \overline{Q}, \overline{m}, \overline{\omega}) \), let \(\kappa: [0,1]\rightarrow C\setminus Q\) be a representative of \(p(b) \) such that \( \kappa(t)\) is equal to the node exactly for \( t=0\) and \(t=1\). Then \(\kappa\) lifts to the normalization \(\hat{C}\simeq \overline{C}\) as a path \(\overline{\kappa}\) with  fixed extremities the two distinct points \(q_1,q_2\in \hat{C}\) coming from the node. The homology class of \(\hat{\kappa}\) with fixed extremities is well-defined in \(\overline{C}\), and we are done.

We denote by \( \gamma(t) = (\overline{C_t}, \overline{Q_t} , \overline{m}_t , \overline{\omega}_t)\), and by \( r_0>0 \) the infimum of the injectivity radii of the surfaces \( (\overline{C_t},\overline{\omega}_t)  \) for \(t\in [0,1]\) (see subsection \ref{ss: singular translation}). In particular, for any \(t\in [0,1]\) and any point \(x\in \overline{C_t}\), the form \(\overline{\omega}_t\) is exact in restriction to the open ball of radius \(B_{\overline{\omega}_t} (x, r_0) \) and the primitive of \(\overline{\omega}_t\) on this ball that vanishes at \(x\) induces a surjective map onto the euclidian disc of radius \(r_0\) centered at the origin in the complex line.  

By uniform continuity of \(\gamma\), for any \(\varepsilon >0\), there exists a time \(\tau>0\), such that for every \(t\in [0,1]\) and \(h\in [0, \inf (\tau, 1- t)]\), the surface \( (\overline{C}_{t+h}, \overline{\omega}_{t+h})\) is obtained from \( (\overline{C}_t, \overline{\omega}_t)\) by performing a surgery on the \(\varepsilon \)-neighborhood of the zero set \(Z(\overline{\omega} _{t_0})\) of \(\overline{\omega}_{t_0} \). We will fix  \(\varepsilon = \frac{r_0 }{ C ^{(2g-2+n) ^2 +1}} \) in the sequel where \(C\) is a large constant (at least \(\geq 2\)) to be determined later on.

Let \(t_0\in [0,1]\) and \(\overline{\kappa}_{t_0}\) be any homology class with fixed distinct extremities in \( \overline{C}_{t_0}\). We will show that one can find a non decreasing continuous surjective map \(\theta : [t_0, t_0 +\varepsilon ]\rightarrow [t_0, t_0+\varepsilon ] \) and a path \( \eta : [t_0, t_0 +\varepsilon ]\rightarrow \Omega \mathcal B_a(p) \) such that \( f\circ \eta = \gamma \circ \theta\).  Applying this inductively to the intervals of a sufficiently thin subdivision of the unit interval, this will prove the claim. 

We define \(\theta \) as follows: \( \theta (t) = t_0\) if \(t\in [t_0, t_0 +\varepsilon/ 2] \) and \( \theta (t) =  2t - (t_0 +\varepsilon/2)\) if \( t\in [t_0 +\varepsilon/ 2, t_0 +\varepsilon ]\). We will now construct a family of homology classes with fixed distinct extremities \(\overline{\kappa}_t\) on \( \overline{C}_{\theta(t) }\), for \( t\in [t_0, t_0 +\varepsilon]\) that satisfies the isoperiodic constraint \eqref{eq: isoperiodic constraint}.

In the interval \([t_0, t_0 +\varepsilon/ 2] \), this family consists in a deformation of \(\overline{\kappa}_{t_0}\) into a family  of homology classes \(\{\overline{\kappa}_t\}_{t\in [t_0, t_0 +\varepsilon/ 2]}\) with fixed distinct extremities on the \textit{same} curve \(\overline{C}_{t_0}\), the objective being that after this deformation the two extremities of \(\overline{\kappa}_{t_0+\varepsilon /2}\) are sufficiently far away from the zero set of \(\overline{\omega}_{t_0}\). We use the following  

\begin{lemma}[Archipelago lemma]
There exists \( r \in [\frac{r_0}{C^{(2g-2+n)^2}} , r_0  ]\) such that one can cover the zero set \(Z(  \overline{\omega}_{t_0}) \) of \(\overline{\omega}_{t_0}\) by a finite family of balls \( B( v , r/C ) \subset \overline{C}_{t_0}\), \(v\in V\subset \overline{C}_{t_0} \), whose centers are separated in \( (\overline{C}_{t_0}, \overline{\omega}_{t_0})\) by a distance larger that \(r\). 
The balls \( B(v, 2r/ C) \), \(v\in V\), will be called the archipelagos. Any point which does not belong to any archipelago is \(\varepsilon \)-distant to the zero set of \(\overline{\omega}_{t_0}\).
\end{lemma}

\begin{proof} Since there are no more than \( 2g-2+n\) zeroes, the pigeon holes principle says that there exists \(r\) of the form \( r= r_0 / C ^{k} \), for \( k\) an integer in between \(0\) and \( (2g-2+n)^2 - 1\) so that the interval \( (r/C, r] \) does not contain any number of the form \( d(v, v') \) for \( v,v'\) two distinct elements in \(Z(\overline{\omega}_{t_0})\).  Choose such a number. The relation on \(Z(\overline{\omega}_{t_0})\) defined as \( v\sim v'\) iff \( d(v,v') \leq r/C\) is then an equivalence relation, for if \( d(v,v') \leq r/C\) and \(d(v',v'') \leq r/C\), then \( d(v, v'')\leq 2r/C\) by triangular inequality, and since \( d(v,v'')\notin (r/C, r]\), this implies that \( d(v,v'')\leq r/C\).  A subset \(V\subset Z(\overline{\omega}_{t_0})\) of representatives of this equivalence relation satisfies the conclusion of the first statement. Now, since the zero set \(Z(\overline{\omega}_{t_0})\) is contained in the union of the balls \( B(v, r/C)\), any point that does not belong to any archipelago, have a  distance to the zero set of \(\overline{\omega}_{t_0}\) bounded from below by \(r/C\geq r_0 / C^{(2g-2+n)^2+1}=\varepsilon\).
\end{proof}

Let \(\overline{\kappa}_{t_0}^\pm\) the two distinct extremities of \(\overline{\kappa}_{t_0}\), and let \( \Theta \) be a direction which is different from any integral of \(\overline{\omega}_{t_0}\) in between \(\overline{\kappa}_{t_0} \) and a zero of \(\overline{\omega}_{t_0}\). The two unit speed geodesics \(\alpha^\pm \) issuing from \(\overline{\kappa}_{t_0}^\pm\) in the \(\Theta\) direction do never cross a zero of \(\overline{\omega}_{t_0}\) and are defined for every positive time.  These geodesics have three properties: 
\begin{enumerate}
    \item after entering an archipelago, they leave it before the time \(4r/C\),  
    \item the time they spend in between two archipelagos is at least \( r- 4r/C\), 
    \item \(\alpha ^+ (s)\neq \alpha ^- (s) \) for any \(s\geq 0\).
\end{enumerate}
The first property is a consequence of the fact that the primitive of \(\overline{\omega}_{t_0}\) can be defined on each archipelago and is mapped onto an euclidean ball of radius \(2r/C\), and so the image of a connected component of the intersection of \(\alpha^\pm\) with a given archipelago are euclidean unit speed geodesic segments that cannot spend a time exceeding \(4r/C\) in that ball. The second property comes from the fact that points of \(V\) are separated by a distance of \(r\). The third from the fact that \(\alpha ^\pm \) have the same direction and different origin. 

The first and second property imply that the density of time that either  \(\alpha ^+\) of \(\alpha ^- \) spend in an archipelago is bounded from above by \( 2\frac{4r/C}{r- 4r/C}=\frac{8}{C-1} \). So if \( C>9\) there exists a time \(s_0\) so that none of \( \alpha^+ (s_0)\) nor \(\alpha^- (s_0)\) does  belong to an archipelago. For \(t\in [t_0, t_0+\varepsilon /2]\), we define \(\overline{\kappa}_t \) as the homology class with fixed extremities 
\[ \overline{\kappa}_t ^\pm  = \alpha^ \pm \left(t_0+ (t-t_0) \frac{2s_0}{\varepsilon}\right)\]
on \(\overline{C}_{t_0}\), which is a continuous deformation of \(\overline{\kappa}_{t_0}\). By construction, the extremities of \( \overline{\kappa}_{t_0+\varepsilon/2}\) and distant from the zero set of \(\overline{\omega}_{t_0}\) by an amount of \(\varepsilon\).

We are now in a position to  define \(\overline{\kappa}_t\) for \(t\in [t_0+\varepsilon /2, t_0+\varepsilon ]\). We first take a representative of \(\overline{\kappa}_{t_0+\varepsilon/2}\) on \(\overline{C}_{t_0}\) that does not cross the archipelagos. Since the translation surfaces \( (\overline{C}_{\theta(t)}, \overline{\omega}_{\theta(t)} ) \) for \(t\in [t_0+\varepsilon/2, t_0+\varepsilon]\) are obtained from \((\overline{C}_{t_0}, \overline{\omega}_{t_0} )\) by a surgery inside the archipelagos, we can define the homology class \(\overline{\kappa}_t\) with fixed extremities as being given by \(\overline{\kappa}_{t_0+\varepsilon/2}\)  in the complementary to the archipelagos. This completes the proof of the lemma.
\end{proof}

\section{Isoperiodic forms with a single zero and real periods on smooth curves}

\subsection{Combinatorial structure of a form with real periods}\label{eq: combinatorial structure of a form} 

A useful tool to understand the structure of forms with real periods is the notion of ribbon graph. This is a graph together with the data of a cyclic orientation of the set of germs of edges getting out of a vertex for each vertex of the graph. A graph embedded in an oriented surface has a natural structure of ribbon graph, and reciprocally, any ribbon graph can be embedded in a germ of oriented surface in a unique way up to orientation preserving homeomorphism. There are some particular cycles (left and right cycles) on such a ribbon graph that will play an important role in the sequel: the left (resp. right) cycles are the concatenation of positive edges obtained by turning left (resp. right) at each vertex. In the sequel we will consider special kind of oriented ribbon graphs: those having an orientation of their edges, in such a way that at each vertex, the edges are successively in and out in the cyclic ordering. A riemannian ribbon graph is a ribbon graph together with a riemannanian metric on its edges. 

Let \( (C, Q, \omega) \in\Omega\mathcal{M}_{g,n}\) be a meromorphic \(1\)-form with $n$ simple poles on a smooth genus $g$ curve. If the periods of $(C,Q,\omega)$ are real, i.e. $$\Lambda_{\omega}=\left\{ \int_a\omega: a\in H_1(C\setminus Q,\mathbb{Z})\right\} \text{ is a subgroup of } \mathbb{R},$$ then the oriented real foliation $\mathcal{G}:=\mathcal{G}_{1}$ of $C$ associated to $1\in\mathbb{S}^1$ (see subsection \ref{ss: singular flat metric and directional foliations}) admits a non-constant real analytic first integral $C\rightarrow \mathbb{R}$, given by $z\mapsto\int_{z_0}^z\Im \omega$. 
The union of saddle connections of a the foliation \(\mathcal G\) has the structure of an oriented riemannian  ribbon graph, that we denote \( R(\mathcal G)\), and that we call the saddle connection graph. It is contained in the pre-image of the finite set of critical values of the first integral \( z \mapsto \int_{z_0} ^z \Im \omega\). 
The saddle connection graph is oriented, and carries a riemannian structure, the one induced by the square of the modulus of the meromorphic form.

A simple pole of a form with real periods is said to be positive or negative if the period of the form on a positive loop around the pole has the corresponding sign.  

The existence of the first integral then implies that there are no minimal components in $C\setminus R(\mathcal{G})$. The only possibility left in this case for each component of $C\setminus R(\mathcal{G})$ is a cylinder of closed geodesics. Moreover, every local separatrix at a zero extends to a saddle connection. The graph $R(\mathcal{G})$ is precisely the union of all of them and the number of vertices and edges of $R(\mathcal{G})$ is uniquely determined by the orders of the zeros. For instance, if the form $\omega$ has a single zero then the saddle connection graph has $2g-1+n$ closed saddle connections at one vertex. Moreover, in this case, all of the cylinders in the complement of the saddle connection graph have infinite volume. Indeed, the existence of a component of closed geodesics of finite volume would imply that the saddle lies in the boundary of both boundary components of the annulus of closed geodesics, and therefore it would be possible to construct a closed geodesic that is orthogonal to $\mathcal{G}$ at each point. The integral of $\omega$ along this geodesic would have non-zero imaginary part, contrary to the current hypothesis.

 Next we remark that the boundary component of any of the semi-infinite cylinders describes a cycle in the ribbon graph $R(\mathcal{G})$ that respects the orientation of the foliation. Indeed,  each saddle connection will have a pole of positive peripheral period on its left and a pole of negative peripheral period to its right. If we start at the saddle and follow the closeby (closed) leaves of the pole on the left (resp. right), the separatrix that we follow after reaching the saddle is the adjacent one on the side of the pole of positive (resp. negative) peripheral period. This is precisely the definition of left (resp. right) cycle on a given ribbon graph. Reciprocally, each left (resp. right) cycle in the ribbon graph corresponds to a unique positive (resp. negative) pole of the form. Homologically this relation can be stated as follows: each peripheral class in $H_1(C\setminus Q,\mathbb{Z})$ can be written uniquely as a sum \begin{equation}[s_1]+\dots+[s_k]\text{ or } -([s_1]+\ldots+[s_k])\label{eq:peripheral relation}\end{equation} where the ordered sequence $(s_1,s_2,\ldots,s_k)$ of pairwise distinct oriented saddle connections represents a cycle in the ribbon graph. The sign chosen in \eqref{eq:peripheral relation} is positive if the cycle always follows an adjacent saddle connection to the left, and negative if it follows the adjacent saddle connection to the right. The set of classes $$\{[s]\in H_1(C\setminus Q, \mathbb{Z}): s\text{ is a saddle connection of } \mathcal{G}\}$$ has $2g-1+n$ distinct elements.  Two saddle connections can only intersect at the saddle. The intersection of any two of them, $[s_i]\cdot [s_j]$ is zero if the union of separatrices of $s_i$ lie on one side of the union of separatrices of $s_j$ (and reciprocally) and $\pm 1$ depending on their relative orientations if they are alternated.  In the next lemma we determine conditions that allow to reconstruct a meromorphic form starting from a ribbon graph.  
\begin{lemma}\label{l:marked ribbon graph}
Let  $\mathcal R$ be an orientable Riemannian ribbon graph with a single vertex, $2k$ oriented  half-edges, having alternated orientations in the cyclic ordering and $n$ cycles (in total, to the left and to the right). Then, up to biholomorphism, there exists a unique meromorphic differential $\omega$ with $n$ simple poles, a single zero and real periods on a smooth complex curve $C$ of genus $g=\frac{(k-1)-n}{2}+1$ such that   $\mathcal{R}=R(\ker(\Re\omega))$. 

Moreover, for any choice of classes $c_1,\ldots, c_{2g-1+n}\in H_1(\Sigma\setminus\PP,\mathbb{Z})$ satisfying
the following conditions with respect to the set of edges $\{s_1,\ldots, s_{2g-1+n}\}$ of $\mathcal{R}$:
\begin{itemize}
    \item $c_i\cdot c_j=[s_i]\cdot [s_j]$  for every $i,j\in\{1,\ldots, 2g-1+n\}$.
\item for each sequence $i_1,\ldots,i_l$ such that $(s_{i_1},\ldots,s_{i_l})$ forms a left (resp. right) cycle in $\mathcal R$, then  $c_{i_1}+\ldots+c_{i_l}$ (resp.  $-(c_{i_1}+\ldots+c_{i_l})$) is the class of a positively oriented peripheral curve, 
\end{itemize}
there exists a unique element \( (C, m, \omega)\in\Omega \mathcal S _{\Sigma,\PP}\) lifting \((C, \omega) \in \mathcal M_{g,n}\) which is such that \( m _* c_i = [s_i]\) for each \( i = 1,\ldots, 2g-1+n\). 

\end{lemma}

\begin{proof}
To construct the compact oriented surface with $n$ punctures we add a flat semi-infinite cylinder to each side of each cycle in $\mathcal R$, having the perimeter of the boundary that corresponds to the sum of the lengths of the saddle connections that form it.  The given surface is endowed with a translation structure on the complement of $\mathcal R$ and it extends as a translation structure to the complement of the vertex and the $n$ punctures. Direct calculation via a triangulation shows that the oriented surface obtained by filling the punctures has Euler characteristic $(n+1)-k$, which gives the desired genus $g=\frac{k-n-1}{2}+1$.  The ribbon graph $\mathcal R$ is a deformation retract of the punctured surface. Its homology group is therefore isomorphic to $H_1(\mathcal R)$, which is an abelian group freely generated by  the classes of the $k=2g+n-1$ saddle connections. On the other hand we can extend the  complex structure to the punctures to define a compact Riemann surface, that we call $C$. The translation structure induces a holomorphic 1-form $\omega$ with isolated singularities at the punctures. Each cylinder corresponds to a simple pole. As the vertex has finite angle $2\pi k$ around it, the meromorphic form has necessarily a point of order $k-1$ at the vertex. The pair $(C,\omega)$ has the desired properties. Each of the leaves of the vertical foliation in $C$ cuts $R$ in a unique point. This allows to define the retraction $C\setminus (\omega)_{\infty}\rightarrow \mathcal R$. 

If $(C',\omega')$ is another form realizing a ribbon graph isomorphic to $\mathcal R$, we can construct a biholomorphism between $(C,\omega)$ and $(C',\omega')$ by using the fibers of the retractions  $C\setminus{(\omega)_{\infty}}\rightarrow \mathcal R$ and $C'\setminus{(\omega')_{\infty}}\rightarrow \mathcal R$ defined by following the leaves of the vertical foliation on $C$ and $C'$, and the fibers of the horizontal foliation on $C$ and $C'$. 

One can uniquely find a homological marking \(m\) permitting to define the lift of \((C,\omega)\in \Omega \mathcal M_{g,n} \) satisfying the second part of the statement with the use of Corollary \ref{c: marking homologically smooth compact curves}. \end{proof}



\subsection{Chord diagrams of ribbon graphs of saddle connections}

The set of germs of separatrices at a vertex  of the saddle connection graph $R(\mathcal{G})$ associated to a meromorphic differential with simple poles and real periods is partitioned into those that get out from their corresponding vertex (forming a set denoted \(S^{out} \)) and those that get in their corresponding vertex (forming a set denoted \(S^{in} \)). 

In the case of a form in the minimal stratum (i.e. having a single zero of multiple order), denoted $\Omega\mathcal M_{g,n}(2g-2+n)$, the combinatorial structure of $R(\mathcal{G})$  can be completely understood by the alternating cyclic ordering on the set \( S^{out} \cup S^{in} \) of germs of separatrices, together with the bijection \begin{equation}\label{eq:involutive corr} S^{out}\rightarrow  S^{in}\end{equation} between them that associates to each separatrix in $S^{in}$ the separatrix in $S^{out}$ that corresponds to the other end of the saddle connection. 

To represent the saddle connection graph we will mostly  use a chord diagram as follows: on an abstract oriented  circle, boundary of an oriented disc, we consider a point for each germ of separatrix of \( S^{out} \cup S^{in} \), ordered in the cyclic ordending of the separatrices. The correspondence \eqref{eq:involutive corr} can be interpreted in the finite set of points of the circle.  For each element of \(S^{in}\) we draw an oriented chord in the disc starting at the corresponding point of the circle and ending at the point of the circle corresponding to the element of \(S^{out}\) via \eqref{eq:involutive corr}.  The intersection between the homology classes of the saddle connections in the ribbon graph - described just before the statement of Lemma  \ref{l:marked ribbon graph}) - can be read in the chord diagram by looking at the intersections of the chords inside the disc: it is zero if they do not intersect and $\pm 1$ if they do, where the sign corresponds to the same sign of the intersection of the oriented chords.  Moreover, each left cycle (resp. right cycle) in the original ribbon graph -corresponding to a peripheral class of positive (resp. negative) pole-  can be obtained in the chord diagram by taking a cycle in the union of the chords with the circle, where we follow the direction of a chord, and, each time we reach an endpoint of a chord, we turn to the right (resp. left) either in the circle (if we arrive from a chord) or on a new chord, if we arrive from the circle. 
In other words, peripherals of positive poles are obtained by considering the cycles turning to the right and peripherals of negative poles by turning to the left, each time we reach a point of intersection of a chord with the circle. This is a consequence of the change in orientation of the saddle connections. 

 Therefore  to each element in $\Omega\mathcal M_{g,n}(2g-2+n)$ with real period, we associate a unique chord diagram having $(2g-2+n)+1$ chords, $n$ cycles and a length associated to each chord.  All examples of combinatorics with four chords are shown in Figures \ref{fig:ribbon  graphs} and \ref{ribbon}. Thanks to Lemma \ref{l:marked ribbon graph} the oriented chord diagram together with a data of lengths completely determines the meromorphic form up to biholomorphism.

\begin{figure}
     \centering
     \includegraphics[scale=0.6]{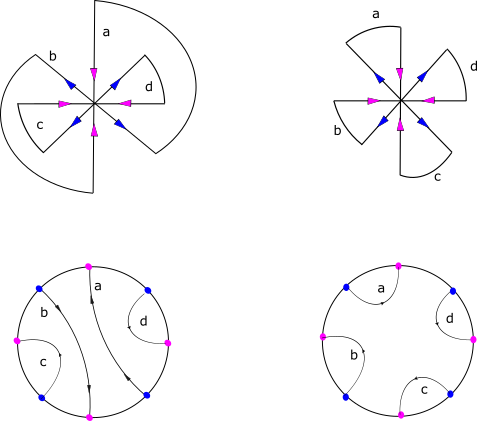}
     \caption{Ribbon graphs and their corresponding chord diagrams in the case of four chords corresponding to forms of type $(0,5)$}
     \label{fig:ribbon  graphs}
 \end{figure}

When we want to characterize the point in the Torelli cover $\Omega\mrs_{\Sigma,\PP}$ of $\Omega \mathcal M_{g,n}$ we need to keep the information on the homology classes associated to each saddle connection (chord) and their intersections. Indeed, by the characterization given in Lemma \ref{l:marked ribbon graph}, the homological marking is completely determined by this information. 

The chord diagram associated to a homologically marked form $(C,m,\omega)\in\Omega\mrs_{\Sigma,\PP}$ with real periods and a single zero  is the chord diagram of $(C,\omega)\in \Omega\mathcal M_{g,n}$ with the following extra information
\begin{itemize}
    \item To each chord we associate the homology class in the homology group of the reference surface $H_1(\Sigma\setminus\PP)$ determined by the class of the  corresponding saddle connection on the surface under the marking.
    \item The intersection of two homology classes  corresponding to two chords coincides with the (signed) intersection of the chords inside the disc.
\end{itemize}

\begin{remark}\label{rem:marking of chords} Whenever we have the information of the period homomorphism $p: H_1(\Sigma\setminus\PP)\rightarrow \mathbb{R}$ of the form, the homology classes of the saddle connections already give the information on the lengths via $p$, so we omit the lengths from the diagram. 
\end{remark}

In Figures \ref{fig:ribbon  graphs} and \ref{ribbon} we find two couples of examples  of saddle connection graphs and their associated chord diagrams. To define a point in $\Omega\mathcal{M}_{g,n}$we need to associate a positive real number to every chord. To define a lift to $\Omega\mrs_{\Sigma,\PP}$ we also need to associate a homological class to each chord and preserve the right relations between those, the peripheral classes, and their intersections. 
\subsection{Schiffer variations connecting forms with a single zero} 

In this section, we are interested in finding isoperiodic deformations (in the form of Schiffer variations)  that join different forms with a single zero. They correspond to the combinatorics of twins in Figure \ref{fig:schiffer_one_closed_twin} with $B=C=2\pi$: two twins forming an angle of $2\pi$ such that the shortest saddle connection among both does not return inside the $2\pi$ angle.

In order to define and analyze the effect of the Schiffer variation on the form,  we select two saddle connections in the same direction of distinct lengths that form an angle $2\pi$. Call the germ of the longer saddle connection {\bf long}, and the germ of the shorter saddle connection {\bf short}. In between {\bf long} and {\bf short}, there is a germ in an opposite direction. We call this third germ {\bf central}. In order to get another form on a smooth surface, we assume that the germs \({\bf short} \) and \({\bf central} \) belong to different saddle connections. The Schiffer variation changes the saddle connections  in the way described in Figure \ref{fig:schiffer_chord}. Let us analyze the effect on the chord diagrams.\\

Consider the chord diagram of the saddle at which the Schiffer variation is performed. The saddle connection containing the \({\bf short} \) germ divides the circle into two intervals: one of the intervals contains the germs \({\bf central} \) and \({\bf long}\). These two germs are located at one of the extremities of the interval. 

\begin{figure}
     \centering
     \includegraphics[scale=1]{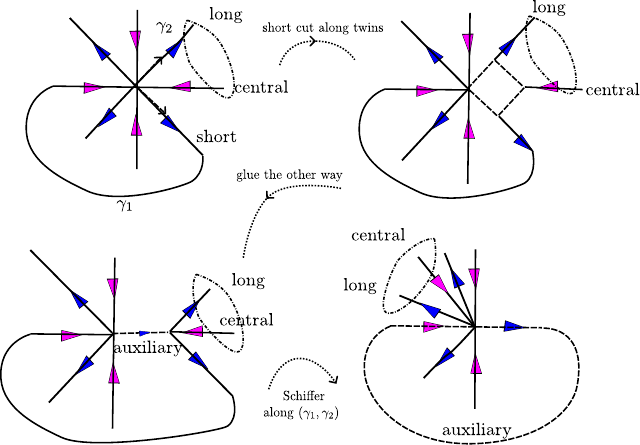}
     \caption{Combinatorial effect of the Schiffer variation along a pair of saddle connections $\gamma_1,\gamma_2$ forming initial angle $2\pi$ where $\gamma_1$ is shorter than $\gamma_2$.}
     \label{fig:schiffer_chord}
 \end{figure}

To obtain the chord diagram after the surgery, we take  the group of two germs \( \{ \text{{\bf central}, {\bf long}}\} \) of the initial diagram and place them in the same order at the other end of the circle interval between the extremities of the short saddle connection which contains the germs. This defines a new cyclic ordering \(S^{out}\cup S^{in} \) of germs of saddle connections on the chord diagram without changing the correspondence between the sets \( S^{out}\) and \(S^{in}\) (see Figure \ref{fig:fake_to_fake}). 

\begin{figure}[h]
     \centering
     \includegraphics[height=5cm]{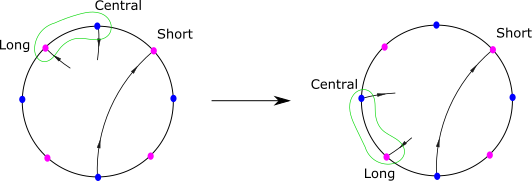}
     \caption{An effect of a Schiffer variation on a chord diagram: the group of two germs \( \{ \text{{\bf central}, {\bf long}}\} \) is moved to the opposite extremity of the interval that contains them.}
     \label{fig:fake_to_fake}
 \end{figure}

Let us further analyze the effect on the homology classes of the saddle connections. Denote by  
\[ [{\bf long}],\ [{\bf central}],\ [{\bf short}] \text{ and the other ones, } [{\bf other_k}]\text{'s},\]
the original set of homology classes of saddle connections. 
 
\begin{lemma}\label{l: homology classes of saddle connections}
If the couple of germs \( \{ \text{{\bf central}, {\bf short}}\} \) and \( \{ \text{{\bf central}, {\bf long}}\} \) are not the extremities of a same saddle connection, the homological classes of saddle connections after the Schiffer variation are  transformed (using the identifications of the homology groups given by the Gauss-Manin connection along the deformation as in  \eqref{eq:period coord on strata})   according to the following formulae
\[ [{\bf long}] \rightarrow [{\bf long}] - [{\bf short}] , \ [{\bf central }]\rightarrow [{\bf central }] +[{\bf short}] ,\] 
\[ [{\bf short}] \rightarrow [{\bf short}] , \text{ and } [{\bf other_k}]\rightarrow [{\bf other_k}]\] 
\end{lemma}

This formula will be useful to understand the change of marking before and after surgery.
 
 \begin{proof} 
 The Schiffer variation defined at one point can be extended to neighbouring points in the set of forms with real periods in the minimal stratum by using the twins defined by the (short) saddle connection and the corresponding (long) twin lying at angle $2\pi$. The relation between homology classes is constant up to small deformation in the set of forms with real periods in the minimal stratum as was shown via the Gauss-Manin connection in the definition of the period coordinates on strata (with values in the cohomology group  \eqref{eq:period coord on strata}). 
 
 Up to a small deformation of the lenghts of the saddle connections,  we can suppose that the period homomorphism $p$ of the form is injective. The initial ribbon graph has lengths given by 
 \[ p([{\bf long}]),\ p([{\bf central}]),\ p([{\bf short}]) \text{ and the other ones, } p([{\bf other_k}])\text{'s},\]
 The ribbon graph obtained after  the Schiffer variation along the given pair of twins (see Figure \ref{fig:schiffer_chord})  has lengths given by
 \[ p([{\bf long}])-p([{\bf short}]),\ p([{\bf central}])+p([{\bf short}]),\ p([{\bf short}]) \text{ and } p([{\bf other_k}])\text{'s},\]
 
 Since Schiffer variations preserve the periods of cycles in the homology of the punctured surface, the injectivity of the period homomorphism $p$ implies the  relation of the homology classes is the one in the statement of the lemma. 
 \end{proof}

 \section{Elliptic differentials with three simple poles and real periods}\label{s:(1,3)}

In this section we prove Theorem \ref{t:Rgn} in the case \(g=1\) and \(n=3\), which is the most difficult case, namely we establish

\begin{theorem}\label{real13}
Let $(\Sigma,\PP)$ be of type $(g,n) = (1,3)$ and \(p\in \text{Hom}(H_1(\Sigma\setminus\PP);\mathbb{R})\)   be a homomorphism with non-zero period on each peripheral class and whose image is not contained in $\mathbb{Q}\otimes p(\Pi)$. Then,  the isoperiodic set $\Omega\mrs_{\Sigma,\PP}(p)$ (fiber of  \(\per_{\Sigma,\PP}\) over $p$)   is connected. 
\end{theorem}

To prove Theorem \ref{real13} we can use Remark \ref{rem:GL2R invariance} and number the peripheral classes $\pi_1, \pi_2$ and $\pi_3$ in $H_1(\Sigma\setminus\PP)$ in such a way that 
\begin{equation}\label{eq:normalized p}p\in \text{Hom}(H_1(\Sigma\setminus\PP);\mathbb{R})\text{  satisfies } p(\pi_2)\leq p(\pi_1) <0 \text{ and } p(\pi_3)>0 .
\end{equation}

A first step consists in reducing the number of different zeros:
\begin{lemma}\label{rigid} Any elliptic form $(C,\omega)\in\Omega \mathcal M_{1,3}$ with three simple poles whose periods $\Lambda_{\omega}$ are contained in $\mathbb{R}$ but not contained in half the group of peripheral periods $$\frac{1}{2}\Pi_{\omega}$$
is isoperiodically connected to a form with a single zero. 
\end{lemma}
\begin{proof}
We claim  that if $\Lambda_\omega\not\subset\Pi_{\omega} $ then after a finite number of Schiffer variations we can reach either a form with a single zero or a stable form with one non-separating node with zero residue and no zeros outside the node. This means that the order of the form at the node is two.  The hypothesis of the Lemma is then used to join the latter to a form with a single zero. Let us prove this last statement first: after normalization of the non-separating node we get a form in $\Omega\mathcal M _{0,5}$ with precisely three poles (all simple) and two extra marked points corresponding to the node, one of which is a simple zero and the other a regular point. Its period homomorphism has image \(\Pi_\omega\). Choose a minimizing geodesic starting at the zero and ending at the other marked point. Let $L$ be its length, and follow the same  geodesic for a second distance of $L$. The minimizing property of the initial geodesic implies that the first moment where it can come back to the starting point is $2L$. By construction and the fact that  $\Lambda_{\omega}\subset \mathbb{R}$, then $\Lambda_\omega$ is the subgroup of \(\mathbb C\) generated by $L$ and \(\Pi_\omega\). If we suppose that following the geodesic for distance $2L$ we get a closed geodesic, we would have \(\Lambda_\omega\subset \frac{1}{2}\Pi_{\omega}\) contrary to assumption. Hence the other endpoint of the geodesic of length \(2L\) is a not the zero of the form. In terms of the original stable form, we have found two twins leaving a node: one is defined by the first half of the geodesic of length $2L$ and the other by the second half; the Schiffer variation along these twins leads to a form on a smooth curve (as shown in Figure \ref{fig:non_sep_shiffer}). Moreover, by construction the form has a single zero .

Let us prove the claim. 

A form in $\Omega \mathcal M_{1,3}$ can have up to  three simple zeros. Such form can be isoperiodically deformed in its stratum by using relative period coordinates to guarantee  that the pairwise distances between the zeros are different, not equal to a period of any closed curve and the integrals between the zeros are non-real numbers. Consider $\gamma$ the shortest geodesic between distinct zeros. Take the geodesic $\sigma$ leaving at angle $2\pi$ of $\gamma$.  

If its endpoint is a different saddle, then $\sigma$  is necessarily a saddle connection longer than $\gamma$. The Schiffer variation along $\gamma$ and its twin in $\sigma$  has a saddle less than the initial form (see text of Figure \ref{fig:schiffer_regular}). In this case we can restart the argument by choosing the shortest geodesic $\gamma$ between the two remaining zeros, and the saddle connection $\sigma$ as before.

If $\sigma$ is shorter than $\gamma$, then it is necessarily closed, which would imply that the integral of the form along $\gamma$ is real, contrary to assumption. The third and last possibility is that $\gamma$ and $\sigma$ have the same length and therefore join the same pair of distinct saddles.  Applying a Schiffer variation along these twins we reach a stable form on a curve with a node of zero residue (corresponding to the period of the cycle defined by the concatenation of one of the twins and the inverse of the other) It is precisely the situation that we obtain after smoothing a node. If we suppose that the cycle is separating, we reach a contradiction. Indeed, if it separates the three poles in  two groups, then the residue theorem necessarily fails on one of the sides, and if it does not separate the poles, then on one part we have a holomorphic differential with real periods, which is not possible due to the Riemann relations. Therefore, the node is non-separating. The next step is to carry a Schiffer variation to show that we can reach a form without zeros outside the node. For this aim we choose a geodesic between the remaining zero and the node. Its twin cannot end at the node, since otherwise all periods of the form would lie in $\Pi_\omega$, against assumption. An argument with exactly the same steps as when we joined the first zeros allows to conclude that we can reach a stable form with a separating node and no zeros outside the node. \end{proof}

Lemma \ref{rigid} implies that to prove Theorem \ref{real13} it suffices to connect elliptic forms with three simple poles and with a single zero (of order $3$) by isoperiodic deformations.
Denote by $ \Omega \mathcal M_{1,3}(3)$ the stratum of moduli space of elliptic differentials with a single zero of order $3$ and $ \Omega \mathcal S_{\Sigma,\PP}(3)$ its Torelli cover.

\begin{lemma} \label{types} If $(C,\omega)\in\Omega \mathcal{M}_{1,3}(3)$ has real periods and two negative peripheral periods, then the ribbon graph is one of the two possibilities of Figure \ref{ribbon}. They can be distinguished by the presence or absence of a peripheral curve among the saddle connections.
\end{lemma}
\begin{proof}
 By the Residue Theorem, $\omega$ has a positive peripheral period at the third marked point, and hence $\omega$ has precisely three poles. The ribbon graph of $\omega$ has four chords. In Figures \ref{fig:ribbon  graphs} and \ref{ribbon} we find all possible combinatorics with four chords. The only cases corresponding to genus \(1\) and \(3\) poles correspond to Figure \ref{ribbon}. They are distinguished by the fact that one has a saddle connection that is a peripheral curve, and the other not.
\end{proof}



 \begin{figure}
\includegraphics[scale=0.6]{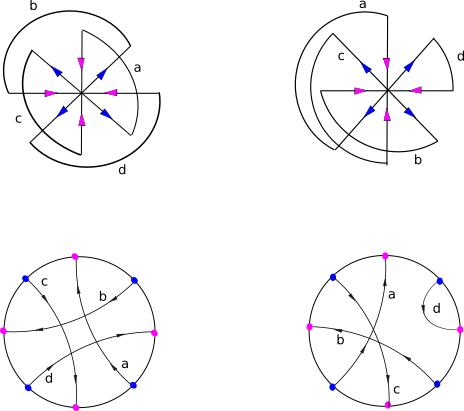}
\caption{Ribbon graphs and their corresponding chord diagrams in the case of four chords corresponding to forms of type $(1,3)$. On the left: saddle diagram and chord diagram of the butterflies (headless) form \(B_p (a,c\,|\,b,d)\) . On the right: saddle diagram and chord diagram of the octopus (head) form \(O_p (a,b,c\,|\,d)\).}
\label{ribbon}
\end{figure} 
Let us identify the homological markings of the two types of forms with a single zero in the spirit of Lemma \ref{l:marked ribbon graph} for forms with two negative peripheral periods that, for convenience, we suppose are the first two marked points:
\begin{itemize}

\item \textbf{Octopus.}  We say that a form $(C,\omega)\in \Omega \mathcal M_{1,3}(3)$ with real periods and two negative peripheral periods is of type "octopus" if its saddle connection graph and chord diagram correspond to some pair of the right in Figure \ref{ribbon}. 


Given a lift $(C,m,\omega)\in \Omega \mathcal S_{\Sigma,\PP}(3)$ we denote by $d\in H_1(\Sigma\setminus\PP)$ the peripheral class of the saddle connection (that we call the head) and $(a,b,c)$ the classes corresponding to the other saddle connections written in the cyclic ordering, following the inwards separatrices after that of $d$ (as in Figure \ref{ribbon}). From the diagrams on the right of  the same Figure we deduce: 

\begin{itemize}
\item[(i)] $a\cdot b = b\cdot c = c\cdot a =1,$ and  $a\cdot d = b\cdot d = c\cdot d = 0$;
\item[(ii)] $\{a+b+c,d\} =\{- \pi_1, -\pi_2\}$;
\item[(iii)] If $p=\per(C,m,\omega)$ we have $p(a), p(b), p(c), p(d)>0$. \\
\end{itemize}


Reciprocally, if we are given a homomorphism $p:H_1(\Sigma\setminus\PP)\rightarrow \mathbb{R}$ and elements $a,b,c,d\in H_1(\Sigma\setminus\PP)$ satisfying $(i),(ii)$ and $(iii)$, then, by Lemma \ref{l:marked ribbon graph} there exists a unique octopus form $(C,m,\omega)\in \Omega\mrs_{\Sigma,\PP}(3)$ with $p=\per(C,m,\omega)$ and whose set of cyclically ordered  homology classes of saddle connections are \(a,b,c,d\), with \(d\) being the head. This octopus  will be denoted by $O_p(a,b,c \, |\, d)$.


As we will perfom Schiffer variations along twins contained in the saddle connection graph, we will need to compare the lengths of the saddle connections. This leads us to introduce two natural subclasses of the set of marked octopodes. If $p(a+b+c) \le p(d)$ (namely, \(d= -\pi_2\), see \eqref{eq:normalized p}) we call the octopus large head octopus (denoted $LHO_p(a,b,c\,|\, d)$); otherwise (namely, \(d=-\pi_1\), see \eqref{eq:normalized p}), we call it the small head octopus (denoted $SHO_p(a,b,c\,|\, d)$).
\\

\item \textbf{Butterflies.} We say that a form $(C,\omega)\in \Omega \mathcal M_{1,3}(3)$ with a single zero, real periods and two poles with negative peripheral period  is of type "butterflies" if none of the saddle connections is a peripheral curve up to sign. By Lemma \ref{types} its chord diagram corresponds to the one on the left in Figure \ref{ribbon}.  Given a lift $(C,m,\omega)\in\Omega\mrs_{\Sigma,\PP}$, we consider the homology classes $a,b,c,d\in H_1(\Sigma\setminus \PP)$ of the saddle connections  in the cyclic order  defined by their incoming separatrices (see Figure \ref{ribbon}, left) we have the following properties:
\begin{itemize}
\item [(i)]$a\cdot b = b\cdot c = c\cdot d = d\cdot a = 1,$ and $a\cdot c = b \cdot d = 0$;
\item [(ii)]$a+c =  -\pi_1, \, b+d =- \pi_2,\, a+b+c+d = \pi_3$;
\item[(iii)] If $p=\per(C,m,\omega)$ we have $p(a), p(b), p(c), p(d)>0$. 
\end{itemize}
Notice that, as opposed to the case of octopodes, there is an indeterminacy in the choice of \( a,b,c,d \): the choice \( (a',b',c',d') = (c,d,a,b)\) is also possible. However, this is the only indeterminacy. 
 
Given \( a,b,c,d\in H_1 (\Sigma\setminus \PP)\) satisfying properties (i), (ii) and (iii), Lemma \ref{l:marked ribbon graph} shows there exists a unique 
butterflies form 
$(C,m,\omega)\in \Omega\mrs_{\Sigma,\PP}(3)$ with $p=\per(C,m,\omega)$ and whose set of cyclically ordered  homology classes of saddle connections are \(a,b,c,d\). For convenience, we denote this form  $B_p(a,c \, |\, b, d)$. \\
 \end{itemize}

 \begin{definition} \label{d: arm module}
     For any octopus $O_p(a,b,c |d)$ the $\mathbb{Z}$-module $M=\mathbb{Z}a+\mathbb{Z}b+\mathbb{Z}c\subset H_1(\Sigma\setminus\PP)$ 
     is called its arm module. This is an $E$-module with $E=\{\pi_1,\pi_3\}$ if it is small head octopus and $E=\{\pi_2,\pi_3\}$ if it is a large head octopus.  
 \end{definition}
Recall from Section \ref{ss: E-modules} that any $E$-module $M$ represents a spherical boundary component in $\mnc_{\Sigma,\PP}$ and that $\sim$ denotes the isoperiodic equivalence relation (see Corollary \ref{c:closure}). The next lemma relates octopus forms with the same arm module. 

\begin{lemma}\label{l:same arm module}
If $O_p(a,b,c\,|\,d)$ and $O_p(a',b',c'\,|\,d)$ are two octopodes having  the same arm-module $M$ (and conseqeuntly the same head), then  $$O_p(a,b,c\,|\,d)\sim O_p(a',b',c'\,|\,d).$$
\end{lemma}

\begin{proof} Firstly, for any octopus with associated $E$-module $M$ (the arm module) we first construct an isoperiodic path in $\Omega\mnc^{sep}_{\Sigma,\PP}$  converging to a point in the spherical boundary stratum $\Omega\mathcal{B}_{M}$. For this aim, choose the pair of geodesic twins in the purely imaginary  direction $i\in\mathbb{C}$ that form angles $\pi/2$ with the separatrices of the saddle connection associated to the head $d$. By construction, both twins converge to the same pole and together form a separating simple closed curve having a cylinder on one side. As shown in Figure \ref{fig:twins_2_one_pole}, the limit stable form has a separating node with a component of genus one with two poles and a single zero on one side, marked by $M$, and a component of genus zero with three poles on the other, marked by $\Pi_E$. 

This means that the limit form belongs to the boundary stratum $\Omega\mathcal{B}_{M}$ and has period $p$. Since the limit form has two poles and a single zero on the genus one part, the zero has order two. The local degree of the integral of the form is at least three at that point.  This implies that the form cannot be associated to a branched double cover of $\mathbb{P}^1$ and hence its period satisfies the conditions of Theorem \ref{2poles}. 

Secondly, by Lemma \ref{p:connected isoperiodic strata} applied using  Theorem \ref{2poles} and Lemma \ref{l: connectedness spherical case} to the parts, the set $\Omega^*\mathcal{B}_M(p)$ is connected. Therefore the two given  octopodes  lie in the same connected component of $\Omega\mnc^{sep}_{\Sigma,\PP}(p)$. By Corollary  \ref{c:closure} they also lie in the same connected component of $\Omega\mrs_{\Sigma,\PP}(p)$. 
   
\end{proof}

\begin{corollary}\label{c:change order}
$O_p(a,b,c\,|\,d)\sim O_p(c,a,b\,|\,d)\sim O_p(b,c,a\,|\,d)$. 
\end{corollary}
We improve the statement of Lemma \ref{rigid} in the following lemma:   

\begin{lemma}\label{lho}
If a form in $\Omega \mrs_{\Sigma,\PP}$ with real periods has periods outside the \(\mathbb Q\)-vector space generated by its peripheral periods, it can be isoperiodically deformed to a large head octopus form.
\end{lemma}

\begin{proof} By Lemma \ref{rigid} it suffices to prove that every form with a single zero can be isoperiodically connected to a large head octopus form. The first part of the proof is to show that every small head octopus form is isoperiodically connected to butterflies form, the second part of the proof is to show that every butterflies form is isoperiodically connected to a large head octopus form.   \\

\textit{Part 1: connecting small head octopodes to butterflies.}  

Consider a small head octopus \(SHO_p(a,b,c \, | \, d)\). 
 Firstly, consider the case when there exists a saddle connection among $a,b,c$ with period shorter than the period of the head $d$. Up to applying Corollary \ref{c:change order} we can suppose that this saddle connection is the one opposite to $d$ in the chord diagram, i.e. $c$ and $p(d)>p(c)$. The Schiffer variation using $d$ as long saddle connection, $c$ as short saddle connection and $b$ as central connection is a form with a single zero that does not have a peripheral saddle connection. It is therefore a form of type butterflies.  \\

Secondly, consider the case when the periods of $a,b$ and $c$ are all longer than the period of $d$. By the assumption of the lemma there exist two connections in the set \(\{a,b,c\}\) whose lengths are rationally independent. Thanks to Corollary \ref{c:change order} we can assume that these two connections are \(a\) and \(b\). Perform a series of Schiffer variations with \(a\) and \(b\) as long and the short saddle connections a in Lemma \ref{l: homology classes of saddle connections} (depending on their lengths) and \(c\)  the central saddle connection. Note that such Schiffer variations do not permute the connections \(\{a,b,c\}\). Moreover thanks to Lemma \ref{l: homology classes of saddle connections}, at each step one of the periods of \(a\) or \(b\) diminishes and the other stays constant and both stay strictly positive and with irrational quotient. Moreover, the sequence of points  in \(\mathbb{R}^2\) constructed in this way converges to the origin \((0,0)\). Hence, after following this sequence for long enough both saddle connections have period shorter than the head and we can apply the first case.  
 \\

\textit{Part 2: connecting butterflies to large head octopodes.} Take a form of type butterflies \(B_p(a,c \, | \, b,d)\) and assume ,without loss of generality, that $p(a+c) \ge p(b+d)$. One can reach at most 4 different octopodes in one Schiffer variation (see Figure \ref{symoct8}):
\begin{itemize}
\item using saddle connections \(a\) and \(d\). If \(p(a)<p(d)\), using \(a\) as the short connection, \(d\) as the long connection, and \(c\) as the central connection, one can reach an octopus $O_p(d-a, b,a \,|\, a+c)$ with head \((a+c)\). If \(p(d)<p(a)\), using \(d\) as the short connection, \(a\) as the long connection, and \(b\) as the central connection, one can reach an octopus $O_p(c,a-d,d \,| \, b+d)$ with head \((b+d)\);
\item using saddle connections \(a\) and \(b\). If \(p(a)<p(b)\), using \(a\) as the short connection, \(b\) as the long connection, and \(c\) as the central connection, one can reach an octopus $O_p(d,b-a,a \,| \,a+c)$ with head \((a+c)\). If \(p(b)<p(a)\), using \(b\) as the short connection, \(a\) as the long connection, and \(b\) as the central connection, one can reach an octopus $O_p(a-b,c,b \,| \,b+d)$ with head \((b+d)\);
\item using saddle connections \(b\) and \(c\). If \(p(b)<p(c)\), using \(b\) as the short connection, \(c\) as the long connection, and \(d\) as the central connection, one can reach an octopus $O_p(a,c-b,b \,| \, b+d)$ with head \((b+d)\). If \(p(c)<p(b)\), using \(c\) as the short connection, \(b\) as the long connection, and \(a\) as the central connection, one can reach an octopus $O_p(b-c,d,c \, | \, a+c)$ with head \((a+c)\);
\item using saddle connections \(c\) and \(d\). If \(p(c)<p(d)\), using \(c\) as the short connection, \(d\) as the long connection, and \(a\) as the central connection, one can reach an octopus $O_p(b,d-c,c \, | \, a+c)$ with head \((a+c)\). If \(p(d)<p(c)\), using \(d\) as the short connection, \(c\) as the long connection, and \(b\) as the central connection, one can reach an octopus $O_p(c-d,a,d \, | \, b+d)$ with head \((b+d)\).
\end{itemize}

\begin{figure}
\includegraphics[scale=.7]{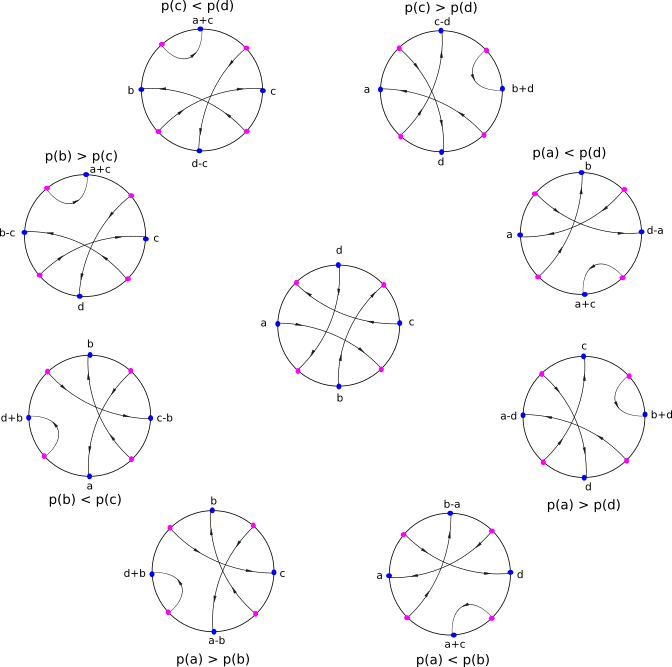}
\caption{Possible octopodes that can be reached from a butterflies form.}
\label{symoct8}
\end{figure}

There are two types of the head that we can obtain in this manner: \(a+c\) and \(b+d\), which correspond to the large head octopodes and the small head octopodes, respectively. If a large head octopus can be reached, then the statement of the lemma is confirmed. If no large head octopus can be reached using the Schiffer variation listed above, then it follows that $p(b),p(d)\le p(a),p(c)$. Note that it is not possible that all four pairs $\left(p(a), p(b)\right), \left(p(a), p(d)\right), \left(p(c), p(b)\right), \left(p(c), p(d)\right)$ are rationally dependent because of the assumption of the lemma. Without loss of generality, assume that the pair $\left(p(a), p(b)\right)$ is rationally independent.\\

By assumption, \(p(b)<p(a)\); using several Schiffer variations along the pair $a$ and $b$ as long and short connections, respectively, and \(c\) as the central connection, we reach another couple of butterflies with marking  \(B_p(a-q b,c+q b \, |\,  b,d)\) where \(q \in \mathbb Z^+\) is chosen such that \(0<p(a-q b) <p(b)\). Using the fact that \(p(a-q b)<p(b)\), we use a Schiffer variation with \(a-q b\) as the short connection, \(b\) as the long connection, and \(c\) as the central connection to go to an octopus $O_p(d, (q+1) b -a,a - q b\,  | \, a+c)$  which is a large head octopus. 
\end{proof}


Before going further, recall that the arm modules of large head octopodes are \(E=\{\pi_2,\pi_3\}\)-modules (see definition \ref{d: arm module} and subsection \ref{ss: E-modules}); namely those are the sub-modules \(M\subset H_1 (\Sigma\setminus\PP) \) that satisfy \(M \cap \Pi = \mathbb Z \pi_1\) and \(H_1 (\Sigma\setminus\PP) = M+ \Pi\). Notice that given such a \(E\)-module, the restriction of the map \(H_1(\Sigma\setminus\PP)\rightarrow H_1(\Sigma) \) induces a surjective map \( i_M : M \rightarrow H_1( \Sigma) \) whose kernel is \(\mathbb Z \pi_1\), and in particular an isomorphism between the quotient \(M/\mathbb Z\pi_1\) and the homology group \(H_1(\Sigma)\). Given any pair \(M,M'\) of \( E\)-modules, there exists a unique linear form \( \varphi \in H_1 (\Sigma) ^*\) such that \(M'= M+\varphi\), where
\begin{equation} \label{eq: M'} M+\varphi  = \{x+ \varphi (i_M (x)) \pi_2 \,|\, x\in M\}. \end{equation}
This defines a free and transitive action of \(H_1 (\Sigma)^*\) on the set of \(E\)-modules, and gives it a structure of an affine space directed by \(H_1(\Sigma)^*\). It is a particular case of the description in subsection \ref{ss: E-modules}, and in this case we can exploit the symplectic structure defined by the intersection product on $H_1(\Sigma)$:  if \(x\in M\), we will denote by \( x^* \in H_1 (\Sigma)^*\) the cohomology class defined by intersection: \( x^* ( y) = i_M(x) \cdot y  \) for all \(y\in H_1(\Sigma)\). Also given \( \varphi \in H^1 (\Sigma)^*\), we denote by \( \varphi^* \in H_1 (\Sigma) \) the cycle such that \( \varphi (x) = \varphi^* \cdot x\) for any \(x\in H_1(\Sigma)\).  

\begin{definition} Two arm modules \(M\) and \(M'\) of LHO forms are said to be (isoperiodically) connected (and denoted by \(M'\sim M\)) if some LHO with an arm module \(M\)  is isoperiodically connected to some LHO with an arm module \(M'\). 
\end{definition}

By Lemma \ref{l:same arm module}, $M\sim M'$ implies that all octopodes with arm modules $M$ or $M'$ having the same period homomorphism lie in the same isoperiodic connected component. In particular being connected is an equivalence relation in the set of arm modules of LHO's with the same period homomorphism. 

In the following lemmas we will be connecting two arm modules which should be understood as finding and connecting two representatives of each module.  

For simplicity, we will assume in the sequel that \( p (\pi_1) = -1\), condition that we can achieve by multiplying \(p\) by a positive constant  (using Remark \ref{rem:GL2R invariance}). Given a \( E=\{\pi_2, \pi_3\}\)-module \(M\), we denote by \( p_M  \in H^1 (\Sigma_1, \mathbb R/\mathbb Z) \simeq \text{Hom} (M/ \mathbb Z \pi_1, \mathbb R/ \mathbb Z)\) the morphism 
\begin{equation} \label{eq: pM} p_M (x +\mathbb Z \pi_1 ) := p(x) \, \text{mod}\, \mathbb Z  \text{ for every } x+\mathbb Z \pi_1 \in M/\mathbb Z\pi_1 .\end{equation}

By Lemmas \ref{l:same arm module} and  \ref{lho}, proving Theorem \ref{real13} is equivalent to proving that arm modules of large head octopodes are connected under the conditions of Theorem \ref{real13}.  We do it in several steps:

\begin{enumerate}
\item we show that we can connect the arm module \(M = \mathbb  Z a+\mathbb Z b +\mathbb Z c\) of the form \(LHO_p(a,b,c \, | \, d)\)   to another LHO with arm module \(M+a^*\) (Lemma~\ref{LHO1});
\item using the previous step, we prove that an arm module \(M\) of a LHO can be connected to the arm module \(M+\phi\) whenever \(p(\phi^*) \notin \mathbb Q /\mathbb Z\)  (Lemma~\ref{LHO2});
\item we conclude the proof of Theorem \ref{real13} by showing that any two LHO can be connected if the image of the period map is not contained in \(\mathbb Q \otimes p(\Pi)\)  (Lemma~\ref{final}). 
\end{enumerate}


Let us start with establishing the first step:
\begin{lemma}\label{LHO1}
Consider $LHO_p(a,b,c \, | \, d)$ with an arm module $M$ generated by $a,b,c$. Then there exists a concatenation of a finite number of Schiffer variations that connects this octopus to a LHO with arm module  $M' = M+ a^*$.
\end{lemma} 

\begin{proof} 

Since the octopus has a large head, let us perform a Schiffer variation along the head \(d\) and the opposite arm \(c\) along their ingoing twins, which results in butterflies with marking  \(B_p(a, b+c| d-c, c)\), see Lemma \ref{l: homology classes of saddle connections}. Consider the cases below - for graphical proof, see Figure~\ref{213}:

\begin{figure}
\includegraphics[scale=.65]{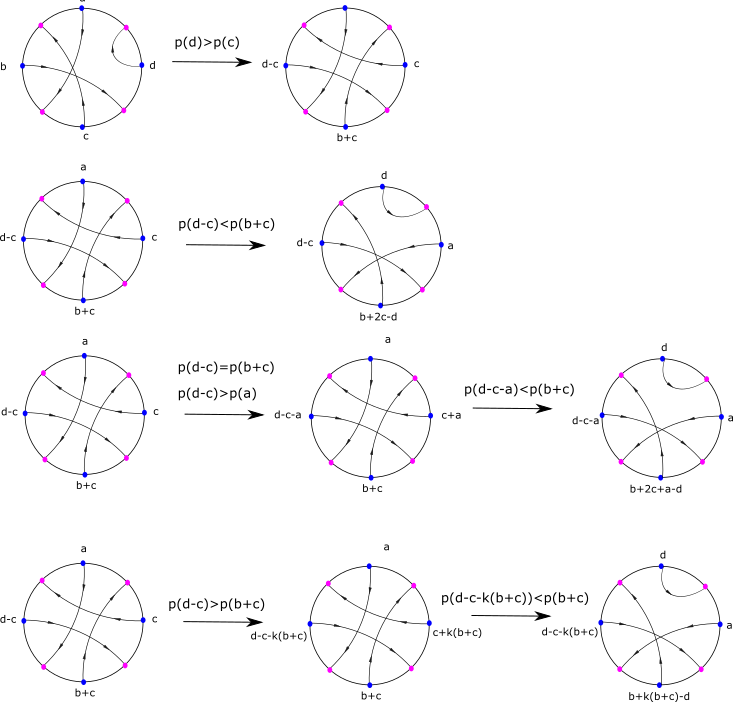}
\caption{Graphical proof of Lemma \ref{LHO1}.}
\label{213}
\end{figure}

\begin{enumerate}
\item First, let us consider the case when \(p(d)-p(c) < p(b)+p(c)\). It follows that we can perform a Schiffer variation along the outward twins \(d-c\) and \(b+c\) which  results in an octopus with a marking \(O_p(b+2c-d, a ,d-c \, | \, d)\), and arm module 
\[ \mathbb Z (b+2c-d)+\mathbb Z a +\mathbb Z (d-c) = M+ a^* . \]
 \item Consider the case \(p(d)- p(c) = p(b)+p(c)\); since \(p(d)-p(c) >p(a)\), perform a Schiffer variation along \(a\) and \(d-c\) which results in \(B_p(a, b+c \, | \, d-c-a,c+a)\). Now \(p(d)-p(c)-p(a) < p(b)+p(c)\), so as in the first case, performing the Schiffer variation along the outward twins \(d-c-a\) and \(b+c\)  leads to the large head octopus \(O_p(b+2c+a - d, a, d-c-a \, | \, d)\) whose arm module is 
 \[ \mathbb Z (b+2c+a-d ) +\mathbb Z a + \mathbb Z (d-c-a) = M+a^*.\]
\item Now consider that \(p(d)-p(c) > p(b)+p(c)\). Let us perform \(k\) Schiffer variations along the inwards twins of \(d-c\) and \(b+c\), that results in the butterfly form \(B_p(a,b+c \, | \, d-c-k (b+c), c+k (b+c))\). This can be done with \(k \in \mathbb Z^+\) is such that \(0< p(d)-p(c)-k(p(b)+p(c)) < p(b)+p(c)\). After this we can perform a Schiffer variation along the outward twins \(d-c-k (b+c)\) and \(b+c\), which results in the large head octopus \(O_p\left((k+1)(b+c) + c- d, a, d-c-k(b+c) \, |\, d\right)\), whose arm module is 
\[\mathbb Z ((k+1)(b+c) + c- d) +\mathbb Z a + \mathbb Z (d-c-k(b+c)) = M+a^*. \]
\end{enumerate}
\end{proof}


\begin{lemma}\label{LHO2}
Let \(M,M'\) be two \(\{2,3\}\)-modules such that  \( M'= M+\psi\) with \(\psi \in H_1 (\Sigma_1) ^*\). If \( p _M (\psi ^*) \notin \mathbb Q /\mathbb Z  \) then \(M\) and \(M'\) are arm modules of large head octopodes and \(M'\sim M\). 
\end{lemma}

\begin{proof}
 We first consider the case when \(\psi\) is primitive. 
 
There is a unique lift \(a\in M \) of \( \psi ^* \in M\slash \mathbb Z\pi_1 \simeq H_1(\Sigma_1) \) such that \( p (a) \in [0,1)\). Note that in fact \( p (a) \neq 0\) since \(p(a) \, \text{mod}\, \mathbb Z = p_M (\psi^*) \) is irrational. 

Let \( \beta_0 \in H_1 (\Sigma_1 )\simeq M/\mathbb Z \pi_1\) be such that \(\psi^* \cdot \beta_0 =1\). For \(k\in \mathbb Z\), let \(\beta= \beta_0+ k \psi^*\). We have \(\psi^* \cdot \beta = 1\) as well, and \( p_M (\beta)=p_M (\beta_0) + k p_M (\psi^*) \). Since \(p_M(\psi^*)\in \mathbb R/\mathbb Z\) is irrational, the set of all \( p_M (\beta) \) for \(k\in \mathbb Z\) is dense in \(\mathbb R/\mathbb Z\), hence there exists \(k\in \mathbb Z\) such that \( p_M(\beta) \in (0, 1- p(a) )\,\text{mod}\, \mathbb Z\). If \(b\in M\) is the unique lift of \(\beta\in M/\mathbb Z \pi_1\), we thus have \(p(b) \in (0, 1- p(a))\). Defining \(c\in M\) by the relation \( a+b+c = -\pi _1\), and \(d=-\pi_2\), we have that:
\begin{itemize}
\item \(a \mod \pi _1= \psi^*\);
\item \(M=\mathbb Z a +\mathbb Z b +\mathbb Z c\) 
\item \( a+b+c = - \pi_1\)
\item \(a\cdot b = b \cdot c = c\cdot a = 1\);
\item \(p(a), p(b), p(c) > 0 \);
\end{itemize}
The last three items and Lemma \ref{l:marked ribbon graph} show that there exists a unique large head octopus marked by \( LHO_p (a,b,c\,|\, d)\). The first two items and Lemma \ref{LHO1} show that \(M\) and \(M+\psi^*\) are connected, so we are done.

If \(\psi\) is not primitive, we write \( \psi = u \psi_0 \) where \(\psi_0\in H_1(\Sigma_1)^*\) is primitive and \(u\) is a positive integer. The trick here is that, by the definitions \eqref{eq: M'} and \eqref{eq: pM}, if we denote by \( \gamma= p (\pi_2)\,\text{mod}\,\mathbb Z\), we have  for any integer \(v\)
\[ p_{M+ v \psi_0} = p_M - v\psi _0 \gamma ,\] 
so that 
\[ p_{M+v\psi_0} (\psi_0^*) =p_M (\psi_0^* ) - v\psi_0 (\psi_0^*) \gamma= p_M (\psi_0^*) .  \]
In particular \( p_{M+ v \psi_0 } (\psi_0^*) \in \mathbb R/\mathbb Z\) is irrational (since its \(u\)-multiple is so). So the previous reasoning shows that \(M+ v\psi_0\) is connected to \( M+(v+1)\psi_0\) for each integer  \(v\). As a consequence 
\[M\sim M+\psi_0 \sim \ldots \sim  M+(u-1)\psi_0 \sim  M+ u \psi_0= M+\psi.\] \end{proof}





\begin{lemma}\label{final} Any two arm modules \(M\) and \(M'\) of large head octopodes are connected if the image of the period map is not contained in the rational space \(\mathbb Q \otimes p(\Pi) \) generated by the peripheral periods. \end{lemma}

\begin{proof}
Write \(M' = M + \psi\) with \(\psi \in H_1(\Sigma_1)^*\). If \( p _M (\psi ^*) \notin \mathbb Q /\mathbb Z  \), then, by Lemma \ref{LHO2}, \(M\) is connected to \(M'\), and the proof is complete. So let us now assume that \( p _M (\psi ^*) \in \mathbb Q /\mathbb Z  \). Notice that 
\[ p(H_1(\Sigma_{1,3^*})) = p (\Pi) + p (M) \]
so that, by the hypothesis of the Lemma,  \( p (M) \) is not contained in \(\mathbb Q \otimes p (\Pi) \). Hence, there exists  a form \( \varphi \in H_1(\Sigma_1)^*\) such that \( p_M (\varphi ^* ) \notin (\mathbb Q \otimes p (\Pi))/\mathbb Z\). Let us define the auxiliary \(\{\pi_2,\pi_3\}\)-module 
\[ M'' = M' +\varphi = M +\psi + \varphi. \]
We have \( p_M ((\psi + \varphi) ^* ) \notin \mathbb Q \otimes p (\Pi)/\mathbb Z\) as well, so \(M\sim M''\) by Lemma \ref{LHO2}. Now,
\[p_{M'} (\varphi^* )= p_M ( \varphi^* ) - \psi (\varphi^*) \gamma \notin \mathbb Q \otimes p(\Pi) /\mathbb Z\]
and so in particular \(p_{M'}(\varphi^*) \notin \mathbb Q / \mathbb Z\). Hence Lemma \ref{LHO2} shows that \( M' \sim M ''\). \end{proof}

\section{The case of generic real period homomorphisms}

In this section we prove Theorem \ref{t:Rgn} for pairs $(\Sigma,\PP)$ of all values of \( (g,n) \). 
The cases \(n=2\)  and \( (g,n) = (1,3) \) are already proved, see respectively Theorems \ref{2poles}, and \ref{real13}.  

Recall that a boundary stratum of codimension one of $\Omega\mnc^{sep}_{\Sigma,\PP}$ defined by a Torelli class of a separating curve that separates the surface in a component of genus zero and another of the same genus as $\Sigma$ is called a spherical boundary stratum. The type of a spherical boundary stratum is the subset of marked point $E\subset\PP$ that belong to the component of genus zero.  A spherical boundary stratum of type $E$ is called an $E$-boundary stratum.

\subsection{Isoperiodic degeneration to some spherical boundary component}

\begin{lemma} \label{l: first degeneration}
Let $(\Sigma,\PP)$ be of type \( (g,n) \) with $n\geq 3$. Any form in $\Omega\mrs_{\Sigma,\PP}$ with at least three simple poles and real periods is isoperiodically equivalent to a stable form on spherical boundary component of $\Omega\mnc_{\Sigma,\PP}$  that has two poles on the spherical part, whose peripheral period sum is non-zero. 
\end{lemma}

\begin{proof}
Let \( p \in H^1 (\Sigma\setminus\PP, \mathbb R)\) be the period homomorphism of a marked form \((C,\omega,m) \in \Omega\mrs_{\Sigma,\PP}\) with at least three simple poles and real periods. 

We denote by \(\varphi : C \setminus Q \rightarrow \mathbb R\) an integral of the imaginary part of \(\omega\), where \(Q\subset C\) is the set of poles of \(\omega\). Up to performing Schiffer variations on the zeros of \(\omega\), we can assume that the critical values of \(\varphi\) are all distinct and simple. It is a first integral for the horizontal foliation $\mathcal{G}_1$ induced by $\omega$ on $C$. Its leaves are either closed or saddle connections (see Section \ref{ss: singular flat metric and directional foliations}). Every connected component of the complement of the levels of zeros and poles of $\omega$ is a cylinder by closed geodesic leaves. Every such cylinder \(A\subset C\) of \((C,m,\omega)\) ends on the top at a unique zero/pole \(z^+(A)\) of \(\omega \) (since the critical values of \(\varphi\) are simple)  and at the bottom at a unique zero/pole \( z_- (A)\). The value \( \varphi(z_+(A)) - \varphi (z_-(A)) \) (it might be infinite if one of \( z_\pm (A)\) is a pole) is the height of the cylinder. 

Assume first that there are at least two distinct  poles with positive peripheral period (i.e. two positive poles). Each positive pole defines a semi-infinite cylinder. We claim that up to  applying a finite number of Schiffer variations, we fall in the situation where the boundary of two semi-infinite cylinders corresponding to distinct positive poles intersect. Assuming this claim, we can find a couple of twin geodesics leaving the zero, orthogonal to $\mathcal{G}_1$,  leaving the zero in the boundary of both cylinders, and joining them to the two distinct poles. The Schiffer variation along those twins leads in the limit to a point in thespherical boundary component with the desired properties (see Figure \ref{fig:schiffer_twins_2_2_poles}).

To prove the claim, define the oriented graph of horizontal cylinders as follows; its vertices are the zeros and poles of \(\omega\), and for each horizontal cylinder \(A\) we glue an oriented edge between \( z_- (A)\) and \( z_+(A)\) (see Figure \ref{fig:graph of cylinders})

 \begin{figure}
\includegraphics[scale=.7]{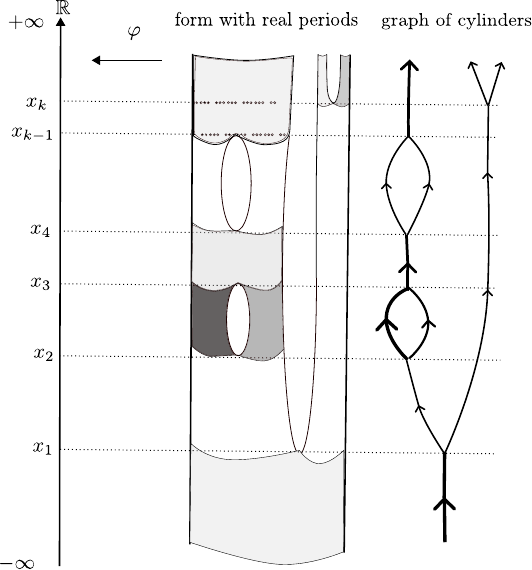}
\caption{Example of the graph of cylinders associated to a meromorphic form of genus two with four simple poles and real periods (three positive poles and one negative) }\label{fig:graph of cylinders}
\end{figure}

A path in the graph of horizontal cylinders is a sequence of oriented or anti-oriented edges such that the end point of an edge equals the starting point of the next one. Such a sequence will be denoted by \( A_1^{\varepsilon_ 1} \ldots A_s^{\varepsilon_s}\), where the \(A_k\)'s are cylinders, and the \(\varepsilon_k\)'s are sign \(\pm 1\). So we have that \( z_{\varepsilon_k}(A_k) = z_{-\varepsilon_{k+1}} (A_{k+1})\) for every~\(k\). 

Given a path in the graph linking two distinct positive poles, we call inversion a sequence of two consecutive edges of the paths with opposite orientations, namely two consecutive cylinders \(A_k\) and \(A_{k+1}\) such that \( \varepsilon_{k+1}= -\varepsilon_k\). Since the extremities of the path are positive poles, the path begins in the negative direction, and ends in the positive one. So the number of inversions is odd. 

Take a path in the graph between two  distinct positive poles having the minimal number of inversions. We claim that this number is one. Indeed, if not, there would exist two consecutive edges, the first one being positive and the second negative. At the vertex in between the two edges starts a positive path going to some positive pole. The concatenation of this latter with one of the part of the original path before or after the vertex leads to a path between distinct positive poles with fewer inversions, leading to a contradiction. Such a path will be of the form 
\[ A_1 ^{-1} \ldots A_r^{-1} A_{r+1}\ldots A_s.\] 

We are now going to reduce the number \(s\) of edges of such a one-inversion path to two, by enabling to perform Schiffer variations on \((C,m,\omega)\). First, observe that if a one-inversion path is not injective, one can reduce its number of edges by cutting the part in between two equal values. In particular, such a path passes through a pole only at its extremities. 

Assume that there are at least three edges in such an injective one-inversion path. 
 One of the  cylinders \( A_r\) and \( A_{r+1}\) might be of infinite height, but not both. Moreover, since the function \(\varphi\) has distinct critical values, the height of both cylinders is distinct.  Take the one with smaller height; assume for simplicity that this is \(A_{r+1}\). 
 Let \(\gamma\) be a geodesic segment starting at \(p=z_-(A_r)=z_-(A_{r+1})\)  which goes across \(A_{r+1}\), crosses the intersection between the common boundaries of \( A_{r+1}\) and \( A_{r+2}\) in a regular point, and enters in \(A_{r+2}\) by walking a small distance. Notice that the direction of the geodesic \(\gamma\) has positive imaginary part. Let \(\gamma'\) be the twin of \(\gamma\) at \(p\); \(\gamma'\) enters \(A_r\) and stays in the interior of \( A_r\) appart at its origin \(p\). Perform the Schiffer variation using the pair of twins \(\{\gamma, \gamma'\}\). Appart from the union \( \gamma\cup \gamma'\) the differential \((C',m',\omega')\) after the surgery is isomorphic to \( (C, m, \omega)\). In particular, the cylinders \( A_1,\ldots, A_{r-1}, A_{r+3},\ldots, A_s\) are not affected by the Schiffer variation: the cylinders they define in \(C'\) are denoted by \( A_1 ',\ldots, A_{r-1}', A_{r+3}',\ldots, A_s'\). 
 
 \begin{figure}
\includegraphics[scale=.8]{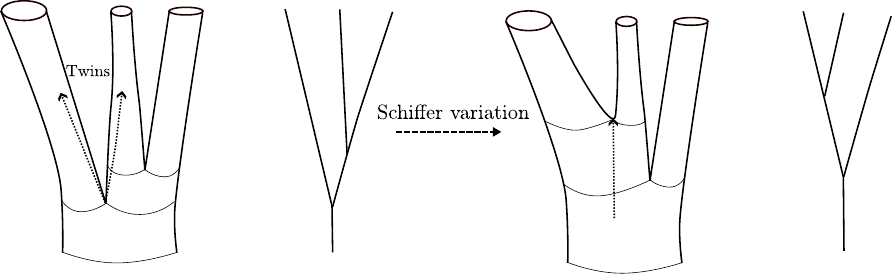}
\caption{Change in cylinder structure induced by Schiffer variation along a pair of twins going to distinct poles}\label{fig:schiffer on diagrams}
\end{figure}
 
 Denote by \( A _r ' \subset C' \) (resp. \( A_{r+2} ' \subset C'\))   the cylinder corresponding to the top of \( A_r \) (resp. \(A_{r+2}\)) under the partial identification between \(C\) and \(C'\). The two cylinders \( A_r '\) and \( A_{r+2}'\) intersect in one point (corresponding to the ends of \(\gamma\) and/or \(\gamma'\)). So in the graph of horizontal cylinders of \(C'\), the path \(A_1', \ldots, A_r', A_{r+2}', \ldots, A_s'\) links two distinct poles of \(\omega'\); its number of edges is reduced by one. 
 
 After a finite number of Schiffer variations of this type, we end at a marked differential satisfying the claim.    

The case where there are at least two distinct negative poles is analogous. 
\end{proof}


\subsection{Connecting various types of spherical boundary components}

\begin{lemma}\label{l:second degeneration}
Let $(\Sigma,\PP)$ be of type $(g,n)$ with \(n\geq 3\) and  \( p \in H^1 (\Sigma\setminus\PP, \mathbb R)\) be such that \( p (H_1(\Sigma\setminus\PP, \mathbb Z))\) is not contained in \(\mathbb Q \otimes p( \Pi)\). Moreover assume that the connectedness Theorem \ref{t:Rgn} is true for any pair $(\Sigma',\PP')$ of type $(h,m)$ with $m< n$. Then, given \(F\subset E\subset \PP\) with \(E\) of cardinality \(n-1\), any form of period $p$ in an \(F\)-boundary component can be isoperiodically connected to some form on a \(E\)-boundary component. 
\end{lemma}

\begin{proof}
Denote by $r_n$ the pole of $\PP\setminus E$, $F=\{r_1,\ldots,r_{k}\}$ and $E\setminus F=\{r_{k+1},\ldots,r_{n-1}\}$ with $k<n-1$ . Consider a form $\omega$ in the \(F\)-boundary stratum that  corresponds to the data of a \(F\)-module \(M_F\), namely a decomposition 
\[ H_1 (\Sigma\setminus\PP) =\Pi_F + M_F\text{ with } M_F\cap \Pi_F = \mathbb Z\pi_F .\]
See section \ref{d: E-module}. 
The associated $\PP$-labelled graph is: 
\begin{center}
\begin{tikzpicture}[node distance = {20mm},thick, main/.style = {draw,circle}]
\node[main] (1) [label=below:$0$] {$\Pi_F$};
\node[main] (2) [right of = 1][label=below:$g$]{$M_F$};
\node (3) [above left of = 1]{$\pi_{r_1}$};
\node (4) [below left of =1]{$\pi_{r_k}$};
\node (5) [left of =1]{$\vdots$};
\node (6) [above of =2]{$\pi_{r_{k+1}}$};
\node (7) [right of =2]{$\pi_{r_{n-1}}$};
\node (8) [above right of = 2]{$\vdots$};
\node (9) [below right of =2]{$\pi_{r_n}$};

\draw (2) -- (9);
\draw (2) -- (8);
\draw (2) -- (7);
\draw (2) -- (6);
\draw (1) -- (2);
\draw (1) -- (3);
\draw (1) -- (4);
\draw (1) -- (5);

\end{tikzpicture}
\end{center}

Remark that $M_F$ contains as many peripheral classses as points in $\PP\setminus F$-- i.e. at most $n-1$--  and that $p(M_F)$ is not contained in the rational space generated by those classes, since otherwise the map $p$ has image contained in $\mathbb{Q}\otimes p(\Pi)$, contrary to assumption. 
We can find a \(E\)-module \(M_E\) that is contained in \( M_F\) and we have $M_F=M_E+\Pi_{E\setminus F}$. By construction, $\pi_{r_n}\in M_E$ and the image $p(M_E)$ cannot be contained in $\mathbb{Q}p(\pi_{r_n})$ by the same reason as before.  The following graph is therefore $p$-admissible: 
\begin{center}
\begin{tikzpicture}[node distance = {20mm},thick, main/.style = {draw,circle}]
\node[main] (1) [label=below: $0$]{$\Pi_F$};
\node[main] (2) [right of = 1][label=below: $0$]{$\Pi_{E\setminus F}$};
\node[main] (10) [right of =2][label=below: $g$]{$M_E$};
\node (3) [above left of = 1]{$\pi_{r_1}$};
\node (4) [below left of =1]{$\pi_{r_k}$};
\node (5) [left of =1]{$\vdots$};
\node (6) [above left of =2]{$\pi_{r_{k+1}}$};
\node (7) [above of =2]{$\cdots$};
\node (8) [above right of = 2]{$\pi_{r_{n-1}}$};
\node (9) [right of =10]{$\pi_{r_n}$};
\draw (10) -- (9);
\draw (2) -- (8);
\draw (2) -- (10);
\draw (2) -- (7);
\draw (2) -- (6);
\draw (1) -- (2);
\draw (1) -- (3);
\draw (1) -- (4);
\draw (1) -- (5);

\end{tikzpicture}
\end{center}
\noindent thus representing a point in the associated boundary stratum with two nodes and period $p$. 

The inductive hypothesis applied to the restriction of \(p_{|M_F}\) and Lemma \ref{p:connected isoperiodic strata} show that we can isoperiodically connect the two points in the boundary depicted in the previous two diagrams. 

Next, by smoothening the node beween the two genus zero parts, we obtain a stratum with diagram
\begin{center}
\begin{tikzpicture}[node distance = {20mm},thick, main/.style = {draw,circle}]
\node[main] (1) [label=below: $0$]{$\Pi_E$};
\node[main] (2) [right of = 1][label=below: $g$]{$M_E$};
\node (3) [above left of = 1]{$\pi_{r_1}$};
\node (4) [below left of =1]{$\pi_{r_{n-1}}$};
\node (5) [left of =1]{$\vdots$};
\node (6) [right of =2]{$\pi_{r_n}$};

\draw (2) -- (6);
\draw (2) -- (6);
\draw (1) -- (2);
\draw (1) -- (3);
\draw (1) -- (4);
\draw (1) -- (5);

\end{tikzpicture}
\end{center}

\noindent which corresponds to a point in a $E$-boundary component.
\end{proof}

In the sequel, we will be interested in the \(E\)-boundary components with \(E\) of cardinality \(n-1\).

\begin{lemma}\label{l: second degeneration}
Let $(\Sigma,\PP)$ be of type $(g,n)$ with \(n\geq 3\), and \( p \in H^1(\Sigma\setminus\PP, \mathbb R)\) be such that \( p (H_1(\Sigma\setminus\PP, \mathbb Z))\) is not contained in \(\mathbb Q \otimes p( \Pi)\). Given \(E, E'\subset \PP\) two subsets of cardinality \(n-1\), any point of period $p$ in a \(E\)-boundary component can be isoperiodically connected to some \(E'\)-boundary component.   
\end{lemma}

\begin{proof}
The proof has two cases, the first when \(n=3\), and the second when \( n>3\).

\vspace{0.4cm}

{\bf First case: \(n=3\).} In this case, the proof is by induction on the genus. If \(g=0\) or \(g=1\), the Lemma follows from the fact that the isoperiodic moduli spaces \(\Omega\mrs_{\Sigma,\PP}(p)= \text{Per}_{\Sigma,\PP}^{-1} (p) \) are connected (Lemma \ref{l: connectedness spherical case} if \(g=0\) and Theorem \ref{real13} if \(g=1\) under the genericity assumption on \(p\)). By Corollary \ref{c:closure},  $\Omega^{*}_{0}\mnc_{\Sigma,\PP}(p)$ is connected and contains all $E$-boundary components (actually with any cardinality for $E$ but in particular for $E$ of cardinality $n-1$).
Next assume that \(g\geq 2\) and the inductive hypothesis up to genus $g-1$. Denote the poles by $\{i,j,k\}$ and write $E=\{i,j\}$ and $E'=\{j,k\}$. Let us introduce the \(R\)-labelled
 graph associated to the \(E\)-boundary component as follows: 
\begin{center}
\begin{tikzpicture}[node distance = {20mm},thick, main/.style = {draw,circle}]
\node[main] (1) [label=below: $0$]{$\Pi_E$};
\node[main] (2) [label=below: $g$][right of = 1]{$W_2$};
\node (3) [above left of = 1]{$\pi_{i}$};
\node (4) [below left of =1]{$\pi_{j}$};
\node (5) [right of =2]{$\pi_{k}$};
\draw (2) -- (5);
\draw (1) -- (2);
\draw (1) -- (3);
\draw (1) -- (4);


\end{tikzpicture}
\end{center}
The part of genus $g$ has two simple poles and we can apply the results of \cite{CD}. The hypothesis on $p$ implies that the composition of $p_{W_2}:W_2\rightarrow \mathbb{R}$ with the quotient map $\mathbb{R}\rightarrow \mathbb{R}/(\mathbb{Q}\otimes p(\Pi_\PP))$  is non-zero. Corollary 3.3 of \cite{CD}  implies that we can find a decomposition $W_2=W_3+W_4$ so that $W_3$ (resp. $W_4$) is isomorphic to the homology of a genus one (resp. $g-1$) surface with two punctures and neither $p(W_3)$ nor $p(W_4)$ are contained in $\mathbb{Q}p(\pi_i)+\mathbb{Q}p(\pi_j)+\mathbb{Q}p(\pi_k)$. Lemma \ref{p:connected isoperiodic strata} implies that there is a non-empty boundary stratum of forms of period $p$ associated to the $\PP$-labelled graph: 
\begin{center}
\begin{tikzpicture}[node distance = {20mm},thick, main/.style = {draw,circle}]

\node[main] (1) [label=below: $0$]{$\Pi_E$};
\node[main] (2) [right of = 1][label=below: $1$]{$W_3$};
\node[main] (3) [right of= 2][label=below: $g-1$]{$W_4$};
\node (4) [above left of = 1]{$\pi_{i}$};
\node (5) [below left of =1]{$\pi_{j}$};
\node (6) [right of =3]{$\pi_{k}$};
\draw (3) -- (6);
\draw (1) -- (2);
\draw (2) -- (3);
\draw (1) -- (4);
\draw (1) -- (5);

\end{tikzpicture}
\end{center}
\noindent that is connected to the initial isoperiodic stratum by smoothing the node between the genus 1 and the genus $g-1$ component. On the other hand, by smoothing the node joining the genus zero and genus one components we fall in a stratum determined by the following graph 

\begin{center}
\begin{tikzpicture}[node distance = {20mm},thick, main/.style = {draw,circle}]

\node[main] (1) [label=below: $1$]{$W_5$};
\node[main] (2) [right of = 1][label=below: $g-1$]{$W_4$};
\node (3) [above left of = 1]{$\pi_{i}$};
\node (4) [below left of =1]{$\pi_{j}$};
\node (5) [right of =2]{$\pi_{k}$};
\draw (2) -- (5);
\draw (1) -- (2);
\draw (1) -- (3);
\draw (1) -- (4);

\end{tikzpicture}
\end{center}
still with the property that $p(W_5)$ is not contained in $\mathbb{Q}p(i)+\mathbb{Q}p(j)$. By the case of $(g,n)=(1,3)$ applied to $p_{|W_5}$ and Lemma \ref{l: attaching map}, the latter boundary stratum is connected to some boundary stratum of the form 
\begin{center}
\begin{tikzpicture}[node distance = {20mm},thick, main/.style = {draw,circle}]

\node[main] (1) [label=below: $1$]{$W_6$};
\node[main] (2) [right of = 1][label=below: $0$]{$W_7$};
\node[main] (3) [right of= 2][label=below: $g-1$]{$W_4$};
\node (4) [left of = 1]{$\pi_{i}$};
\node (5) [above of = 2]{$\pi_{j}$};
\node (6) [right of =3]{$\pi_{k}$};
\draw (3) -- (6);
\draw (1) -- (2);
\draw (2) -- (3);
\draw (1) -- (4);
\draw (2) -- (5);

\end{tikzpicture}
\end{center}
\noindent with $W_5=W_6+W_7$ and  $p(W_6)$ not contained in $\mathbb{Q}p(\pi_ i)$.   By smoothing the node between the part of genus $g-1$ and the part of genus zero, the latter stratum is isoperiodically connected to  the stratum:
\begin{center}
\begin{tikzpicture}[node distance = {20mm},thick, main/.style = {draw,circle}]
\node[main] (1) [label=below: $1$]{$W_6$};
\node[main] (2) [right of = 1][label=below: $g-1$]{$W_8$};
\node (3) [left of = 1]{$\pi_{i}$};
\node (4) [above right of =2]{$\pi_{j}$};
\node (5) [below right of =2]{$\pi_{k}$};
\draw (2) -- (5);
\draw (1) -- (2);
\draw (1) -- (3);
\draw (2) -- (4);


\end{tikzpicture}
\end{center}

\noindent  where $W_8=W_7+W_4$ satisfies that $p(W_8)$ is not contained in \(\mathbb{Q}p(\pi_ i)+\mathbb{Q}p(\pi_j)+\mathbb{Q}p(\pi_k)\) (this was already true for the submodule $W_4$). We can apply the case of three poles to \(p_{W_8}\) and isoperiodically connect the latter stratum to a stratum of type:
\begin{center}
\begin{tikzpicture}[node distance = {20mm},thick, main/.style = {draw,circle}]

\node[main] (1) [label=below: $1$]{$W_6$};
\node[main] (2) [right of = 1][label=below: $g-1$]{$W_{9}$};
\node[main] (3) [right of= 2][label=below: $0$]{$\Pi_{E'}$};
\node (4) [left of = 1]{$\pi_{i}$};
\node (5) [above right of =3]{$\pi_{j}$};
\node (6) [below right of =3]{$\pi_{k}$};
\draw (3) -- (6);
\draw (1) -- (2);
\draw (2) -- (3);
\draw (1) -- (4);
\draw (3) -- (5);

\end{tikzpicture}
\end{center}
Again smoothing the node between the parts corresponding to $W_6$ and $W_{9}$ and writing $W_{10}=W_6+W_9$ we obtain an isoperiodic connection with a stratum of the form:
\begin{center}
\begin{tikzpicture}[node distance = {20mm},thick, main/.style = {draw,circle}]

\node[main] (1) [label=below: $g$]{$W_{10}$};
\node[main] (2) [right of= 1][label=below: $0$]{$\Pi_{E'}$};
\node (4) [left of = 1]{$\pi_{i}$};
\node (5) [above right of =2]{$\pi_{j}$};
\node (6) [below right of =2]{$\pi_{k}$};
\draw (2) -- (6);
\draw (1) -- (2);
\draw (1) -- (4);
\draw (2) -- (5);

\end{tikzpicture}
\end{center}

\noindent which corresponds to a boundary stratum of type \(E'=\{j,k\}\).

{\bf Second case: \(n>3\).} Write \(E=\PP\setminus\{j\}\) and \(E'=\PP\setminus\{i\}\). Denote $\hat{\PP}=\{\pi_l:l\in \PP\setminus\{i,j\}\}$, which contains at least two elements. The $\PP$-labelled graph of an element of the $E$-boundary is of the form:
\begin{center}
\begin{tikzpicture}[node distance = {20mm},thick, main/.style = {draw,circle}]
\node[main] (1) [label=below: $0$]{$\Pi_E$};
\node[main] (2) [right of = 1][label=below: $g$]{$W_2$};
\node (3) [above of = 1]{$\pi_{i}$};
\node (4) [left of =1]{$\hat{\PP}$};
\node (5) [right of =2]{$\pi_{j}$};
\draw (2) -- (5);
\draw (1) -- (2);
\draw (1) -- (3);
\draw (1) -- (4);


\end{tikzpicture}
\end{center}
\noindent where the edge between $\hat{\PP}$ and the vertex represents an edge for each element of $\hat{\PP}$ (that is at least two edges). By the connectedness of isoperiodic sets in genus $0$ and the fact that in genus $0$ with at least three poles there are no restrictions to realize any homomorphism, the previous stratum can be isoperiodically connected to a stratum of the form: 
\begin{center}
\begin{tikzpicture}[node distance = {20mm},thick, main/.style = {draw,circle}]

\node[main] (1) [label=below: $0$]{$W_3$};
\node[main] (2) [right of = 1][label=below: $0$]{$W_4$};
\node[main] (3) [right of= 2][label=below: $g$]{$W_2$};
\node (4) [left of = 1]{$\hat{\PP}$};
\node (5) [above of = 2]{$\pi_{i}$};
\node (6) [right of =3]{$\pi_{j}$};
\draw (3) -- (6);
\draw (1) -- (2);
\draw (2) -- (3);
\draw (1) -- (4);
\draw (2) -- (5);

\end{tikzpicture}
\end{center}
By smoothing the node between the genus $g$ component and the genus $0$ component, we isoperiodically connect the previous to a stratum of the form 
\begin{center}
\begin{tikzpicture}[node distance = {20mm},thick, main/.style = {draw,circle}]

\node[main] (1) [label=below: $0$]{$W_{3}$};
\node[main] (2) [right of= 1][label=below: $g$]{$W_5$};
\node (4) [left of = 1]{$\hat{\PP}$};
\node (5) [above right of =2]{$\pi_{i}$};
\node (6) [below right of =2]{$\pi_{j}$};
\draw (2) -- (6);
\draw (1) -- (2);
\draw (1) -- (4);
\draw (2) -- (5);

\end{tikzpicture}
\end{center}
We claim that \(p(W_{5})\) is not contained in \(\mathbb{Q}p(\pi_i)+\mathbb{Q}p(\pi_j)\). Indeed, if it were, then, since $p(W_3)\subset p(\Pi)$, the image of $p$ would be contained in \(\mathbb{Q}\otimes p(\Pi)\), contrary to assumption. Thanks to Lemmas \ref{l: first degeneration} and \ref{l:second degeneration},  the component of genus $g$ with three poles can be  isoperiodically deformed  to a spherical boundary stratum characterized by a sum $W_5=W_6+W_7$ to obtain an isoperiodic deformation of the latter stratum to one of type 
\begin{center}
\begin{tikzpicture}[node distance = {20mm},thick, main/.style = {draw,circle}]

\node[main] (1) [label=below: $0$]{$W_3$};
\node[main] (2) [right of = 1][label=below: $0$]{$W_6$};
\node[main] (3) [right of= 2][label=below: $g$]{$W_7$};
\node (4) [left of = 1]{$\hat{\PP}$};
\node (5) [above of = 2]{$\pi_{j}$};
\node (6) [right of =3]{$\pi_{i}$};
\draw (3) -- (6);
\draw (1) -- (2);
\draw (2) -- (3);
\draw (1) -- (4);
\draw (2) -- (5);

\end{tikzpicture}
\end{center}
Finally, by observing that $\Pi_{E'}=W_3+W_6$,  a smoothing of the node between the two genus 
zero components to obtain an isoperiodic connection with the stratum of the form 
\begin{center}
\begin{tikzpicture}[node distance = {20mm},thick, main/.style = {draw,circle}]

\node[main] (1) [label=below: $0$]{$\Pi_{E'}$};
\node[main] (2) [right of = 1][label=below: $g$]{$W_7$};
\node (3) [above of = 1]{$\pi_{j}$};
\node (4) [left of =1]{$\hat{\PP}$};
\node (5) [right of =2]{$\pi_{i}$};
\draw (2) -- (5);
\draw (1) -- (2);
\draw (1) -- (3);
\draw (1) -- (4);


\end{tikzpicture}
\end{center}
\noindent which corresponds to a a form on a $E'$-boundary component.\end{proof}

\subsection{Connecting isoperiodic \(E\)-boundary strata}
 
\begin{lemma}\label{l: connecting E-module I}
Assume $\Sigma$ is of genus \(g\geq 2\) and $\PP$ has $n=3$ points. Let \(p\in H^1 (\Sigma\setminus\PP, \mathbb R) \) be a period such that \( p (H_1(\Sigma\setminus\PP, \mathbb Z) ) \) is not contained in \( \mathbb Q\otimes p(\Pi)\). Let  \(E\subset \PP\) a subset of cardinality \(2\). Then, for any $E$-module $M$,  we have that \( \mathcal{B}_M(p)\)  is non-empty and connected, and all such \(E\)-boundary strata belong to the same connected component of \(\Omega\mnc_{\Sigma,\PP}(p)\). 
\end{lemma}

\begin{proof}
Let \( M\) be a \(E\)-module, and denote by \(i\) the pole which does not belong to \(E=\{j,k\}\). Observe that the restriction of \(p\) to \(M\) does not have image in \(\mathbb Q p(\pi_i)\) by our assumption on \(p\). 
Consider the $\PP$-labelled graph associated to elements of the boundary stratum \(\mathcal B_M(p)\):
\begin{center}
\begin{tikzpicture}[node distance = {20mm},thick, main/.style = {draw,circle}]
\node[main] (1) [label=below: $0$]{$\Pi_E$};
\node[main] (2) [right of = 1][label=below: $g$]{$M$};
\node (3) [above left of = 1]{$\pi_{k}$};
\node (4) [below left of =1]{$\pi_{j}$};
\node (5) [right of =2]{$\pi_{i}$};
\draw (2) -- (5);
\draw (1) -- (2);
\draw (1) -- (3);
\draw (1) -- (4);


\end{tikzpicture}
\end{center}
\noindent The first part of the Lemma is due to the connectedness result in each part: the genus zero case in Lemma \ref{l: connectedness spherical case} for the part corresponding to $\Pi_E$ and Theorem \ref{2poles} applied to the restriction of \(p\) to \(M\) (that necessarily takes some values that are not rationally dependent to \( p(\pi_E)\)). Lemma \ref{p:connected isoperiodic strata} then guarantees that the isoperiodic boundary stratum \(\mathcal B_M(p)\) associated to the $E$-boundary component defined by $M$ is connected. 

For the second part, observe that, thanks to the results in \cite{CD}, there exists a symplectic decomposition of the quotient \(E\)-module 
\(\overline{M}= \overline{H_1}\oplus\overline{H_2}\), with \(\overline{H_1}\) of genus \(g-1\) and \(\overline{H_2}\) of genus \(1\), such that the restriction of \(p\) modulo \(\mathbb Q p(\pi_i)\) to each of the summands does not vanish. We denote by \(H_1\) (resp. \(H_{2}\))  the submodule of \(M\) which is the pre-image of \(\overline{H_1}\) (resp. \(\overline{H_2}\)) by the quotient map \(M\rightarrow \overline{M}\). 

Let \(\overline{\psi} \in H^1 (\Sigma \setminus \{i\},  \Pi_E) \simeq H^1 (\Sigma,  \Pi_E) \simeq \text{Hom} (\overline{M}, \Pi_E) \) be a form that vanishes on \(\overline{H_1}\),  and denote \(M'= M+\psi\) where \(\psi\) is the composition of the projection \(M\rightarrow \overline{M}\) with \(\overline{\psi}\). We claim that  that we can isoperiodically connect \(\mathcal B_M(p)\) to \(\mathcal B_{M'}(p)\) as follows:

From Theorem \ref{2poles} we can isoperiodically connect the stratum \(\mathcal B_M(p)\) to the one corresponding to the $\PP$-labelled graph: 
\begin{center}
\begin{tikzpicture}[node distance = {20mm},thick, main/.style = {draw,circle}]
\node[main] (1) [label=below: $0$]{$\Pi_E$};
\node[main] (2) [right of = 1][label=below: $1$]{$H_2$};
\node[main] (3) [right of= 2][label=below: $g-1$]{$H_1$};
\node (4) [above left of = 1]{$\pi_{k}$};
\node (5) [below left of =1]{$\pi_{j}$};
\node (6) [right of =3]{$\pi_{i}$};
\draw (3) -- (6);
\draw (1) -- (2);
\draw (2) -- (3);
\draw (1) -- (4);
\draw (1) -- (5);


\end{tikzpicture}
\end{center}
By smoothing the node between the part of genus $0$ and the part of genus $1$ we connect the previous stratum to one of the form: 
\begin{center}
\begin{tikzpicture}[node distance = {20mm},thick, main/.style = {draw,circle}]
\node[main] (1) [label=below: $1$]{$W_3$};
\node[main] (2) [right of = 1][label=below: $g-1$]{$H_1$};
\node (3) [above left of = 1]{$\pi_{k}$};
\node (4) [below left of =1]{$\pi_{j}$};
\node (5) [right of =2]{$\pi_{i}$};
\draw (2) -- (5);
\draw (1) -- (2);
\draw (1) -- (3);
\draw (1) -- (4);


\end{tikzpicture}
\end{center}
\noindent where $W_3=\Pi_E+H_2$. The restriction $p_{|W_3}$ are the periods of a form of genus $1$ with three poles and some period not contained in $\mathbb{Q}p(\pi_j)+\mathbb{Q}p(\pi_k)$. The module $H_3:=H_2+\psi_{|H_2}\subset W_3$ is an $E$-module for $p_{W_3}$ with $W_3=H_3+\Pi_E$. By Theorem \ref{real13} applied to $p_{|W_3}$ we can connect the  stratum associated to last diagram to one of type :
\begin{center}
\begin{tikzpicture}[node distance = {20mm},thick, main/.style = {draw,circle}]
\node[main] (1) [label=below: $0$]{$\Pi_E$};
\node[main] (2) [right of = 1][label=below: $1$]{$H_3$};
\node[main] (3) [right of= 2][label=below: $g-1$]{$H_1$};
\node (4) [above left of = 1]{$\pi_{k}$};
\node (5) [below left of =1]{$\pi_{j}$};
\node (6) [right of =3]{$\pi_{i}$};
\draw (3) -- (6);
\draw (1) -- (2);
\draw (2) -- (3);
\draw (1) -- (4);
\draw (1) -- (5);


\end{tikzpicture}
\end{center}
Remark that, since $\psi$ is zero in restriction to $H_1$ we have that $M'=M+\psi= (H_{2}+H_1)+\psi=(H_{2}+\psi_{|H_2})+H_1=H_3+ H_1$. So smoothing the node between the component of genus 1 and the component of genus $g-1$ of the last diagram we isoperiodically connect last diagram with the stratum associated to the graph: 
\begin{center}
\begin{tikzpicture}[node distance = {20mm},thick, main/.style = {draw,circle}]
\node[main] (1) [label=below: $0$]{$\Pi_E$};
\node[main] (2) [right of = 1][label=below: $g$]{$M'$};
\node (3) [above left of = 1]{$\pi_{k}$};
\node (4) [below left of =1]{$\pi_{j}$};
\node (5) [right of =2]{$\pi_{i}$};
\draw (2) -- (5);
\draw (1) -- (2);
\draw (1) -- (3);
\draw (1) -- (4);


\end{tikzpicture}
\end{center}
\noindent which represents  precisely the stratum $\mathcal{B}_{M'}(p)$ as desired. 

To end the proof of the Lemma, fix a symplectic decomposition \(H_1(\Sigma_g,\mathbb Z) = \overline{H_1}+\ldots+\overline{H_g}\) as a sum of genus one submodules, and   notice that any form $\psi$ in \( H^1 (\Sigma_g, \Pi_E)\) is a sum of forms that vanish on all but one of the \(\overline{H_r}\)'s. 
\end{proof}

\begin{lemma}\label{l: connecting E-module II}
Assume \(n\geq 4\) is the cardinality of $\PP\subset \Sigma$  and \(g\geq 1\) is the genus of $\Sigma$. Let \(p\in H^1(\Sigma\setminus\PP, \mathbb R) \) be a period with non zero peripheral periods and such that \( p (H_1(\Sigma\setminus\PP, \mathbb Z) ) \) is not contained in  \(\mathbb{Q}\otimes p(\Pi)\). Let  \(E\subset \PP\) a subset of cardinality \(n-1\) and $M$ be an $E$-module. Then \( \mathcal{B}_M(p)\)  is non empty and connected, and all such isoperiodic \(E\)-boundary strata belong to the same connected component of \(\Omega\mnc_{\Sigma,\PP}(p)\). 
\end{lemma}

\begin{proof} 
 As in the first part of Lemma \ref{l: connecting E-module I}, the proof of the first part of the present lemma is due to Theorem \ref{2poles} applied to  the restriction of \(p\) to an $E$-module \(M\) which takes some values that are not rationally dependent to \( p(\pi_E)=p(i)\). Lemma \ref{p:connected isoperiodic strata} then guarantees that the isoperiodic boundary stratum \(\mathcal B_M(p)\) associated to the $\Pi_E$-boundary component defined by $M$ is connected. 
 
 For the second part, let \( M\) be a \(E\)-module, and denote by \(i\) the pole which does not belong to \(E\), i.e. \(E=\PP\setminus i\). It suffices to connect \(\mathcal B_M (p) \) to \(\mathcal B_{M'} (p)  \) with \(M' = M + \psi \pi_j\) for each \( j\in E\) and each \(\psi \in H^1 (\Sigma \setminus \{i\}, \mathbb Z)\simeq H^1 (\Sigma, \mathbb Z) \). Denote $\widehat{E}:=\{\pi_e: e\in E\setminus{j}$\}. The following steps prove the Lemma: 
 
 \vspace{0.2cm} 
  \textit{Step 0:} We consider the $\PP$-labelled graph associated to \(\mathcal B_M (p) \) :
  \begin{center}
\begin{tikzpicture}[node distance = {20mm},thick, main/.style = {draw,circle}]
\node[main] (1) [label=below: $0$]{$W_1$};
\node[main] (2) [right of = 1][label=below: $g$]{$M$};
\node (3) [above left of = 1]{$\pi_{j}$};
\node (4) [below left of =1]{$\widehat{E}$};
\node (5) [right of =2]{$\pi_{i}$};

\draw (2) -- (5);
\draw (1) -- (2);
\draw (1) -- (3);
\draw (1) -- (4);

\end{tikzpicture}
 \end{center}
 where we interpret the edge between $\widehat{E}$ and the vertex as a union of edges joining each element of $\widehat{E}$ with the component. 
 
 \textit{Step 1:} The part of genus zero  with $n$ poles can be isoperiodically degenerated to two parts of genus zero, one containing $E\setminus\{j\}$ and the other containing $j$ and the other polar class. In this way we reach a stratum associated to a graph of the form: 
  \begin{center}
 \begin{tikzpicture}[node distance = {20mm},thick, main/.style = {draw,circle}]

\node[main] (1) [label=below: $0$]{$W_2$};
\node[main] (2) [right of = 1][label=below: $0$]{$W_3$};
\node[main] (3) [right of= 2][label=below: $g$]{$M$};
\node (4) [left of = 1]{$\widehat{E}$};
\node (5) [above of = 2]{$\pi_{j}$};
\node (6) [right of =3]{$\pi_{i}$};
\draw (3) -- (6);
\draw (1) -- (2);
\draw (2) -- (3);
\draw (1) -- (4);
\draw (2) -- (5);

\end{tikzpicture}
 \end{center}
 
  \textit{Step 2:} By smoothing the node belonging  to the part corresponding to $M$ and writing $W^{g}_{4}=M+\mathbb{Z}j$ we reach the stratum associated to the graph: 
  \begin{center}
  \begin{tikzpicture}[node distance = {20mm},thick, main/.style = {draw,circle}]

\node[main] (1) [label=below: $0$]{$W_{2}$};
\node[main] (2) [right of= 1][label=below: $g$]{$W_4$};
\node (4) [left of = 1]{$\widehat{E}$};
\node (5) [above right of =2]{$\pi_{i}$};
\node (6) [below right of =2]{$\pi_{j}$};
\draw (2) -- (6);
\draw (1) -- (2);
\draw (1) -- (4);
\draw (2) -- (5);

\end{tikzpicture}
\end{center}

\textit{Step 3:} The restriction of $p$ to \(W^g_4\) has some period not contained in \(\mathbb{Q}p(\pi_i)+\mathbb{Q}p(\pi_ j)\) (otherwise, \(p\) would have image in \(\mathbb{Q}\otimes p(\Pi)\)). Apply Lemma \ref{l: connecting E-module I} to isoperiodically connect the part corresponding to \(W^g_4\) to a boundary stratum of type: 

\begin{center}
\begin{tikzpicture}[node distance = {20mm},thick, main/.style = {draw,circle}]

\node[main] (1) [label=below: $0$]{$W_2$};
\node[main] (2) [right of = 1][label=below: $0$]{$W_5$};
\node[main] (3) [right of= 2][label=below: $g$]{$M'$};
\node (4) [left of = 1]{$\widehat{E}$};
\node (5) [above of = 2]{$\pi_{j}$};
\node (6) [right of =3]{$\pi_{i}$};
\draw (3) -- (6);
\draw (1) -- (2);
\draw (2) -- (3);
\draw (1) -- (4);
\draw (2) -- (5);

\end{tikzpicture}
\end{center}
\textit{Step 4:} By smoothening the node between the two components of genus zero in the previous diagram, we obtain an isoperiodic connection with a form in the stratum 
\begin{center}
\begin{tikzpicture}[node distance = {20mm},thick, main/.style = {draw,circle}]
\node[main] (1) [label=below: $0$]{$W_1$};
\node[main] (2) [right of = 1][label=below: $g$]{$M'$};
\node (3) [above left of = 1]{$\pi_{j}$};
\node (4) [below left of =1]{$\widehat{E}$};
\node (5) [right of =2]{$\pi_{i}$};

\draw (2) -- (5);
\draw (1) -- (2);
\draw (1) -- (3);
\draw (1) -- (4);

\end{tikzpicture}
\end{center}
\noindent which corresponds to some point in \(\mathcal{B}_{M'}(p)\).
\end{proof}

\subsection{Proof of Theorem \ref{t:Rgn}}\label{ss: proof Rgn}
The theorem is true if \(g=0\) by Lemma \ref{l: connectedness spherical case}, if \(n=2\) by Theorem \ref{2poles}, or if \((g,n)= (1,3)\) by Lemma \ref{l: connecting E-module I}. Now assume that \(g\geq 2\) or \( n \geq 4\). Take a subset  \(E\subset \PP\) of the set of poles of cardinality \(|E|=n-1\). Lemma \ref{l: second degeneration} shows that any form in \( \Omega\mrs_{\Sigma,\PP} (p)\) can be isoperiodically degenerated to a \(E\)-boundary stratum of \(\Omega\mnc_{\Sigma,\PP}(p) \). Moreover, Lemmas \ref{l: connecting E-module I} and \ref{l: connecting E-module II} prove that all the forms in the \(E\)-boundary strata are isoperiodically connected in \(\Omega\mnc_{\Sigma,\PP}(p) \). This shows that \(\Omega\mnc_{\Sigma,\PP}(p) \) is connected. By Corollary \ref{c:closure},  \(\Omega\mrs_{\Sigma,\PP}(p) \) is also connected.

\section{The case of non-real period homomorphisms}

The goal of this section is to prove by induction that the following statement is true:

\vspace{0.3cm}

\noindent \((C_{g,n^*})\): Let $(\Sigma,\PP)$ be of type $(g,n)$. Then, for each \(p\in H^1(\Sigma\setminus\PP, \C)\) such that \( p(H_1(\Sigma\setminus\PP, \Z))\) is not contained in a real line, the set \(\Omega\mrs_{\Sigma,\PP}(p)\) is connected.

\vspace{0.3cm}

In this section we will apply the concept of $E$-modules associated to a subset $E\subset \PP$ of the poles--introduced in section \ref{ss: E-modules} -- for the particular case of $E=\PP$ and call them \(\PP\)-modules. Since the element $\pi_\PP$ is trivial,  a \(\PP\)-module is a sub-module \(V\) of \(H_1(\Sigma_g,\Z)\) that satisfies 
\begin{equation} \label{eq: decomposition}  H_1 (\Sigma\setminus\PP, \Z)= \Pi \oplus V,\end{equation} 
where \(\Pi=\Pi_\PP\) is the whole peripheral module. Such a module determines a spherical boundary stratum, \( \mathcal B _V\), corresponding to a
pointed Torelli class  of a simple closed curve \(c\subset \Sigma\setminus\PP \)  that separates the surface into two components: one of genus zero containing all the marked points, and the other of genus $g$ having its homology that surjects onto \(V\) via the inclusion map. We call such a boundary stratum a \(\PP\)-boundary stratum. 

The set of \(\PP\)-boundary strata, which is in bijection with the set of \(\PP\)-modules, acquires a structure of an affine space directed by \( \text{Hom} (H_1(\Sigma, \Z) , \Pi)\).

\subsection{Degeneration to \(\PP\)-boundary stratum}

\begin{lemma}\label{l:twinstodistinctpoles}
Let $\omega$ be a meromorphic differential with simple poles on a  smooth connected curve. Suppose $q_{1}, q_{2}$ are two distinct poles that have non-opposite peripheral period directions, i.e. $\frac{\int_{\pi_{q_{2}}}\omega}{\int_{\pi_{q_{1}}}\omega}\notin \mathbb{R}_{<0}$. Then up to a finite sequence of Schiffer variations,   $\omega$ admits a pair of  disjoint twins, each leading to one of the $q_{i}$'s.
\end{lemma}

\begin{proof}

By the residue theorem, $\omega$ has at least three poles. 
As presented in subsection \ref{ss: singular flat metric and directional foliations}, given a  directional foliation $\mathcal{G}_{\theta}$ of $\omega$ that is not co-linear to any peripheral period, the poles of $\omega$ are divided into two types: those for which the local real analytic foliation is a node and those for which it is a source. The singularity type at the pole depends on the angle between the peripheral period and $\theta$. The hypothesis on $q_{1}$ and $q_{2}$ allows to find a sector of directions for which both are attracting poles. Since the set of directions of saddle connections of $\omega$ is countable, for a dense set of directions $\theta$, the foliation $\mathcal{G}_{\theta}$ has no leaf which is a saddle connection. This means that every generalized saddle connection in the ribbon graph $\mathcal{R}(\mathcal{G}_{\theta})$ joins a zero with a pole. Moreover, since the form has more than two poles, we cannot have that all the leaves that accumulate on one pole have the other endpoint at a single pole (this happens only for $(\mathbb{C}^*,\frac{adz}{z})$). Hence, for each pole of $\omega$ there is a generalized saddle connection of $\mathcal{G}_{\theta}$ joining the pole to some zero.

\underline{Case 1: If $\omega$ has a single zero.} The constructed generalized separatrices for $q_1$ and $q_2$ approach the same zero of $\omega$ and they constitute a pair of twin semi-infinite geodesics leading to the chosen poles.

\underline{Case 2: If $g=0$ and $n\geq 3$.} Since the set of marked forms with given peripheral periods is connected (see Lemma \ref{l: connectedness spherical case}) , it suffices to find one example of forms with the given peripheral periods. For the differential on a sphere with three simple poles of residues \(\text{Res}_{q_{1}}(\omega),\text{Res}_{q_{2}}(\omega)\) and \(-\text{Res}_{q_{1}}(\omega)-\text{Res}_{q_{2}}(\omega) \), there is a unique zero, so the first step gives the existence of a pair of twins going to the poles of residues \(\text{Res}_{q_{1}}(\omega),\text{Res}_{q_{2}}(\omega)\). Consider a differential on the sphere with \(n-1\) simple poles, the first of residue \(\text{Res}_{q_{1}}(\omega)+\text{Res}_{q_{2}}(\omega) \) and the other ones of residues \( \text{Res} _{q_k} (\omega)\) for \(k=3, \ldots, n\). Cut the two spheres along a geodesic surrounding the poles of residue \(\pm (\text{Res}_{q_{1}}(\omega)+\text{Res}_{q_{2}}(\omega) )\) and glue them back to get the desired example.

\underline{Case 3: If $n=3$, and $g\geq 0$.} The proof goes by induction on the genus. In the case $g=0$, there is a single zero and the statement is true by Case 1. Suppose $g\geq 1$ and that we have proven the Lemma for any form with three poles up to genus $g-1$, and let us prove it for genus $g$. If $\omega$ has a single zero we are done by Case 1. Suppose $\omega$ has at least two zeros and let us prove that up to a Schiffer variation we can either diminish the number of zeros or find the desired pair of twins. Up to some small Schiffer variations we can suppose that the distances between distinct pairs of saddles are different and the integrals associated to the shortest saddle connection does not lie in the group of periods $\Lambda_\omega\subset\mathbb{C}$. Take one of the shortest saddle connection and consider its (one or two) adjacent twins. If one of them ends at a regular point, the Schiffer variation along that pair of twins leads to a form on a connected smooth curve with a zero less.  If one of them ends at the endpoint of the initial saddle connection, the union of the twins forms a closed path $\gamma$ satisfying $\int_{[\gamma]}\omega=0$. The Schiffer variation along this pair of twins converges to a boundary stable form $(C_0,Q_0,\omega_0)$ with no zero component, a node of zero residue  and three poles.  
Choose some lift $(C_0,Q_0,m_0,\omega_0)\in\Omega_0\overline{\mathcal{S}}_{\Sigma,\PP}$.
In the notations of Subsection \ref{ss:boundary strata of non-separating type} we can consider the stratum $\Omega \mathcal{B}_{[c]}$ containing $(C_0,Q_0,m_0,\omega_0)$ where $[c]$ is the pointed Torelli class of a simple closed curve $c\subset \Sigma\setminus\PP$ such that $m_0([c])=[\gamma]\in H_1(C_0\setminus Q_0)$ and $p=\percompact (C_0,Q_0,m_0,\omega_0)\in H^1(\Sigma\setminus \PP,\mathbb{C})$. 

Suppose first that $\gamma$ is non-separating. The map defined in Proposition \ref{p:connectedness of non compact boundary stratum} -- obtained by composing the normalization map on $\Omega\mathcal{B}_{[c]}(p)$ with the map that forgets the two points corresponding to the node -- \begin{equation}
f:\Omega\mathcal{B}_{[c]}(p)\rightarrow \Omega\mrs_{\overline{\Sigma},\overline{\PP}}(\overline{p})
\end{equation}
where $\overline{\Sigma}$ is the genus $g-1$ compact surface obtained from the geometric completion of $\Sigma\setminus c$ by collapsing each boundary component to a point and $\overline{\PP}\subset\overline{\Sigma}$ is the set of points that correspond to points in $\PP\subset\Sigma\setminus c$. By the inductive hypothesis, there is some point  \((\overline{C_1},\overline{Q_1},\overline{m_1},\overline{\omega}_1)\in\Omega\mrs_{\overline{\Sigma},\overline{\PP}}(\overline{p})\) that has two twin paths from a zero to the poles $q_1$ and $q_2$. By Proposition \ref{p:connectedness of non compact boundary stratum} there is a point $(C_1,Q_1,m_1,\omega_1)$ in the connected component of $\Omega \mathcal{B}_c(p)$ containing $(C_0,Q_0,m_0,\omega_0)$ such that $f(C_1,Q_1,m_1,\omega_1)=(\overline{C_1},\overline{Q_1},\overline{m_1},\overline{\omega}_1)$. By construction, there are two twins for the stable form with one node $\omega_1$ joining a zero of $\omega_1$ with the two poles corresonding to $q_1$ and $q_2$. We claim that we can suppose that they do not pass through the node. Indeed, if they do,  we can deform one of the twins locally around the node each time it passes through the node, and take its twin, to avoid any of them passing through the node. Smoothing the node we obtain the desired pair of twins on a smooth curve. We have shown that the original form and this latter form lie in the same connected component of $\Omega_0\mnc_{\Sigma,R}(p)$. By Corollary \ref{c:closure}, they also lie in the same connected component of the manifold $\Omega\mrs_{\Sigma,\PP}(p)$. By the discussion of subsection \ref{ss:local description} they can also be joined by a finite sequence of Schiffer variations.

If $\gamma$ is separating, then as there are only three poles, it leaves all poles on one side thanks to the residue theorem applied to each part. On the other side the form is holomorphic and non-zero, so the genus cannot be zero. This means that the component that has the three poles has genus at most $g-1$ and the pair of poles with the non-opposite peripheral period direction condition there. Apply the inductive hypothesis to that component to find a point in its isoperiodic component that has the desired pair of twins to the poles $q_1$ and $q_2$. By using attaching maps as in  the proof of Lemma \ref{p:connected isoperiodic strata}, with the other factor fixed, we get an isoperiodic path in the boundary  associated to the nodal curve. The endpoint of this path has a separating node, and two twins to the poles $q_1$ and $q_2$ as desired. If the twins pass through the node, we can make a small local deformation of one of the twins around the node, and follow it with the other twin, and suppose that they do not pass through the node. Smoothing the node, we obtain the desired pair of twins on a smooth curve.

\underline{Case 4: $g\geq 0$ and $n\geq 3$:}
By induction on the genus. For genus $g=0$ apply Case 2. Take $g\geq 1$ and suppose Lemma is true for all genera up to $g-1$ and any number of poles. Again we proceed as in Case 3. Everything works word by word up to the case where $\gamma$ is non-separating where we have to substitute the the expression "three poles" with the number of poles of the form $\omega_0$. When $\gamma$ is a separating curve  that separates the set of poles in two non-empty parts, we proceed in a similar manner as in Case 3: apply the inductive hypothesis if the poles that we want to join by twins lie on the same side, or find a pair of twins starting at the node, if they lie on distinct sides. The only difference with Case 3 is that $\gamma$ can leave all poles on one part. Then, the other part has necessarily positive genus, since it has holomorphic  form with isolated zeros. This means that the side that contains the poles has at most genus $g-1$ and we are done by the inductive hypothesis and the smoothing procedure. 




\end{proof}

\begin{proposition}
Let $\omega$ be a  meromorphic differential on a  smooth connected curve $C$ of genus $g\geq 1$. Suppose $(\omega)_{\infty}=q_{1}+\ldots+q_{n}$ with $q_i\neq q_j$ if $i\neq j$ and  \[\Lambda_{\omega}=\left\{\int_{a}\omega:a \in H_{1}(C\setminus\{q_{1},\ldots,q_{n}\},\mathbb{Z})\right\}\text{ is not contained in a real line in }\mathbb{C}.\] Then $\omega$ is isoperiodically connected to a stable form  on a spherical boundary component where the spherical part contains all $n$ poles  (and the other genus $g$ part is holomorphic). 
\end{proposition}

\begin{proof}
The statement is true for $n=2$ and $g\geq 1$ by \cite{CD}. 

Fix some $n\geq 3$ and suppose we have proved the statement for all meromorphic forms on smooth curves with at worst $n-1$ simple poles. 

Take a form $\omega$ of genus $g$ with $n$ simple poles.  By the residue theorem there exists at least a pair of  distinct poles, say $q_{1},q_{2}$, that have non-opposite peripheral period direction.  By Lemma \ref{l:twinstodistinctpoles} we can suppose that $\omega$ has a pair of twins  leading to the pair $q_{1},q_{2}$.  
The Schiffer variation along that pair of twins produces a stable form with one separating node that has a part $\omega_{1}$ of genus $g$ with $n-1$ poles and a part $\omega_{0}$ of genus zero with $3$ poles.

We claim that, up to changing the choice of twins to distinct poles, we can guarantee that $\Lambda_{\omega_{1}}$ is not contained in a real line in $\mathbb{C}$. If this is the case, we can apply the inductive hypothesis  and find an isoperiodic deformation of  $\omega_{1}$ towards a stable form with a separating node having a genus $g$ holomorphic component a genus zero meromorphic form with $n-1$  poles.  By using Proposition \ref{p:connected isoperiodic strata} we can take the isoperiodic deformation to the boundary stratum of the initial stable form. In the limit of that isoperiodic deformation we find a stable curve with two nodes. A further Schiffer variation (smoothing the initial node) connects the latter to a form as in the statement of the Lemma, having the same holomorphic part of genus $g$. 

Let us prove that,  if the first pair of twins leads to a form $\omega_{1}$ satisfying that $\Lambda_{\omega_{1}}$ is contained in a real line then we can find another pair of twins to distinct poles that does not satisfy this condition. 
Without loss of generality, we can suppose $\Lambda_{\omega_1}\subset\mathbb{R}$. Since $\Lambda_\omega$ is not contained in a real line, we have that $\Lambda_{\omega_0}$ is not contained in $\mathbb{R}$. Since the underlying curve has genus zero , $\Lambda_{\omega_0}$  is generated by the peripheral periods, and there is at least one pole, with non-real peripheral period. But since there are only three poles for $\omega_0$, the sum of their peripheral periods is zero, and the nodal peripheral has real period , we have that both poles $q_1,q_2$ have non-real peripheral period.
The argument can be restarted taking a pole $\tilde{q_1}$ of $\omega_0$ and a pole $\tilde{q}_2$ of $\omega_1$ and consider them as the poles of $\omega$.  We obtain forms  $\tilde{\omega}_0$ and  $\tilde{\omega}_1$ by the same procedure. If we find two periods of $\tilde{\omega}_1$ that are not $\mathbb{R}$-collinear we are done. If  $n\geq 4$, the form $\tilde{\omega}_1$ has a pole with real peripheral period  (corresponding to some pole of $\omega_1$) and a pole with non-real peripheral period (corresponding to one of the poles of  $\omega_0$). When $n=3$ the part $\tilde{\omega}_1$ has two poles. In this case it is enough to find a non-peripheral curve  with non-zero period for $\tilde{\omega}_1$ to conclude. By Lemma \ref{l: Haupt}, applied to the case of two poles, the periods of $\tilde{\omega}_1$ cannot lie in the cyclic module generated by the (non-real) peripheral period and we are done. 
\end{proof}

\subsection{Slices of \(\PP\)-boundary strata}
\label{s:slices}
Observe that the map \(H_1(\Sigma\setminus\PP , \Z) \rightarrow H_1(\Sigma, \Z)\) induced by inclusion preserves the intersection form on both parts, and restricts to any \(\PP\)-module as an isomorphism. In particular the restriction of the intersection form to any \(\PP\)-module is symplectic and unimodular. 

We recall that a period \( p\in H^1 (\Sigma, \C) \) on the closed compact surface $\Sigma$ of genus $g$ satisfies Haupt conditions if its volume \(\text{vol} (p)= \Re(p) \cdot \Im (p) \) is positive, and if it is different from the volume of the quotient \( \C / p (H_1(\Sigma, \Z))\). In the case the image of \(p\) is non discrete, we use the convention that the volume of \( \C / p (H_1(\Sigma, \Z))\) is infinite. Otherwise the volume of \(p\) is an integer multiple of the volume of \(\C / p (H_1(\Sigma, \Z))\), the integer is called the degree and is denoted by \(\text{deg} (p)\). So Haupt's conditions can be rephrased as follows \begin{equation}\label{eq:Haupt conditions} \text{vol} (p) >0\text{ and } \text{deg} (p) \geq 2.\end{equation} 
Those conditions characterize period homomorphisms of non zero abelian differentials on smooth homologically marked curves. See \cite[\S 6]{CDF2} for more on these concepts.

Given a period \(p\in H^1(\Sigma\setminus\PP , \C)\), we will say that a \(\PP\)-module \( V\) is \(p\)-admissible if the restriction \(p_{|V}\) satisfies Haupt's conditions, namely 
\[ \text{vol} (p_{|V}) >0 \text{ and } \text{deg} (p_{|V}) \geq 2 . \]

\begin{lemma} The intersection $\Omega\mathcal{B}_V(p)$ of a \(\PP\)-boundary component \(\mathcal \Omega B_V\) with  \(\Omega\mnc_{\Sigma,\PP}(p)\) is non empty if and only if \(V\) is \(p\)-admissible, and it is connected if \(\text{deg} (p_{|V}) >2\).  \end{lemma}

\begin{proof}
By equation \eqref{eq: decomposition} and Lemma \ref{p:connected isoperiodic strata}, the non-emptyness (resp. connectedness)  of $\mathcal{B}_V(p)$ is equivalent to that of $\Omega\mrs_{\Sigma,\emptyset}(p_{|V})$ and $\Omega\mrs_{\Sigma_0,\PP_0}(p_{|\Pi})$ where $\Sigma_0$ has genus zero and $\PP_0\subset\Sigma_0$ has $n$-points. Lemma \ref{l: connectedness spherical case} treats the case of genus zero and Theorem \ref{t:isoperiodic sets of holo diff} the holomorphic case (without poles) in genus $g\geq 1$. 
\end{proof}


\begin{lemma}\label{l: connectedness slices} Let \(g, n\) be integers with \(n\geq 3\). Assume \(C_{h,m}\) is true for any genus \(h\) and any number of points \(m\leq n-1\). Let \(V,V'\) be two \(\PP\)-modules, and  \(E\subset \PP\) be a subset of cardinality \(|E|\leq n-2\). Assume that 
\[V+ \Pi_E = V' + \Pi_E .\]
Then, any connected component of \(\Omega\mathcal B_V (p)\) is connected in \(\Omega\mnc_{\Sigma,\PP} (p)\) to any connected component of \(\Omega\mathcal B_{V'} (p)\). \end{lemma} 
\begin{proof}
The boundary stratum \(\mathcal B_V\) containing a form of period $p$ is characterized by the graph:
\begin{center}
\begin{tikzpicture}[node distance = {20mm},thick, main/.style = {draw,circle}]
\node[main] (1) [label=below: $0$]{$\Pi$};
\node[main] (2) [right of = 1][label=below: $g$]{$V$};
\node (3) [left of = 1]{$\PP$};
\draw (1) -- (2);
\draw (1) -- (3);


\end{tikzpicture}
\end{center}
\noindent where the edge between $\PP$ and the vertex represents an edge for each element of the set. We have $\text{vol}(p_{|V})>0$ and $\Pi=\Pi_E+\Pi_{\PP\setminus E}$. 

By using Lemma \ref{l: connectedness spherical case} and Proposition \ref{p:connected isoperiodic strata} we can realize the period $p$ by a form in the stratum whose graph is: 
\begin{center}
\begin{tikzpicture}[node distance = {20mm},thick, main/.style = {draw,circle}]
\node[main] (1) [label=below: $0$]{$\Pi_E$};
\node[main] (2) [right of = 1][label=below: $g$]{$V$};
\node[main] (3) [left of = 1][label=below: $0$]{$\Pi_{\PP\setminus E}$};
\node (4) [left of =3]{$\PP\setminus E$};
\node (5) [above of =1 ]{$E$};
\draw (1) -- (2);
\draw (1) -- (3);
\draw (1) -- (5);
\draw (3) -- (4);


\end{tikzpicture}
\end{center}
Moreover, since the fibers or the period map over periods of genus zero are connected, the two previous forms are isoperiodically connected. 
Smoothing the node belonging to the genus $g$ component connect the previous forms isoperiodically to a form with the following graph: 
\begin{center}
\begin{tikzpicture}[node distance = {20mm},thick, main/.style = {draw,circle}]
\node[main] (1) [label=below: $0$]{$\Pi_{\PP\setminus E}$};
\node[main] (2) [right of = 1][label=below: $g$]{$\Pi_E+V$};
\node (3) [left of = 1]{$\PP\setminus E$};
\node (4) [right of =2]{$E$};
\draw (1) -- (2);
\draw (1) -- (3);
\draw (2) -- (4);


\end{tikzpicture}
\end{center}
The same argument can be used to isoperiodically connect any form in $\Omega\mathcal{B}_{V'}(p)$ to a stratum whose corresponding graph is: 
\begin{center}
\begin{tikzpicture}[node distance = {20mm},thick, main/.style = {draw,circle}]
\node[main] (1) [label=below: $0$]{$\Pi_{\PP\setminus E}$};
\node[main] (2) [right of = 1][label=below: $g$]{$\Pi_E+V'$};
\node (3) [left of = 1]{$\PP\setminus E$};
\node (4) [right of =2]{$E$};
\draw (1) -- (2);
\draw (1) -- (3);
\draw (2) -- (4);


\end{tikzpicture}
\end{center}
The hypothesis on $V,V'$ implies that the two forms lie in the same boundary stratum of the ambient space. Moreover, the inductive hypothesis applied to $p_{|\Pi_E+V}$ (which does not have image in a real line in $\mathbb{C}$ since $p_{|V}$ is the period of a holomorphic one form) and Proposition \ref{p:connected isoperiodic strata} shows that they lie in the same isoperiodic component, thus finishing the proof.
\end{proof}

\subsection{Pairs of periods satisfying Haupt's conditions: a lemma}

\begin{lemma}\label{l: pair of Haupt conditions} Let $\Sigma$ be a surface of genus $g\geq 1$ , \( p , p '\in H^1 (\Sigma, \C) \) and \( c,c '\in \C\). Assume that \( \Im ( c \overline{p})\) and \( \Im (c' \overline{p'}) \) are not zero and that they do not differ by multiplication by a negative real number. Then there exists \(\varphi \in H^1 (\Sigma, \Z) \) such that both \( p+ c \varphi\) and \( p' + c' \varphi \) satisfy Haupt conditions.  
\end{lemma}

\begin{proof}
By assumption, there exists \(\varphi\in H^1(\Sigma, \Z) \) such that 
\begin{equation}\label{eq: positivity} \Im (c\overline{p} ) \cdot \varphi >0 \text{ and } \Im (c'\overline{p'} ) \cdot \varphi >0 .\end{equation}
From the formula \(\text{vol} (q) = \frac{i}{2} q \cdot \overline{q} \), we have the expressions 
\[ \text{vol} (p+ c \varphi ) = \text{vol} (p) + \Im (c\overline{p} ) \cdot \varphi \text{ and } \text{vol} (p'+c' \varphi) = \text{vol} (p') + \Im (c'\overline{p'} ) \cdot \varphi, \]
so letting \( \varphi_k= k\varphi\) and \(p_k=p+c\varphi_k\), \(p_k'=p'+c'\varphi_k\), for a sufficiently large integer \(k\), we get that both volumes \(\text{vol} (p_k)\) and \( \text{vol} (p'_k) \) tend linearly to \(+\infty\) when \(k\) tends to \(+\infty\). Notice in particular that we can assume that \(p\) and \(p'\) (up to replacing \(p\) and \(p'\) by \(p_k\) and \(p_k'\) for large \(k\) if necessary) have positive volume.  

We can also assume that the form \( \varphi\) satisfying \eqref{eq: positivity} does not vanish identically on \(\text{ker} (p)\)  unless \(p\) is injective. Indeed, assuming that \(\varphi\) is identically vanishing on \(\text{ker} (p)\neq \{0\}\), take \(\psi\) that is not vanishing there and replace \(\varphi\) by \( l\varphi+\psi\) for a large integer \(l\).  The same argument can be made to ensure that \( \varphi\) does not vanish identically on \(\text{ker} (p')\) as well unless \(p'\) is injective.

Notice that \(p\) and \(p'\) have positive volume, hence their kernel have corank at least two. Because \(\varphi\) does not vanish identically on the submodules \(\text{ker} (p)\) and \(\text{ker} (p')\), appart in cases they are reduced to \(\{0\}\),  their intersections with \(\text{ker} (\varphi)\) are submodules of \(\text{ker} (\varphi)\) of corank at least two as well. Hence there exists \(M\subset \text{ker} (\varphi)\) a rank two submodule such that \(M\cap \text{ker} (p) = \{0\}  \) and \(M\cap \text{ker}(p')=\{0\}\).  

We claim that the degree of \(p_k\) tends to infinity when \(k\) tends to \(+\infty\). Indeed, notice that in restriction to \(M\), the period \(p_k\) equals \(p\), which is injective on \(M\) by construction, so \(p_k(M)\) is a fixed rank two submodule \(\Lambda\subset \C\). If for some \(k\), \(p_k\) has finite degree, then this module needs to be a lattice, and we get 
\[ \text{deg}(p_k) = \frac{\text{vol}(p_k)}{\text{vol} (\C/p_k(H_1(\Sigma_g,\Z)))}\geq \frac{\text{vol}(p_k)}{\text{vol} (\C/\Lambda)} \rightarrow _{k\rightarrow +\infty} +\infty ,\]
hence our claim follows.
A similar argument shows that \(\text{deg} (p_k')\) tends to \(+\infty\) when \(k\) tends to \(+\infty\), finishing the proof of the lemma. \end{proof}

\subsection{Connecting \(\PP\)-boundary strata: proof of \(C_{g,n}\) by induction} \label{ss: connecting R-bnd strata}

In \cite{CD} it is proved that \(C_{g,2}\) is true for any genus \(g\). Let $(\Sigma,\PP)$ be of type \((g,n)\) with \(n\geq 3\), and assume that \(C_{g,m}\) is true for any \(m<n\). Take $p\in H^1(\Sigma\setminus\PP,\mathbb{C})$ whose image is not contained in a real line in $\mathbb{C}$. 

Let \(V,V'\subset H_1 (\Sigma\setminus\PP, \Z)\) be two admissible \(\PP\)-modules. Write 
\[ V' = V + \Phi \text{ with } \Phi = \sum_{i=1,\ldots, n} \varphi_i \pi_i \]
where \(\varphi_i \in H^1(\Sigma, \Z)\) and each $\pi_i$ is a peripheral class in $H_1(\Sigma\setminus\PP)$ around a point of $\PP$. The \(n\)-tuple \((\varphi_1, \ldots,\varphi_n) \) is well-defined up to adding to each term a linear form \(\varphi\). 

Notice that if there exists two distinct indices \(i\neq j\) so that \( \varphi_i = \varphi_j \) then any component of \( \mathcal B_{V}(p)\) is connected in \(\Omega\mnc_{\Sigma,\PP}(p) \) to any connected component of \( \mathcal B_{V'} (p)\). Indeed, in that case substracting \(\varphi= \varphi_i= \varphi_j\) to all the \(\varphi_k\)'s -- which does not change \(\Phi\) -- we see that \(\Phi\) takes values in \( \Pi_{\PP\setminus \{i,j\}}\), so letting \(E:=\PP\setminus \{i,j\}\) we get \(V+\Pi_E= V'+\Pi_E\), and the conclusion follows from Lemma \ref{l: connectedness slices} and the fact that \( C_{g,m}\) is true for every \(m<n\). 

We claim that there exist two distinct indices \(i\neq j\) such that \( \Im (p (\pi_i) \overline{p_{|V}}) \) and \( \Im (p(\pi_j) \overline{p_{|V'}}) \) -- which are not vanishing since \(V\) and \(V'\) are \(p\)-admissible -- are not \(\R ^{\geq 0}\)-co-linear. Indeed, if this were not the case, since the number of indices is at least three, that would mean that all the periods \(\Im (p (\pi_i)  \overline{p_{|V}}) \) for \(i=1,\ldots, n\) would be \(\R^{\geq 0}\)-co-linear, and this is contradictory with the fact that \(\sum_i \pi_i = 0\). 

Introduce the modules 
\[ V'' := V+ (\varphi_i - \varphi ) \pi_i \text{ and } V''':= V' + (\varphi - \varphi_j)\pi_j ,\]
where \( \varphi \in H^1 (\Sigma, \Z)\) will be determined later. Define \(c:=- p (\pi_i) \), \( c':= p (\pi_j)\), and 
\[ q: = p_{|V} +\varphi_i p(\pi_i) \text{ and } q':= p_{V'}- \varphi _j p(\pi_j) \] that we see as elements of \( H_ 1(\Sigma, \Z) \) under the identifications \( V\simeq H_ 1 (\Sigma, \Z) \simeq V'\). We have 
\[ \Im ( c \overline{q}) = -\Im (p (\pi_i) \overline{p_{|V}} ) \text{ and } \Im (c' \overline{q'}) = \Im (p (\pi_j) \overline{p_{|V'}}) , \]
so \(\Im ( c \overline{q})\) and \(\Im ( c' \overline{q'})\) are not \( \R^{\leq 0}\)-co-linear.  Lemma \ref{l: pair of Haupt conditions} says that one can find \(\varphi \in H^1 (\Sigma_g, \Z)\) so that \( q + c\varphi \) and \( q' + c' \varphi\) both satisfy Haupt's conditions. That means that \(V''\) and \(V'''\) are both \(p\)-admissible. 

Since \( V''\) differs from \(V\) by a form with values in \( \Z \pi _i\), every connected component of \(\mathcal B_V (p)\) is connected in \(\Omega\mnc_{\Sigma,\PP}(p)\) to any connected componet of \( V''\). Similarly, every connected component of \( \mathcal B_{V'''} (p)\) is connected in \(\Omega\mnc_{\Sigma,\PP}(p)\) to any connected component of \( \mathcal B_{V''} (p)\). To conclude, notice that 
\[ V''' - V'' =  \sum _{k\neq i,j} (\varphi_k - \varphi) \pi_k .  \]
Hence, every connected component of \( \mathcal B_{V''} (p)\) is connected in \(\Omega\mnc_{\Sigma,\PP}(p)\) to any connected component of \( \mathcal B_{V'''} (p)\). This shows that \(C_{g,n}\) is true.

\section{Dynamics of modular group} 

Let \(n\geq 3\) and \(p_1, \ldots, p_n\in \mathbb C ^*\). In this section we classify the orbit closures of the action of the modular group \(\text{Mod} _{\Sigma,\PP} \) on \( H^1 _{p_1,\ldots, p_n} (\Sigma\setminus\PP , \mathbb C) \), the subgroup of \( H^1(\Sigma\setminus\PP, \mathbb C) \)  formed by periods that satisfy \( p(\pi_i) =p_i\) for the peripheral classes $\pi_i$ \(i=1,\ldots, n\) around each of the $n$ points of $\PP$.  

\begin{definition} A \textit{discrete factor} of the period \( p \in H^1(\Sigma\setminus\PP , \mathbb C) \) is a quotient \( \mathbb C \rightarrow \mathbb C / K\),
where \(K\) is a closed subgroup of \( \mathbb C\), which is such that the composition \(p_K\) of \( p\) with the quotient map \( \mathbb C \rightarrow \mathbb C / K\) has a discrete image. Two periods \(p,p'\) having the same discrete factor \(\mathbb C\rightarrow \mathbb C/ K\) have the same image in this latter if the images of \(p_K\) and \(p_K'\) are the same.\end{definition} 

Suppose that, up to post-composition of \(p\) with a \(\mathbb R\)-linear invertible map of \(\mathbb C\), we have \( p (\Pi) =\mathbb Z\) and \(\overline{ p \left(H_1(\Sigma\setminus\PP,\mathbb Z)\right)}= \mathbb R +i\mathbb Z\). In this case, we define \( V(p) \in \mathbb R/ \mathbb Z\) by the following: choose a section \( \sigma : H_1 (\Sigma, \mathbb Z ) \rightarrow H_1(\Sigma\setminus\PP , \mathbb Z) \) and define 
\begin{equation} \label{eq: vol mod Z} V(p ) := \text{Vol} (p\circ \sigma) \text{ mod } \mathbb Z ,\end{equation} 
where \( \text{Vol} (q) = \Re (q) \cdot \Im (q) \) as in the beginning of Section \ref{s:slices}. It corresponds to the volume modulo \(\mathbb Z\) of the restriction of \(p\) to a \(\PP\)-module. Notice that \( V(p)\) does not depend on the chosen section \(\sigma\), for if \(\sigma'\) is another choice, the difference \( \sigma' -  \sigma\) takes values in \(\Pi\), so \( q = p\circ \sigma' - p \circ \sigma  \in H^1 (\Sigma , \mathbb Z) \), and in particular: 
\[ \text{Vol} (p\circ \sigma' ) = \text{Vol}( p\circ \sigma + q ) = \text{Vol} (p\circ \sigma) + q \cdot \Im (p\circ \sigma) = \text{Vol} (p\circ \sigma)\text{ mod } \mathbb Z. \]
Notice also that under the same assumptions on \(p\), \( V(p) \) is invariant by the action of the modular group, namely 
\[ V( p \circ M) = V(p) \text{ for every } M\in \text{Mod} _{\Sigma,\PP} .\]
The main result of this section is the following:

\begin{proposition} \label{p: orbit closures}
The closure of the \(\text{Mod} _{\Sigma,\PP} \)-orbit of a period \(p\in H^1 _{p_1,\ldots, p_n} (\Sigma\setminus\PP , \mathbb C)\) consists of the periods \(p'\) having the following properties 
\begin{itemize} 
\item 
\(p' \left( H_1(\Sigma\setminus\PP , \mathbb Z) \right) \subset \overline{p \left(H_1(\Sigma\setminus\PP,\mathbb Z )\right)} .\)
\item Each discrete factor of \(p\) is a discrete factor of \(p'\). Moreover, \(p\) and \(p'\) have the same image in this discrete factor.
\item if up to post-composing \(p\) with a \(\mathbb R\)-linear invertible map of \(\mathbb C\), we have \( p (\Pi) =\mathbb Z\) and \( \overline{ p \left(H_1(\Sigma\setminus\PP,\mathbb Z)\right)}= \mathbb R +i\mathbb Z \), then \( V (p') = V (p)\). 

\end{itemize} 
\end{proposition}

Before proceeding to the proof of this result, we begin by a useful reduction. We denote by \( \Lambda _\Pi \) the closure in \(\mathbb C\) of the module generated by the \(p_i\)'s. Notice that the reduction of a period of \( H^1 _{p_1,\ldots, p_n} (\Sigma\setminus\PP , \mathbb C) \) modulo \( \Lambda_\Pi\) vanishes on each peripheral cycle \(\pi_i\), so it induces a period of the closed surface \(\Sigma\); this gives a map 
\begin{equation}\label{eq: map in cohomology} H^1 _{p_1,\ldots, p_n} (\Sigma\setminus\PP , \mathbb C) \rightarrow H^1  (\Sigma , \mathbb C/ \Lambda_\Pi )\end{equation}  

\begin{lemma}\label{l: reduction}
A closed \(\text{Mod} (\Sigma,\PP ) \)-invariant subset of \(H^1 _{p_1,\ldots, p_n} (\Sigma\setminus\PP , \mathbb C) \) is the pre-image under the map \eqref{eq: map in cohomology} of a closed \( \text{Mod} (\Sigma) \)-invariant subset of \(H^1  (\Sigma , \mathbb C/ \Lambda_\Pi ) \). 
\end{lemma}

 \begin{proof}
The image of the modular group in \( \text{Aut} (H_1(\Sigma\setminus\PP , \mathbb Z) )\)
is the group \(\Gamma\) acting by the identity on the peripheral subgroup \(\Pi \), and on the quotient \( H_1(\Sigma\setminus\PP , \mathbb Z) /\Pi \simeq H_1 (\Sigma, \mathbb Z) \) by preserving the intersection form. Let \(U\) be the subgroup of \(\Gamma\)
defined as the kernel of the map \(\Gamma \rightarrow \text{Sp} (H_1 (\Sigma , \mathbb Z) )\). Given an element \(u \in U\), the map \( u-I\) takes values in the peripheral module \( \Pi\), and vanishes on this latter. Every such map belongs to \(U\), so 
\[ U \simeq H^1 (\Sigma, \Pi ) . \] 

The action of an element \(u\in U\) on \( H^1 _ {p_1,\ldots, p_n} (\Sigma\setminus\PP , \mathbb C) \) is by translation by the period of \( H^1 _ {0,\ldots, 0} (\Sigma\setminus\PP , \oplus_i \mathbb Z p_i) \) which is image of \(u\) by the composition of morphisms
\[ U \rightarrow H^1 (\Sigma, \Pi) \rightarrow H^1 _{0,\ldots, 0}(\Sigma\setminus\PP , \Pi ) \rightarrow H^1 _{0,\ldots, 0}(\Sigma\setminus\PP , \sum _i \mathbb Z p_i  ) , \]
where the second is induced by inclusion \( \Sigma\setminus\PP \rightarrow \Sigma\), and the third is induced by the map  \( \Pi \rightarrow \sum_i \mathbb Z p_i \) given by \(\pi_i \mapsto p_i\).
This composition of morphisms is surjective, so that the group acting  on \( H^1 _ {p_1,\ldots, p_n} (\Sigma\setminus\PP , \mathbb C) \) made of translations by  \( H^1 _ {0,\ldots, 0} (\Sigma\setminus\PP , \sum_i \mathbb Z p_i) \) is contained in the action of the modular group.
This group being dense in \( H_{0,\ldots, 0} ^1 (\Sigma\setminus\PP , \Lambda_\Pi)\), we infer that any closed \(\text{Mod} (\Sigma\setminus\PP ) \)-invariant subset of \(H^1 _{p_1,\ldots, p_n} (\Sigma\setminus\PP , \mathbb C)\) is invariant by translations by the closed subgroup \( H^1 _{0,\ldots, 0} (\Sigma\setminus\PP , \Lambda_\Pi)\). The result then follows from  the exact sequence 
\[ 0\rightarrow H^1_{0,\ldots, 0}  (\Sigma\setminus\PP, \Lambda_\Pi) \rightarrow H^1_{p_1,\ldots, p_n}  (\Sigma\setminus\PP, \mathbb C) \rightarrow H^1   (\Sigma \setminus\PP , \mathbb C / \Lambda_\Pi ) \rightarrow 0 ,  \]
the last arrow is precisely \eqref{eq: map in cohomology}.
\end{proof}

\begin{remark}\label{r: duality} We notice that the action of \( \text{Mod} (\Sigma)\) on \( H^1 (\Sigma, \mathbb C / \Lambda_\Pi)\) transits via the natural action of the unimodular symplectic group \(\text{Aut} (H_1 (\Sigma, \mathbb Z))\) on \( H^1 (\Sigma, \mathbb C / \Lambda_\Pi)\). This latter is isomorphic to \( \text{Aut} (H^1 (\Sigma, \mathbb Z))\) by duality. We will use this in the sequel. \end{remark}

\begin{proof} [Proof of Proposition \ref{p: orbit closures}]
We study several cases: 

\vspace{0.2cm}

\textit{Case 1: \( \Lambda _\Pi \) is infinite cyclic.} Notice that the third item of the proposition only occurs in that case. The closed invariant subsets of \( H^1 (\Sigma, \mathbb C / \Lambda_\Pi) \) have been classified in \cite{CD}, and as a consequence Proposition \ref{p: orbit closures} holds in that case. When $n=2$ we automatically fall in this case. 

\vspace{0.2cm}

\textit{Case 2: \(\Lambda_{\Pi} \) is a lattice.} Up to real linear automorphism of \(\mathbb C\), we have \( \Lambda _\Pi = \mathbb Z + i \mathbb Z\). Lemma \ref{l: reduction} and remark \ref{r: duality} reduce the question to the understanding of orbit closures of the action of \(\text{Aut} (H^1 (\Sigma ,\mathbb Z))  \) on \( H^1 (\Sigma, \mathbb C / \mathbb Z + i \mathbb Z)\).  An analog of Lemma \ref{l: reduction} shows that orbit closures of \( \text{Aut}( H^1 (\Sigma, \mathbb Z) ) \) on \( H^1 (\Sigma, \mathbb C /\mathbb Z + i \mathbb Z)\) are in bijective correspondence with orbit closures of the action of the group \(G^\mathbb Z =\text{Aut} (H^1 (\Sigma, \mathbb Z)) \ltimes H^1 (\Sigma, \mathbb Z+ i \mathbb Z) \) on \( H^1 (\Sigma, \mathbb C)\) (here the twisted product is defined by the linear representation of \(\text{Aut} (H^1 (\Sigma, \mathbb Z))\) on \( H^1 (\Sigma, \mathbb C)\) given by the natural representation on real and imaginary part of the cohomology, and the action of \( H^1 (\Sigma, \mathbb Z+i\mathbb Z) \) on \( H^1 (\Sigma, \mathbb C)\) is merely given by translations). We leave to the reader to verify that this corresponds to Lemma \ref{l: reduction} applied to the case \(n\geq3\) with $\Lambda_{\Pi}\subset\mathbb{C}$ a lattice.

As in \cite{CDF2} and \cite{CD}, we will make use of Ratner's theory for the classification of orbit closures of \( G^\mathbb Z\) on \( H^1 (\Sigma, \mathbb C)\), as in \cite{CD}, from which we borrow notations. So we let \( G = \text{Aut} (H^1 (\Sigma, \mathbb R))  \ltimes H^1 (\Sigma, \mathbb C) \), where as before the twisted prodcut is defined by the linear representation of \( \text{Aut} (H^1 (\Sigma , \mathbb R) )\) on \( H^1 (\Sigma, \mathbb C)\) given by the natural representation on real and imaginary part, and observe that \( G^\mathbb Z \) embeds naturally as a lattice in \( G\). Given a period \( P \in H^1 (\Sigma, \mathbb C)\), its stabilizer in \(G\) is the group 
\begin{equation}\label{eq: stabilizer}  S_P = \{ (L, t) \in G \ |\ t= P - LP \}. \end{equation}
Since it is generated by unipotents, Ratner's theorem shows that there exists a Lie subgroup \(G_P \subset G \) such that 
\begin{itemize}
    \item \(S_P\subset G_P\) 
    \item \(G_P^\mathbb Z := G_P \cap G^\mathbb Z \) is a lattice in \(G_P\)
    \item \(\overline{G^\mathbb Z P} = G^\mathbb Z G_P P \).
\end{itemize}

We denote \(B_P \subset G_P  \) the intersection of \(G_P\) and of the kernel of the natural projection \( G \rightarrow \text{Aut} (H^1 (\Sigma, \mathbb R))\), and by \(A_P\) the quotient of \( G \) by \( B_P\). As in \cite[Lemma 11.2]{CD}, one proves that \( B_P^\mathbb Z:= B_P \cap G^\mathbb Z\) is a lattice in \( B_P\), and that the image of \( G_P^\mathbb Z\) in \(A_P\) is a lattice denoted \( A_P^\mathbb Z\). We denote by \(i: A_P \rightarrow \text{Aut} (H^1 (\Sigma, \mathbb R))\) the restriction of the projection \( G \rightarrow \text{Aut} (H^1 (\Sigma, \mathbb R) ) \); this is an injective immersion. Since \( S_P\) is a section of the projection \( G \rightarrow \text{Aut} (H^1 (\Sigma, \mathbb R))\), one finds that \(i\) is surjective (even in restriction to the connected component of \(A_P\) containing the neutral element). We conclude that 
\[ G_P = S_P \ltimes B_P\]
and that \( B_P\) is invariant by \( \text{Aut} (H^1 (\Sigma, \mathbb R))\).

\begin{lemma}
    The connected component \( B_P^0\) of \( B_P\) containing the neutral element \(0\) is either 
    \begin{enumerate}
        \item \(\{0\}\),
        \item \(H^1 (\Sigma, \mathbb C)\),
        \item or a subgroup of the form \(\{\lambda u + \mu u i\ |\ u\in H^1 (\Sigma, \mathbb R)\}\) for some real numbers \( \lambda, \mu \in \mathbb R\).
    \end{enumerate}  
\end{lemma}

\begin{proof}
    This is because \(B_P^0\) is a \(\mathbb R\)-linear vector subspace of \( H^1 (\Sigma, \mathbb C)\) which is invariant under \( \text{Aut} (H^1 (\Sigma, \mathbb R))\), and that (1), (2) and (3) are the only such subspaces. 

    Indeed, the \(\mathbb R\)-linear subspace \( B_P^0\cap H^1 (\Sigma, \mathbb R) \) is invariant under \( \text{Aut} (H^1 (\Sigma, \mathbb R))\), hence since its action is irreducible, we have either  
    \[B_P^0 \cap H^1 (\Sigma, \mathbb R)= \{0\} \text{ or }  H^1(\Sigma, \mathbb R)\subset B_P .\]

In the case \(B_P^0 \cap H^1 (\Sigma, \mathbb R)= \{0\}\), \(B_P^0\) is a graph of \( \Im : H^1 (\Sigma, \mathbb C) \rightarrow H^1 (\Sigma, \mathbb R)\), over the image \( \Im B_P^0\). This latter being \( \text{Aut} (H^1 (\Sigma, \mathbb R))\)-invariant, irreducibility shows that it is either \(\{0\}\) (in which case \(B_P^0= \{0\}\)) or the whole \(H^1 (\Sigma, \mathbb R) \). In this latter case, one finds that \( B_P^0\) is of the form 
\[ B_P^0 = \{ Mu + i u \ |\ u \in H^1 (\Sigma, \mathbb R) \}\]
with \(M \in \text{End} (H^1 (\Sigma, \mathbb R))\). Invariance by \(\text{Aut} (H^1(\Sigma, \mathbb R) )\) shows that \(M\) commutes with every element of \( \text{Aut} (H^1(\Sigma, \mathbb R) )\). This implies that any element of \(H^1(\Sigma, \mathbb R)\) is an eigenvector of \(M\) (indeed, for every \(u \in H^1(\Sigma, \mathbb R) \setminus \{0\}\), choose a symplectic basis \(a_1, b_1, \ldots\) with \(a_1= u\), and observe that \( L\) defined as \( La_1 = 2 a_1 \), \(Lb_1 = \frac{1}{2} b_1 \), \(La_k = a_k \), \(L b_k= b_k \) for \(k\geq 2\), is symplectic, hence \( \text{ker} (L- 2 I) = \mathbb R u\) is \(M\)-invariant). So \(M= \lambda I\) for some \(\lambda \in \mathbb R\) and we are in case (3) with \(\mu = 1\). 

In the case where \(H^1(\Sigma, \mathbb R)\subset B_P \), one has \( B_P^0 = H^1 (\Sigma , \mathbb R) + i E\), for some real subspace of \(H^1( \Sigma, \mathbb R) \) invariant under \(\text{Aut} (H^1(\Sigma, \mathbb R) )\). Such a space is either \(\{0\}\) or \(H^1 (\Sigma, \mathbb R) \) and one deduces that \( B_P^0 \) is either \( H^1 (\Sigma, \mathbb R) \) (case (3) with \(\lambda=1\) and \(\mu=0\)) or \( H^1 (\Sigma, \mathbb C)\) (case (2))  as desired. 
\end{proof}

We will study the three cases occuring in the preceeding lemma separately. 

\underline{Case where \( B_P^0 = \{0\}\):} In this case \(B_P\) is a discrete subgroup of \( H^1 (\Sigma, \mathbb C) \) which is invariant by \( \text{Aut} (H^1 (\Sigma, \mathbb R) \) so \( B_P=\{0\}\) as well. We conclude that \( G_P = S_P \ltimes \{0\}\), and that \( S_P^\mathbb Z = S_P \cap G^\mathbb Z\) is a lattice in \(S_P\). In particular, there exists a finite index subgroup \( \Gamma \subset \text{Aut} (H^1 (\Sigma, \mathbb Z) ) \) in restriction to which the function \( L\mapsto P- LP \) takes integral values. This occurs iff the period \(P\) takes values in \( \mathbb Q +i \mathbb Q \), or equivalently the reduction of \(P\) modulo \(\mathbb Z+i \mathbb Z \) takes a finite number of values. 

\underline{Case where \( B^P_0 = H^1(\Sigma, \mathbb C)\):} In this case, one has \( G_P = G\), and the orbit of \(P\) is dense in \( H^1 ( \Sigma, \mathbb C).\)

\underline{Case where \( B_P ^0 = \{\lambda u + \mu u i\ |\ u\in H^1 (\Sigma, \mathbb R)\}\):} 
The group \((B_P^0 ) ^\mathbb Z = B_P^0 \cap H^1(\Sigma, \mathbb Z+i\mathbb Z)\) is a lattice, so the \(\mathbb R\)-vector space \( B_P^0\) is rational in \(H^1 (\Sigma, \mathbb C)\). This implies that \( \lambda,\mu\) can be chosen integers, by appropriately multiplying them by the same number. The lattice \( (G_P^0)^\mathbb Z = G_P^0 \cap G^\mathbb Z \)  in \( G_P^0\) can be expressed as an extension 
\[ 0\rightarrow B_P^0 \cap H^1 (\Sigma, \mathbb Z) \rightarrow (G_P^0)^\mathbb Z \rightarrow \Gamma \rightarrow 1\]
where \( \Gamma\) is a finite index subgroup of \( \text{Aut} (H^1 (\Sigma, \mathbb Z))\). Then each \( L \in \Gamma\) has a lift \( (L, t) \in (G_P^0)^\mathbb Z \subset G_P^0 \), and since \( G^0_P = S_P \ltimes B_P^0 \), there exists an element \(b \in B_P^0\) such that \(t= P-LP +b\). Denoting \( \varphi = \mu \Re - \nu \Im  : \mathbb C \rightarrow \mathbb R \), one has \(\varphi b=0 \), so \[\varphi P- L \varphi P= \varphi (P- LP) = - \varphi (t)\in H^1(\Sigma, \mathbb Z),\]
and since this happens for every element \(L\) of the lattice \(\Gamma\), one deduces that \(\varphi P \in H^1(\Sigma, \mathbb Q)\). In this case, \(\varphi \) is a discrete factor of \(P\) and the orbit of \(P\) is made of all periods having the same image in that discrete factor, so we are done. 

\vspace{0.2cm}

\textit{Case 3: up to \(\mathbb R\)-linear automorphisms of \(\mathbb C\), \(\Lambda_\Pi\) is isomorphic to \(\mathbb R +i \mathbb Z\).} Lemma \ref{l: reduction} reduces the question to the understanding of orbits closures of the action of \(\text{Aut} (H^1 (\Sigma,\mathbb Z)) \) on \(H^1 (\Sigma,  \mathbb R /\mathbb Z)\). These latter are dense appart orbits of rational points which are discrete. Details can be found in Section 11 of  \cite{CD}. Hence Proposition \ref{p: orbit closures} hold in this case.

\vspace{0.2cm}

\textit{Case 4: \(\Lambda_\Pi =\mathbb C\).} In this case, from Lemma \ref{l: reduction} the orbit of \(p\) is dense so Proposition \ref{p: orbit closures} holds. 
\end{proof}

\begin{proof}[Proof of Theorem \ref{t:leaf closures}:] By the connectedness of the fibers of the period map in theorems \ref{t:Rgn} and \ref{t:Cgn}, we can apply the transfer priciple for each  leaf $L$ of  $\mathcal{F}^\alpha_{g,n}$ such that $\Pi_{L}$ and $\Lambda_{L}$ satisfy the hypothesis of Theorem \ref{t:leaf closures} . Its leaf closure $\overline{L}\subset\Omega^{\alpha}\mathcal{M}_{g,n}$ corresponds to the points whose period homomorphism lies in the  mapping class group orbit closure of  $p\in H^1_\alpha(\Sigma\setminus \PP, \mathbb{C})$. 
 Proposition \ref{p: orbit closures} tells us that the closure is the set described in the statement of Theorem \ref{t:leaf closures}.

 The ergodicity result is a consequence of Moore's Theorem \cite{Moore}.

The leaves having $\Lambda_L$ discrete are closed by the analysis given above. 
The non-algebraicity of $L$ when \(\Lambda\) is a lattice comes from the fact that its closure in $\Omega^\alpha\overline{M}_{g,n}$ cuts the boundary in a non-algebraic way. Indeed, a fiber corresponding to a lift of $L$ to $\Omega\mnc(\Sigma,\PP)$ cuts all the boundary components corresponding to $\PP$-modules with different volume values in the holomorphic part. There are an infinite number of such  boundary components (see \ref{ss: connecting R-bnd strata}), and they cannot be in the same orbit of action of the group $\text{Mod}_{\Sigma,\PP}$ on $\Omega\mnc(\Sigma,\PP)$. Therefore they describe an infinite set of components of intersection of the closure of $L$ in $\Omega^\alpha\overline{M}_{g,n}$  with some boundary component 
corresponding to $\PP$-modules in $\Omega\mnc(\Sigma,\PP)$.\end{proof}

Next remark that if $L$ is a leaf of $\mathcal{F}_{g,n}$ the period group $\Lambda_{L}=\Lambda_{\omega}$ and the peripheral period group $\Pi_{L}=\Pi_\omega$  do not depend on \( (C,Q,\omega)\) but merely of \(L\). 

The number of different discrete factors of an element $p\in H^1(\Sigma\setminus \PP,\mathbb{C})$ tells us what the closure $\Lambda$ of the image of $p$ in $\mathbb{C}$ is. If $p$ has no discrete factors then  $\Lambda=\mathbb{C}$, if it has at least two different discrete factors then $\Lambda$ is discrete and if it has a unique discrete factor, then $\Lambda$ a $\mathbb{R}$-linear copy of $\mathbb{R}+i\mathbb{Z}$ in $\mathbb{C}$. 

From Theorem \ref{t:leaf closures} and the previous trichotomy on the number of different factors we easily deduce the following: 

\begin{corollary}\label{c: leaf closures}
Let $L$ be a leaf of $\mathcal{F}^{\alpha}_{g,n}$. Up to the action of $\text{GL}_2(\mathbb{R})$ we may assume that $\mathbb{Z}\subseteq \Pi_{L}$ and if it is discrete, $\Pi_{L}=\mathbb{Z}$. Suppose that, if \(\Pi_{L}\) is real, $\Lambda_{L}\not\subset\Pi_{L}\otimes \mathbb{Q}$, and denote $\Lambda=\overline{\Lambda}_{L}\subset\mathbb{C}$  the topological closure of $\Lambda_{L}$. Then 
\begin{enumerate}
    \item If $\Lambda$ is discrete, then $\overline{L}=\{(C,Q,\omega)\in\Omega^{\alpha}\mathcal{M}_{g,n}: \Lambda_{\omega}=\Lambda\}=L$
    \item If $\Lambda=\mathbb{C}$ then $\overline{L}=\Omega^{\alpha}\mathcal{M}_{g,n}$
    \item If $\Lambda$ is neither discrete nor dense, then
    \begin{enumerate}
        \item If $\Im\Lambda =0$ then $\overline{\mathcal {L}}=\{(C,Q,\omega)\in\Omega^{\alpha}\mathcal{M}_{g,n}: \Im \Lambda_{\omega}=0\}$
        \item If $\Im \Lambda\neq 0$ is discrete, and $\Pi_{L}$ is not discrete, then $$\overline{L}=\{(C,Q,\omega)\in\Omega^{\alpha}\mathcal{M}_{g,n}: \Im(\Lambda_{\omega})=\Im(\Lambda)\}$$
        \item If $\Im \Lambda\neq 0$ and $\Pi_{L}$ are both discrete (therefore $\Pi_{L}=\mathbb{Z}$) and we denote $V_{L}$ the value of $V$ on an element of $L$, then $$\overline{L}=\{(C,Q,\omega)\in\Omega^{\alpha}\mathcal{M}_{g,n}: \Im(\Lambda_{\omega})=\Im(\Lambda)\text{ and } V(\omega)=V_{L}\}$$
       \item If $\Im\Lambda$ is dense, then for a unique \(\beta\in\mathbb R\), \((\Re+\beta\Im) (\Lambda)\) is finite  in $\mathbb{R}$ and $$\overline{L}=\{(C,Q,\omega)\in\Omega^{\alpha}\mathcal{M}_{g,n}: (\Re+\beta\Im)(\Lambda_{\omega})=(\Re+\beta\Im)(\Lambda)\}$$
    \end{enumerate}
\end{enumerate}
\end{corollary}

\begin{proof}[Proof of Corollary \ref{c: leaf closures}:] From Theorem \ref{t:leaf closures} each case corresponds to one of the possibilities of presence of discrete factors. Case 1) corresponds to the presence of two distinct discrete factors. Case 2) corresponds do the case of the absence of discrete factors. Case 3) corresponds to the presence of a unique discrete factor, with a distinction depending on wether its kernel contains the peripheral module (subcases (a) and (b) and (c)) or not (case (d)).  \end{proof}



\begin{thebibliography}{}

\bibitem{ACG} {\sc E. Arbarello, M. Cornalba, P. Griffiths.} Geometry of Algebraic Curves, Vol. II, Springer Verlag, (2011) 
\bibitem{AC}{\sc E. Arbarello, M. Cornalba.} Calculating cohomology groups of moduli spaces of curves
via algebraic geometry, Publications math\' ematiques de l'I.H.E.S, tome 88 (1998), p. 97-127

\bibitem{5A18}{\sc M. Bainbridge, D. Chen, Q. Gendron, S. Grushevsky, M.  M\" oller.} Compactification of strata of Abelian differentials,  Duke Math. J. 167 (2018), no. 12, pp. 2347--2416.
\bibitem{5A19} {\sc M. Bainbridge, D. Chen, Q. Gendron, S. Grushevsky, M.  M\" oller.} The moduli of multi-scale differentials, arXiv:1910.13492, pp. 1--122
\bibitem{boissy} {\sc C. Boissy.} Connected components of the strata of the moduli space of mermorphic differentials, Comment. Math. Helv. 90 (2015), 255--286
\bibitem{bps} {\sc G. Calsamiglia, B. Deroin, S. Francaviglia.} Branched projective structures with Fuchsian holonomy. Geometry and Topology 18 (2014) pp. 379---446.
\bibitem{CD} {\sc G. Calsamiglia, B. Deroin.} Isoperiodic meromorphic forms: two simple poles, arXiv:2109.01796, accepted for publication at Groups, Geometry and Dynamics, 2024
\bibitem{CDF2} {\sc G. Calsamiglia, B. Deroin, S. Francaviglia.}   A transfer principle: from periods to isoperiodic foliations. Geom. Funct. Anal. Vol 33 (2023), pp. 57--169. https://doi.org/10.1007/s00039-023-00627-w
\bibitem{CGPT}{\sc D. Chen, Q. Gendron, M. Prado, G.Tahar.} Isoresidual curves, Preprint  2024, arXiv:2412.16810 
\bibitem{CFG}{\sc S. Chenakkod, G. Faraco and S. Gupta.} Translation surfaces and periods of meromorphic differentials., Proc. London Math. Soc. (3) 2022, 124 478-557.
\bibitem{FM} {\sc B. Farb, D. Margalit.} A primer on mapping class groups, Princeton University Press, 2012
\bibitem{GT} {\sc Q. Gendron, G. Tahar.} Isoresidual fibration and resonance arrangements, Letters in Mathematical Physics 112(33), pp. 1--36 (2022)
\bibitem{G}  {\sc S. Ghazouani.} Mapping class group dynamics and the holonomy of branched affine structures, Math. Z. 289 (2018), no. 1-2, pp. 1--23
\bibitem{Harris} {\sc J. Harris, I. Morrison.} Moduli of curves, Graduate Texts in Mathematics 187, Springer, New York (1998)
\bibitem{Haupt} {\sc O. Haupt.} Ein Satz \"uber die Abelschen integrale I. Gattung. Math. Z. 6 (1920), no. 3-4, 219-237.
\bibitem{Kapovich} {\sc M. Kapovich.} Periods of abelian differentials and dynamics. "Dynamics: Topology and Numbers" (Proceedings of Kolyada Memorial Concerence), Contemporary Mathematics, AMS, vol. 744, 2020, pp. 297---315. 
\bibitem{LT} {\sc M. Lee, G. Tahar.} One-dimensional strata of residueless meromorphic differentials, International Mathematical Research Notices 2025(12) (2025) 

\bibitem{McMullen} {\sc C. McMullen.} Moduli spaces of isoperiodic forms on Riemann surfaces.    Duke Math. J. Vol. 163, Number 12 (2014), 2271-2323.
\bibitem{Moller_linear}{\sc M. M\"oller.} Linear manifolds in the moduli space of one-forms. Duke Math. J. 144 (2008), no. 3, pp. 447--487. 
\bibitem{Moore} {\sc C. C. Moore.} Ergodicity of flows on homogeneous spaces. Amer. J. Math. 88 (1966) pp. 154–178
\bibitem{Tahar1} {\sc G. Tahar.} Counting saddle connections in flat surfaces with poles of higher order, Geom Dedicata (2018) 196, pp. 145--186
\end{thebibliography}
\end{document}